\documentclass[a4paper,10pt]{amsart}

\usepackage[latin1]{inputenc}
\usepackage{amsmath, amsthm, amssymb, comment}
\usepackage[dvips]{graphicx}
\usepackage{overpic}
\usepackage{enumitem}
\usepackage[colorlinks=true, citecolor=red]{hyperref}
\usepackage{pinlabel}

\usepackage[margin=3cm]{geometry}

\usepackage{color}

\newtheorem{theorem}{Theorem}[section]
\newtheorem{corollary}[theorem]{Corollary}
\newtheorem{proposition}[theorem]{Proposition}
\newtheorem{definition}[theorem]{Definition}
\newtheorem{lemma}[theorem]{Lemma}
\newtheorem{claim}[theorem]{Claim}

\newtheorem*{theorem*}{Theorem}
\newtheorem*{proposition*}{Proposition}
\newtheorem*{definition*}{Definition}
\newtheorem*{lemma*}{Lemma}
\newtheorem*{claim*}{Claim}
\newtheorem*{corollary*}{Corollary}
\newtheorem*{convention*}{Convention}
\newtheorem{observation}[theorem]{Observation}

\theoremstyle{definition}
\newtheorem{convention}[theorem]{Convention}
\newtheorem{example}{Example}
\newtheorem{question}{Question}
\newtheorem*{notation}{Notation}

\theoremstyle{remark}
\newtheorem{rem}[theorem]{Remark}
\newtheorem*{rem*}{Remark}

\newtheorem{case}{Case}

\newcommand{\wt}[1]{\widetilde{#1}}

\newcommand\bR{\mathbb R}

\newcommand\bZ{\mathbb Z}
\newcommand\bN{\mathbb N}
\newcommand\bC{\mathbb C}

\newcommand{\R}{\mathbb R}

\DeclareMathOperator{\TL}{\mathbf{TL}}

\DeclareMathOperator{\id}{Id}

\newcommand{\putinbox}[1]{\noindent\fbox{\parbox{.98\textwidth}{#1}}\medskip}

\newcommand\orb{ \mathcal O }

\newcommand{\cC}{\mathcal{C}}

\newcommand{\cF}{\mathcal{F}}
\newcommand{\cL}{\mathcal{L}}

\newcommand{\cP}{\mathcal{P}}

\newcommand{\cE}{\mathcal{E}}

\newcommand{\Pbound}{\partial P}

\newcommand{\fix}{\cP}
\newcommand{\fixnc}{\cP^{\mathrm{NC}}}

\newcommand{\acts}{\curvearrowright}

\DeclareMathOperator{\idealpts}{\partial\mathcal{F}}

\newcounter{notes}
\title{Orbit equivalences of pseudo-Anosov flows}

\author[Thomas Barthelm\'e]{Thomas Barthelm\'e}
\address{Queen's University, Kingston, Ontario}
\email{thomas.barthelme@queensu.ca}
\urladdr{sites.google.com/site/thomasbarthelme}

\author[Steven Frankel]{Steven Frankel}
 \address{Washington University, St.~Louis, Mi}
 \email{steven.frankel@wustl.edu}

\author[Kathryn Mann]{Kathryn Mann}
 \address{Cornell University, Ithaca, NY}
 \email{k.mann@cornell.edu}
\urladdr{https://e.math.cornell.edu/people/mann}

\begin{document}

 \begin{abstract}
 We prove a classification theorem for transitive Anosov and pseudo-Anosov flows on closed 3-manifolds, up to orbit equivalence.  
 In many cases, flows on a 3-manifold $M$ are completely determined by the set of free homotopy classes of their (unoriented) periodic orbits.  The exceptional cases are flows with a special structure in their orbit space called a ``tree of scalloped regions"; in these cases the set of free homotopy classes of unoriented periodic orbits together with the additional data of a choice of sign for each $\pi_1(M)$-orbit of tree gives a complete invariant of orbit equivalence classes of flows.  

 The framework for the proof is a more general result about {\em Anosov-like actions} of abstract groups on bifoliated planes, showing that the homeomorphism type of the bifoliation and the conjugacy class of the action can be recovered from knowledge of which elements of the group act with fixed points. 
  
As a consequence, we show that Anosov flows are determined up to orbit equivalence by the action on the ideal boundary of their orbit spaces, and more generally that transitive Anosov-like actions on bifoliated planes are determined up to conjugacy by their actions on the plane's ideal boundary: any conjugacy between two such actions on their ideal circles can be extended uniquely to a conjugacy on the interior of the plane.  
 \end{abstract}

 \maketitle
 
 \section{Introduction}
 In \cite[II.3]{Smale}, Smale frames the ``very important" problem of describing all transitive Anosov flows on a given compact manifold $M$. 
Flows on a fixed manifold may be considered up to either orbit equivalence or isotopy equivalence: an \emph{orbit equivalence} between two flows $\varphi^t$ and $\psi^t$ is a homeomorphism of $M$ that takes the orbits of $\phi^t$ to orbits of $\psi^t$; an \emph{isotopy equivalence} is an orbit equivalence that is isotopic to the identity.  Classification up to such equivalences amounts to understanding the topology of the possible foliations given by orbits of such flows\footnote{Note that our convention does not require the homeomorphism to preserve  orientation of orbits. Also, since we are concerned with a purely topological question, we will drop the time parameter $t$ and denote flows by $\varphi$ and $\psi$.  
 }. 
 
The problem is already challenging and important on manifolds of dimension 3.   Work of Plante, Ghys and Barbot gives finiteness results on certain geometric (solvable or Seifert fibered) 3-manifolds \cite{Plante81, Ghys84, Barbot96}.  By contrast, much more recent work of Bonatti--Bin--Yu \cite{BBY}, Bowden and the second author \cite{BM}, and Clay--Pinsky \cite{CP} shows that, for any given $n$, one can find examples of closed 3-manifolds (either hyperbolic or with non-trivial JSJ decomposition) that admit at least $n$ inequivalent transitive Anosov flows.  A main difficulty is certifying that the flows in consideration are indeed inequivalent.  More generally, 
the classification problem asks for invariants to distinguish inequivalent flows.  This is what we do here.  
 
 Given a flow $\varphi$, let $\cP(\varphi)$ denote the set of conjugacy classes of elements $\gamma \in \pi_1(M)$ such that either $\gamma$ or $\gamma^{-1}$ represents the homotopy class of a periodic orbit of $\varphi$, i.e. the set of unoriented free homotopy classes of closed loops in the foliation of $M$ by orbits. 
It is immediate from the definition that if $f\colon M \to M$ is an orbit equivalence between $\varphi$ and $\psi$, then $f_\ast P(\varphi) = P(\psi)$.  We show that $P(\varphi)$ is a complete invariant of transitive pseudo-Anosov flows  in many cases, specifically if the orbit space of the flow does not contain a very special structure called a ``tree of scalloped regions" (see Definition \ref{def_tree_scalloped}).  Additional discrete data, essentially a choice of sign for each orbit of tree of scalloped regions,  gives a complete invariant applicable to all transitive pseudo-Anosov flows (see Theorem \ref{thm_main_Anosov} below). 

 \begin{theorem} \label{thm_notree_Anosov}
Let $\varphi$ and $\psi$ be two pseudo-Anosov flows on a closed $3$-manifold $M$.  Assume at least one flow is transitive, and at least one flow does not have a tree of scalloped regions in its orbit space. Then, 
\begin{enumerate} 
\item $\varphi$ and $\psi$  are isotopically equivalent if and only if $\cP(\varphi) = \cP(\psi)$.
\item More generally, $\varphi$ and $\psi$ are orbit equivalent if and only if there is an automorphism $\Phi$ of $\pi_1(M)$ such that $\Phi_\ast \cP(\varphi) = \cP(\psi)$. In this case, the orbit equivalence can be taken to induce $\Phi$ on $\pi_1(M)$. 
\end{enumerate}
 \end{theorem}
Note that, perhaps surprisingly, one does not need to know any information about the {\em number} of orbits in a free homotopy class, or their length; $\cP(\varphi)$ as a set suffices.  
 
 While we postpone the definition of tree of scalloped regions to Section \ref{sec_scalloped_and_trees}, we know exactly which flows admit such regions and they are relatively rare (see Proposition \ref{prop:tree_iff_periodic} for the precise statement). In particular, we have
\begin{proposition}\label{prop_conditions_no_tree}
Let $M$ be a $3$-manifold $M$ with pseudo-Anosov flow $\varphi$. Under {\em any} of the following conditions, $\varphi$ does not admit a tree of scalloped regions; thus $P(\varphi)$ is a complete invariant of isotopy equivalences of such flows.  
\begin{enumerate}
\item $M$ is hyperbolic;
\item More generally, $M$ only has atoroidal pieces in its JSJ decomposition;
\item There is no essential torus in $M$ transverse to the flow $\varphi$;
\item For any Seifert piece of the JSJ decomposition of $M$, a regular Seifert fiber
 is \emph{not} freely homotopic to any periodic orbit of $\varphi$.
\end{enumerate}
\end{proposition}

The general statement, applicable to all transitive pseudo-Anosov flow is as follows. 

 \begin{theorem}[Classification of transitive pseudo-Anosov flows] \label{thm_main_Anosov}
Let $\varphi$ and $\psi$ be two pseudo-Anosov flows on a closed $3$-manifold $M$ with at least one flow transitive.  
If $\Phi_\ast \cP(\varphi) = \cP(\psi)$ for some automorphism $\Phi$ of $\pi_1(M)$, and the {\em signs} of each $\pi_1(M)$-orbit of tree of scalloped regions agree, then there is an orbit equivalence between $\varphi$ and $\psi$ that induces $\Phi$ on $\pi_1(M)$.
\end{theorem} 

As will be shown in Section \ref{sec:nonskewed_proof}, when $\Phi_\ast \cP(\varphi) = \cP(\psi)$ there is a natural bijection between the orbits of trees of scalloped regions in the orbit space of the flow, which gives us the definition of ``sign".  As in our discussion above, the converse to Theorem \ref{thm_main_Anosov} is also true and immediate from the definitions.  Theorem \ref{thm_notree_Anosov} is simply the special case of this statement when there are no trees of scalloped regions present.

 \subsubsection*{Sharpness of Results}
Notice that, in Theorems  \ref{thm_notree_Anosov} and \ref{thm_main_Anosov}, we only assume one of the flows to be transitive, but as a consequence, the other flow must also be transitive.
 In Section \ref{sec:transitive} we show transitivity is a necessary hypothesis by building non-transitive, inequivalent flows with the same free homotopy classes represented by periodic orbits.
In an upcoming article with Sergio Fenley \cite{BFM_in_preparation}, we show the necessity of taking the sign data for the trees of scalloped regions into account, building examples of transitive, non-orbit equivalent Anosov flows with the same set of free homotopy classes (which of course do have trees of scalloped regions in their orbit spaces).  This construction is done by gluing {\em periodic Seifert pieces} as studied in \cite{BF_totally_per}.

 \subsubsection*{Earlier work}
Two special cases of Theorem \ref{thm_main_Anosov} were previously known: In \cite{BMB}, the first and last authors proved the theorem for $\bR$-covered Anosov flows -- we use this result in our proof. In \cite[Proposition 7.2]{FenPot}, Fenley and Potrie proved a version of this theorem for the special case of pseudo-Anosov flows in hyperbolic manifolds with no freely homotopic periodic orbits (this in particular implies there are no trees of scalloped regions), under the more general condition that $\cP(\varphi) \subset \cP(\psi)$, this is entirely independent of our work here. 
In a different vein, Barbot and Fenley \cite{BF_totally_per} gave a classification result for pseudo-Anosov flows on graph manifolds such that {\em all} Seifert pieces have fibers represented by periodic obits.  
Unsurprisingly, their classification relies on much more topological data than the free homotopy classes of periodic orbits.  

 \subsubsection*{Proof strategy}
To prove Theorem \ref{thm_main_Anosov}, we consider the action of $\pi_1(M)$ on the {\em orbit space}, a topological plane obtained by lifting the flow to $\widetilde{M}$ and identifying each orbit to a point.  The action of $\pi_1(M)$ by deck transformations induces an action on this plane that preserves a pair of transverse, $1$-dimensional, singular foliations $\cF^\pm$ coming from the weak-stable and unstable foliations of the flow.  The main technical result underlying Theorem \ref{thm_main_Anosov} is a rigidity theorem for groups acting on bifoliated planes (Theorem~\ref{thm_main_general}), stating that under suitable dynamical assumptions, such actions (as well as the topology of the bifoliation) are determined up to conjugacy by the algebraic data of the {\em set of group elements acting with fixed points}.  

\subsection*{Group actions on bifoliated planes} 
A {\em bifoliated plane} is a topological plane $P$ equipped with a transverse pair of singular foliations $\cF^+$ and $\cF^-$.  A group action $G \acts P$ is called {\em Anosov-like} if it preserves the foliations and satisfies a list of dynamical requirements motivated by the dynamics of a 3-manifold fundamental group acting on the orbit space of a pseudo-Anosov flow.  These include such requirements as topological hyperbolicity of fixed points; a complete list is given in Definition \ref{def_anosov_like}.  We prove the following.  

\begin{theorem}\label{thm_main_general}
Let $\rho_i$, $i=1,2$ be two Anosov-like actions of a group $G$ on nontrivial bifoliated planes $(P_i,\cF_i^+, \cF_i^-)$. 

If at least one plane has no {\em tree of scalloped regions}, and the same set of elements of $G$ acts with fixed points under $\rho_1$ as under $\rho_2$, then the two actions are conjugate by a homeomorphism $h \colon P_1 \to P_2$ sending the pair $\cF_1^{\pm}$ to $\cF_2^{\pm}$.

In general, if the same set of elements of $G$ acts with fixed points under $\rho_1$ as under $\rho_2$, then there is a natural bijection between $G$-orbits of trees of scalloped regions, and the two actions are conjugate by a homeomorphism as above if and only if the signs of their trees of scalloped regions agree.
\end{theorem}

Here, a bifoliated plane is {\em nontrivial} if it is not homeomorphic to $\R \times \R$ with the two coordinate foliations. As we discuss in Section \ref{sec:bifoliated}, an Anosov-like action on a {\em trivial} bifoliated plane is necessarily an action by affine transformations. Affine actions are not generally distinguishable up to conjugacy by their fixed spectra; however, the only examples arising from Anosov flows on compact manifolds come from suspension flows of hyperbolic toral automorphisms, which do satisfy spectral rigidity.

Note that we do not assume a priori that the foliations on $P_i$ are homeomorphic -- instead this follows from the theorem. Note also that the homeomorphism $h$ sends the foliations of $P_1$ to those of $P_2$ but may interchange them i.e.~we may have $h (\cF_1^+)=\cF_2^-$.

Because it passes through Theorem \ref{thm_main_general}, the proof of Theorem \ref{thm_main_Anosov} does not make use of the topology of the manifold $M$, rather, only on the dynamics of group actions on the planes.  This motivates the general question, in the spirit of the convergence group theorem: {\em 
Can one detect, purely from the dynamics of a group $G$ acting (faithfully) on a bifoliated plane, that $G$ is a 3-manifold group and the action comes from a pseudo-Anosov flow?}  

In Theorem \ref{thm:pA_is_anosov_like} we show that Anosov-like is a necessary condition.  In the special cases when the bifoliated plane is trivial or skewed, it is {\em not} sufficient (see Remarks \ref{rem_affine_counterexample} and \ref{rem:skew_counterexample} below).   However, we do not know whether this holds in general.  See Section \ref{sec:further_q} for discussion.

 \subsection*{Rigidity of boundary actions} 
 Theorem \ref{thm_main_general} is similar in spirit to Theorem 2.6 of \cite{BMB}, which gives a general theorem about conjugacy of {\em hyperbolic-like}  actions on the line, a class of actions inspired by the actions of skew-Anosov flows on the ideal boundary of the orbit space.   The work in \cite{BMB} was motivated by the fact that skew-Anosov flows are easily seen to be determined up to orbit equivalence from their boundary actions, thus studying the dynamics on the boundary is a simple substitute for studying the action on the orbit space.   For arbitrary pseudo-Anosov flows, the situation is significantly more complicated.  
 
A compactification of the orbit space of a pseudo-Anosov flow by an ideal circle was carried out by Fenley in \cite{Fen_ideal_boundaries}; this is natural in the sense that the action of $\pi_1(M)$ on $P$ extends to an action by homeomorphisms on the ideal circle.  In addition to the classification problem, one of the motivations for the present work was to understand whether this representation $\pi_1(M) \to \mathrm{Homeo}(S^1)$ determines the orbit equivalence class of the flow in this general setting.  As a consequence of our work, we answer this in the affirmative:
 
 \begin{theorem} \label{cor_main_Anosov}
  Let $\varphi$ and $\psi$ be two transitive, pseudo-Anosov flows on a closed $3$-manifold $M$. If the actions of $\pi_1(M)$ on the ideal circles of the orbit spaces of $\varphi$ and $\psi$ are conjugate, then $\varphi$ and $\psi$ are orbit equivalent.
 \end{theorem}
 
 In fact, we prove more generally that a conjugacy between the ideal circle actions of two Anosov-like actions extends to a conjugacy between their bifoliated planes (Theorem \ref{thm:ideal_circle}).  
 Notice that no hypotheses on trees of scalloped regions are required in this statement, the data of conjugate boundary actions suffices.

\subsection*{Outline}

Section \ref{sec:bifoliated} contains background material on bifoliated planes and Anosov-like actions, and a reduction of Theorem \ref{thm_main_Anosov} to Theorem \ref{thm_main_general}.   

In Section \ref{sec:ideal_circle} we introduce the {\em ideal circle} of a bifoliated plane, following Fenley and Frankel.  Homeomorphisms of the bifoliated plane extend naturally to homeomorphisms of the ideal circle, and we prove some general results about the dynamics of the actions of elements of Anosov-like groups that will be used later in our work.  These results are new, but should feel familiar to experts well acquainted with the dynamics on the orbit space in the Anosov flow case.  

Section \ref{sec:linking} contains the groundwork for the proof of Theorem \ref{thm_main_general}, showing that some topological properties (e.g. intersections or ``linking" of leaves) of a bifoliated plane can be detected from the spectral data of the set of elements of an Anosov-like action on the plane.   

Section \ref{sec:main_proof} is the technical heart of the paper.  In essence, we show how the ``linking'' of points can be used to reconstruct the bifoliated plane together with the Anosov-like action, which is how we obtain the proof of Theorem \ref{thm_main_general}.

In section \ref{sec:corollary_for_action_ideal_circle}, we prove a more general version of Theorem \ref{cor_main_Anosov} (Theorem \ref{thm:ideal_circle}), that shows, in the context of Anosov-like group actions, that the induced action on the ideal circle determines the action on the bifoliated plane.  

Finally, section \ref{sec:transitive} shows that the transitivity hypothesis in Theorem \ref{thm_main_Anosov} is necessary, building counterexamples in the nontransitive case, and Section \ref{sec:further_q} gives some discussion and questions for further study.

\subsection*{Acknowledgements}
The authors thank Sergio Fenley for detailed comments on an earlier version of the manuscript, which led to numerous improvements.  
TB was partially supported by the NSERC (Funding reference number RGPIN-2017-04592).
SF was partially supported by NSF CAREER grant DMS-2045323. KM was partially supported by NSF CAREER grant DMS-1933598 and a Sloan fellowship, and thanks the CNRS and Inst.~Math.~Jussieu for support and hospitality.

\section{Preliminaries} \label{sec:bifoliated}

A pseudo-Anosov flow $\varphi$ on a closed $3$-manifold $M$ lifts to a flow $\widetilde{\varphi}$ on the universal cover $\widetilde{M}$. By results of Verjovsky and Palmeira, $\widetilde{M}$ is always homeomorphic to $\bR^3$, and by \cite{Bar_caracterisation,Fen_Anosov_flow_3_manifolds} for Anosov flows and \cite{FenMosher} for the pseudo-Anosov case, 
 the quotient of $\widetilde{M}$ by the orbits of $\widetilde{\varphi}$, is homeomorphic to a plane.  We call this plane the {\em orbit space}, denoted by $P_\varphi$.  The deck group action $\pi(M) \acts \widetilde{M}$ takes orbits to orbits, so induces an action $\pi_1(M) \acts P_\varphi$.  
 
 The pseudo-Anosov structure of $\varphi$ gives $P_\varphi$ the extra structure of a transverse pair of $1$-dimensional singular foliations $\cF^+$ and $\cF^-$ which are preserved by the action of $\pi_1(M)$, and constrains the dynamics of this action: For instance an element $g \in \pi_1(M)$ fixes a point $p \in P$ if and only if it represents a (positive or negative) power of the free homotopy class of a closed orbit,
in which case the action is topologically contracting along the leaf of $\cF^-$ and expanding on $\cF^+$ (or vice versa).  
 
The proof of our main theorem takes place entirely in the orbit space $P_\varphi$.  Thus, we begin our work by distilling the relevant properties of the action of $\pi_1(M)$ on the orbit space, making an abstract definition of  ``Anosov-like group actions on bifoliated planes", and then recast the main theorem using the correspondence between closed orbits of a flow and fixed points of group elements on the orbit space. 

\subsection{Bifoliated planes and Anosov-like actions}

\begin{definition} A {\em bifoliated plane} $(P, \cF^+, \cF^-)$ is a topological plane $P$ with a transverse pair of singular foliations $\cF^+, \cF^-$ whose leaves are properly embedded $n$-prongs. 
A bifoliated plane $(P, \cF^+, \cF^-)$  is called \emph{trivial} if there is a homeomorphism $P \to \R^2$ that takes $\cF^+$ and $\cF^-$ to the foliations by horizontal and vertical lines.  
\end{definition}

\begin{notation}
For a point $x \in P$, we let $\cF^+(x)$ and $\cF^-(x)$ denote the leaves of $\cF^+$ and $\cF^-$ through $x$, such a leaf is either homeomorphic to $\R$, or to a $n$-prong for some $n\geq 3$.
A \emph{half-leaf} of $\cF^{\pm}(x)$ a closed subset of $\cF^{\pm}(x)$, homeomorphic to $[0,+\infty)$, and having $x$ as its (unique) boundary point.  A {\em regular point} is one where the foliation is locally trivial, other points (i.e. the centers of prongs) are {\em singular}.  
\end{notation} 

\begin{definition} 
The {\em leaf space} of a singular foliation $\cF$ of a plane $P$ is the quotient of $P$ obtained by collapsing each leaf of $\cF$ to a point, equipped with the quotient topology.  
\end{definition} 

Before stating the main definition, we recall that a homeomorphism $g$ of a topological space $X$ is {\em topologically contracting} if there is a point $x \in X$ such that $g^n(y) \to x$ as $n \to \infty$ for all $y \in X$, {\em and topologically expanding} if $g^{-1}$ is topologically contracting.  In particular, $x \in X$ is a unique fixed point for the action of $g$. 

\begin{definition}[Anosov-like action] \label{def_anosov_like}
An action of a group $G$ on a bifoliated plane, preserving each foliation, is called \emph{Anosov-like} if it satisfies the following: 
	\begin{enumerate}[label = (A\arabic*)]
	\item\label{Anosov_like_A1} 
	If a nontrivial element $g \in G$ fixes a leaf $l \in \cF^\pm$, then it has a fixed point $x \in l$, and is topologically expanding on one leaf through $x$ and topologically contracting on the other. 
		\item\label{Anosov_like_topologically_transitive} The action is topologically transitive, i.e. has a dense orbit.
		\item\label{Anosov_like_dense_fixed_points} The set of points that are fixed by some nontrivial element of $G$ is dense in $P$.
		\item \label{Anosov_like_prongs_are_fixed} Each singular point is fixed by some nontrivial element of $G$.
		\item \label{Anosov_like_periodic_non-separated} If two leaves $l_1,l_2$ in either $\cF^+$ or $\cF^-$ are non-separated in the corresponding leaf space, then some nontrivial element $g\in G$ fixes both.
\item \label{Anosov_like_totallyideal} There are no \emph{totally ideal quadrilaterals} in $P$ (see Definition \ref{def_totally_ideal_quad}).
	\end{enumerate}
\end{definition}

One could perhaps equally well call these 	{\em transitive Anosov-like actions}, emphasizing Axiom \ref{Anosov_like_topologically_transitive} and the relationship with transitive Anosov flows; however, transitivity will be a standing assumption in our work so we refer to these simply as Anosov-like actions. 

\begin{rem} \label{rem:dense_fixed}
In the case of an Anosov-like action coming from a pseudo-Anosov flow on a closed $3$-manifold, Axiom \ref{Anosov_like_dense_fixed_points} is equivalent to asking for the flow to be transitive, hence, thanks to the Anosov closing lemma, it is also equivalent to asking for the action to be topologically transitive, i.e. Axiom \ref{Anosov_like_topologically_transitive}. In general, it does not seem trivial to show that an Anosov-like action satisfying \ref{Anosov_like_dense_fixed_points} would automatically be topologically transitive, hence the additional assumption.   Several of our early lemmas do not require topological transitivity of the action as a hypothesis, but we have included it in our axioms to streamline the overall proof.    We comment on other possibilities of streamlining the axioms later in Section \ref{sec:digression_on_axioms}.  

 Axiom \ref{Anosov_like_totallyideal} will be used only once in this work, albeit crucially, in Proposition \ref{prop_unique_limit_for_NC}.  A {\em totally ideal quadrilateral} is a trivially foliated region $Q$ in $P$ that is bounded by exactly four complete leaves.
These are called $(4,0)$-ideal quadrilaterals by Fenley in \cite{Fen_qgpA16}, who showed the bifoliated plane of a pseudo-Anosov flow contains no such quadrilaterals. 
 \end{rem}

The next theorem justifies our choice of axioms, showing that transitive pseudo-Anosov flows give Anosov-like actions.

 \begin{theorem} \label{thm:pA_is_anosov_like}
Let $\varphi$ be a transitive pseudo-Anosov flow on a closed $3$-manifold $M$.   Let $P_\varphi$ be its orbit space; with stable and unstable foliations $\cF^+$ and $\cF^-$, respectively.  Then the action of $\pi_1(M)$ on $(P_\varphi,\cF^+, \cF^-)$ is a  transitive Anosov-like action on a bifoliated plane.
	Moreover,  $(P_\varphi,\cF^+, \cF^-)$ is trivially foliated if and only if $\varphi$ is a suspension of an Anosov diffeomorphism of the torus.
\end{theorem}

\begin{proof}
We assume the reader has basic familiarity with the definitions of Anosov and pseudo-Anosov flow.  See for example \cite[Definitions 4.51 and 6.41]{Calegari_book}.  As mentioned above, the orbit space of a pseudo-Anosov flow is a topological plane ( \cite{Bar_caracterisation,Fen_Anosov_flow_3_manifolds,FenMosher}) equipped with two singular foliations given by the projections of the stable and unstable foliations.

The fact that axioms \ref{Anosov_like_A1}-\ref{Anosov_like_periodic_non-separated} hold is fairly standard:
Since fixed points in $P$ of elements of $\pi_1(M)$ correspond to positive or negative powers of closed orbits, Axiom \ref{Anosov_like_A1} follows from the contraction/expansion of weak stable/unstable leaves: Indeed, that the flow is hyperbolic implies that there is a unique closed orbit on every invariant leaf and that fixed points are topologically hyperbolic (expanding on one leaf and contracting on the other) or hyperbolic prongs, as in Axiom \ref{Anosov_like_A1}.  
Singular orbits of a pseudo-Anosov flow are closed orbits (by definition) so the corresponding singular points are fixed points in $P_\varphi$, giving \ref{Anosov_like_prongs_are_fixed}.   Axiom \ref{Anosov_like_periodic_non-separated} is Theorem D of \cite{Fenley_structure_branching}. As noted in Remark \ref{rem:dense_fixed}, Conditions \ref{Anosov_like_topologically_transitive} and  \ref{Anosov_like_dense_fixed_points} are both equivalent to topological transitivity in this case. 
Fenley \cite[Proposition 4.4]{Fen_qgpA16} proved that the orbit space of pseudo-Anosov flows do not contain totally ideal quadrilaterals, giving Axiom \ref{Anosov_like_totallyideal}.

Finally, the fact that the orbit space is trivially foliated if and only if $\varphi$ is a suspension of an Anosov diffeomorphism was proved in \cite[Th\'eor\`eme 2.7]{Bar_caracterisation}, the proof of Theorem \ref{thm_trivial_skewed_or_nowheredense} below gives a generalized version of this argument.
\end{proof}
 
The next few sections give some preliminary structure theory on Anosov-like actions on bifoliated planes.

\subsection{Affine actions and trivial bifoliated planes}

\begin{lemma}\label{lem_same_fixed_leaf_same_fixed_point}
 Let $(P,\cF^+, \cF^-)$ be a bifoliated plane with an Anosov-like action of a group $G$, and let $l$ be a leaf of either $\cF^+$ or $\cF^-$.  
Then the stabilizer $G_l$ of $l$ in $G$ acts on $l$ with a common fixed point, and the finite-index subgroup of $G_l$ that preserves each half-leaf through that fixed point is Abelian.  
\end{lemma}

\begin{proof}
We first show that any two elements fixing $l$ have the same fixed point.  
Suppose $g$ and $h\in G$ both fix $l$.  Axiom \ref{Anosov_like_A1} implies that $g$ and $h$ fix a single point each on $l$.  If $l$ is singular, then $g$ and $h$ necessarily fix the singular point, which is unique by \ref{Anosov_like_A1}.  
If $l$ is not singular, we may identify it with $\bR$. Then, up to replacing $g$ and $h$ by their squares, which does not affect their fixed points given that fixed points on leaves are unique, we can assume that they preserve orientation, and consider the group generated by $g$ and $h$ as a group of orientation-preserving homeomorphisms of $\bR$. By Axiom \ref{Anosov_like_A1}, each element of this group acts with exactly one fixed point. Solodov's Theorem (see \cite[Th. 2.236]{Navas_book} for a proof) implies that either the group acts with a global fixed point (in which case we are done), or is semi-conjugate to a group of Affine transformations. In this latter case, in the absence of a global fixed point, there is a nontrivial subgroup acting by topological translations, which would contradict axiom \ref{Anosov_like_A1}.  

Now consider the finite index subgroup of $G_l$ that preserves each half-leaf based at the common fixed point.  Each half leaf is homeomorphic to $\R$, and the restriction of the action to the half leaf is free, so by H\"older's theorem (see \cite{Navas_book}) the subgroup is Abelian. \end{proof}

The following theorem is also shown in \cite[Th\'eor\`eme 2.7]{Bar_caracterisation}; we include a short proof for completeness. 

\begin{proposition}[Trivial planes have affine actions] \label{prop:trivial_affine} 
Any Anosov-like action of a group on a trivially bifoliated plane is conjugate to an action by affine transformations on each factor, i.e. maps of the form $(x,y) \mapsto (ax+b, cy+d)$. 
\end{proposition}

\begin{proof} 
	That $G$ preserves the two coordinate foliations implies that each element acts by maps of the form $f(x, y) = (f_1(x), f_2(y))$.  It remains to show that the induced actions on the two coordinate leaf spaces are conjugate to actions by affine transformations.  
	
	Since the plane is trivially foliated, every leaf of $\cF^+$ intersects every leaf of $\cF^-$.  Thus, if some $f\in G$ fixes a point $p$, this fixed point is necessarily unique.  In particular, its first coordinate map $f_1$ has at most one fixed point.  We conclude that the action of $G$ on $\R$ (the leaf space of the vertical foliations) has the property that each element has at most one fixed point.  Furthermore, Axiom \ref{Anosov_like_dense_fixed_points} together with the fact that fixed points are unique implies that the action does not have a global fixed point.   Thus, by Solodov's Theorem, this induced action is semi-conjugate to an action by affine linear transformations, and \ref{Anosov_like_topologically_transitive} (topological transitivity) implies that this semi-conjugacy is in fact a conjugacy.  The same argument applies to the second coordinate maps, completing the proof.  
\end{proof} 

Note that topological hyperbolicity of the action on leaves (Axiom \ref{Anosov_like_A1}) implies that for an affine map $(x,y) \mapsto (ax+b, cy+d)$ in an Anosov-like action on a trivially foliated plane, we have $a = 0$ iff $c=0$. Otherwise, we have $|a|>1$ and $|c|<1$, or vice versa.  Axiom \ref{Anosov_like_dense_fixed_points} means that the action is not by translations.  Beyond this, there are many examples of topologically transitive such actions, including many which do not come from Anosov flows.  One such family of examples is the following: 

\begin{example} 
	Let $\lambda_1, \lambda_2 \in \bR$ be such that $1$, $\lambda_1$ and $\lambda_2$ are all algebraically independent.
	Let $\tau_1 \colon \bR^2 \to \bR^2$ be the translation $(x, y) \mapsto (x+1, y)$, and let $\tau_2$ be $(x, y) \mapsto (x, y+1)$.  Let $f_{\lambda_i}\colon \bR^2 \to \bR^2$ be the linear map $(x,y) \mapsto (\lambda_i x, \lambda_i^{-1} y)$.
	Then the action generated by $\tau_1, \tau_2 ,f_{\lambda_1}, f_{\lambda_2}$ on $\bR\times \bR$ is an Anosov-like affine action. However, it does not come from an Anosov flow: affine actions coming from Anosov flows must come from a suspension of an Anosov diffeomorphism, they are thus algebraic, and commensurable to the semi-direct product of $\bZ \times \bZ$ with $\bZ$ acting by an hyperbolic element of $\mathrm{SL}(2,\bZ)$. 
\end{example} 

\begin{rem}  \label{rem_affine_counterexample}
	Building on the example above, note also that, if $\lambda_3\in \bR$ is yet another algebraically independent number, then the group generated by $\tau_1, \tau_2 ,f_{\lambda_1}, f_{\lambda_3}$ is isomorphic to group generated by $\tau_1, \tau_2 ,f_{\lambda_1}, f_{\lambda_2}$.  One may see this by looking at the action on the $x$-coordinate.  Fixing $i = 2$ or $i=3$, each element can be written uniquely as $x \mapsto (\lambda_1)^m (\lambda_i)^n x + p(\lambda_1, \lambda_i)$, where $p(t,s)$ is an element of the Laurent polynomial ring $\bZ[t, t^{-1}, s, s^{-1}]$.  This gives an abstract isomorphism between each group and the semi-direct product of $\bZ \times \bZ$ with $\bZ[t, t^{-1}, s, s^{-1}]$, where $((m,n), p(t, s))$ in this semi-direct product corresponds to the map $x \mapsto (\lambda_1)^m (\lambda_i)^n x + p(\lambda_1, \lambda_i)$. 
	
An element acts with fixed points if and only if its linear part is nontrivial, i.e. if the term $(\lambda_1)^m (\lambda_i)^n \neq 1$.  Since $\lambda_1$ and $\lambda_i$ are algebraically independent, both groups have precisely the same elements acting with fixed points in $\bR\times \bR$.  However, the actions are not conjugate --- the subgroup generated by $x \mapsto x+1$ and $f_{\lambda_1}$  already acts minimally on the $x$-coordinate, so uniquely determines the action of the group up to conjugacy.  
This shows that Theorem \ref{thm_main_general} does not hold for Anosov-like affine actions on trivial planes. 
\end{rem}

\subsection{Perfect fits and lozenges}
 We introduce a few definitions to describe configurations of leaves in a bifoliated plane.  
 
 \begin{definition}\label{def_perfect_fit}
  Two leaves (or half leaves) $l^+\in \cF^+$ and $l^-\in \cF^-$ are said to make a \emph{perfect fit} if they have empty intersection, but  ``just miss'' each other, as illustrated in Figure~\ref{fig:perffits}. That is, there is an arc $\tau^+$ starting at a point of $l^+$ and transverse to $\cF^+$ and an arc $\tau^-$ starting at $l^-$ transverse to $\cF^-$ such that
\begin{enumerate}[label=(\roman*)]                                            
\item every leaf $k^+ \in \cF^+$ that intersects the interior of ${\tau}^+$ intersects $l^-$, and
\item every leaf $k^- \in \cF^-$ that intersects the interior of ${\tau}^-$ intersects $l^+$.
 \end{enumerate}
 \end{definition}

\begin{observation} \label{obs:fix_perf_fit}
If $G$ is an Anosov-like action and $g \in G$ preserves orientation of $P$ and fixes a half-leaf $l$, then $g$ fixes all leaves that make a perfect fit with $l$.   This can be seen as follows:  without loss of generality, say $l \in \cF^+$ and take $\tau^-$ to be a small segment of $\cF^{-}(x)$, where $x$ is the unique fixed point of $g$ on $l$.  Up to replacing $g$ with its inverse, we have then $g\tau \subset \tau$, and the conclusion follows readily from the definition.   
\end{observation} 
 
 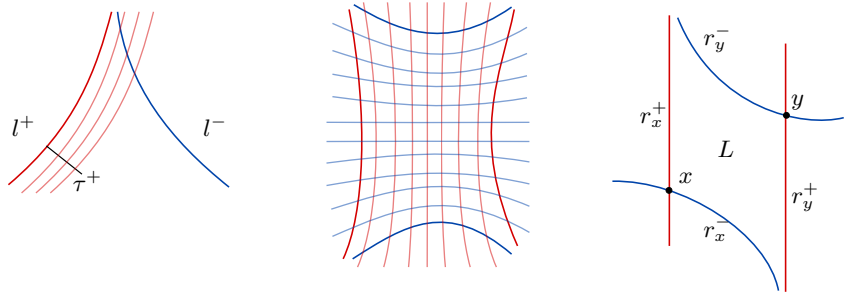
\begin{figure}
   \labellist 
  \small\hair 2pt
     \pinlabel $l^+$ at 10 105 
    \pinlabel $l^-$ at 130 105 
    \pinlabel $\tau^+$ at 50 70 
     \pinlabel $x$ at 425 72 
  \pinlabel $r_x^+$ at 405 112 
 \pinlabel $r_x^-$ at 445 40
   \pinlabel $r_y^+$ at 500 60 
 \pinlabel $r_y^-$ at 445 160
  \pinlabel $y$ at 495 120 
   \pinlabel $L$ at 450 90 
 \endlabellist
     \centerline{ \mbox{
\includegraphics[width=11cm]{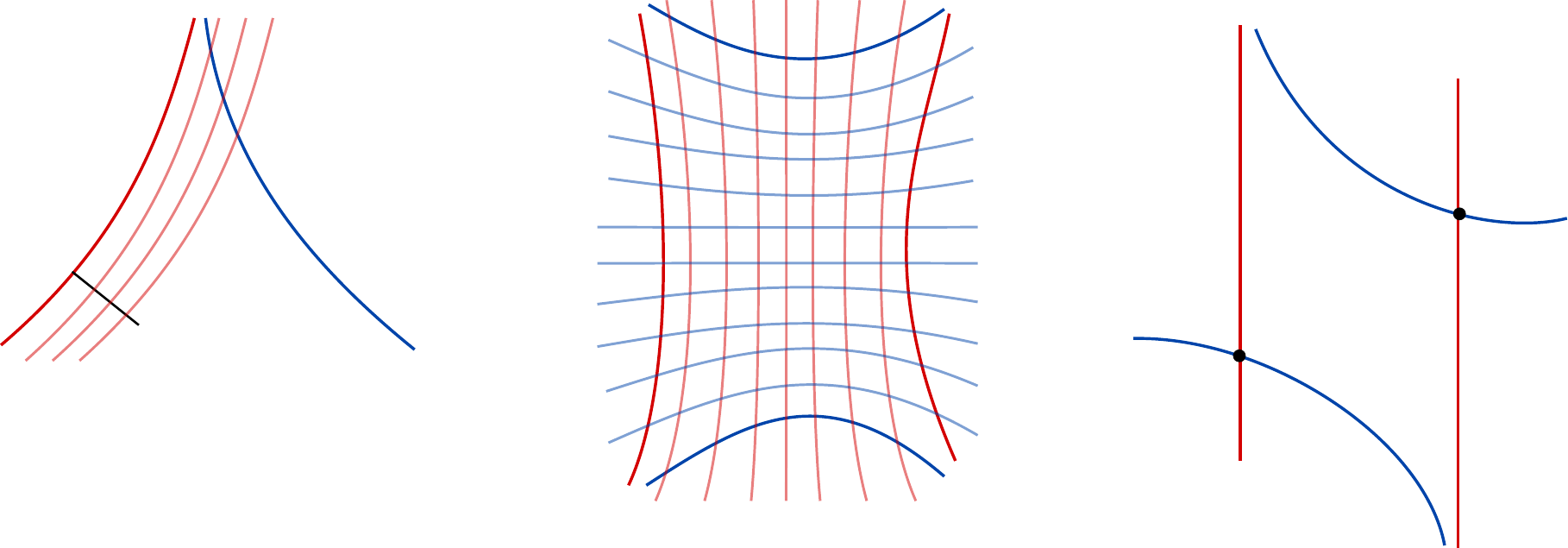}}}
\caption{A perfect fit, a totally ideal quadrilateral, and a lozenge with corners $x$ and $y$}
 \label{fig:perffits} 
\end{figure}

 \begin{definition}[Totally ideal quadrilateral]\label{def_totally_ideal_quad}
  A \emph{totally ideal quadrilateral} is an open set $Q$ bounded by four leaves $l_1, l_3 \in \cF^+$ and $l_2, l_4\in \cF^-$ such that 
  \begin{itemize}
   \item $l_i$ makes a perfect fit with $l_{i+1}$ for $i=1, \dots, 4$ (with the convention that $l_5=l_1$);
   \item Any leaf $l\in \cF^-$ (resp.~$\cF^+$) intersects $l_1$ if and only if it intersects $l_3$ (resp.~it intersects $l_2$ if and only if it intersects $l_4$).
  \end{itemize}
 \end{definition}

Recall that Axiom \ref{Anosov_like_totallyideal} says that such regions are not allowed in a bifoliated plane equipped with an Anosov-like action.

\begin{definition}[Lozenges and chains]
Let $x,y$ be two points in $(P,\cF^+,\cF^-)$ and let $r_x^\pm \subset \cF^\pm(x)$, and $r_y^\pm \subset \cF^\pm(y)$ be four half leaves, starting at $x$ and $y$ respectively, such that $r_x^\pm$ and $r_y^{\mp}$ make perfect fits, as shown in Figure \ref{fig:perffits} (right).

A \emph{lozenge $L$ with corners $x$ and $y$} is the open subset of $\orb$ ``bounded" by these half leaves; precisely, it is given by 
\begin{equation*}
 L := \lbrace p \in P \mid \cF^+(p) \cap r_x^- \neq \emptyset \text{ and } \cF^-(p) \cap r_x^- \neq \emptyset \rbrace.
\end{equation*}
The half-leaves $r_x^\pm$ and $r_y^{\pm}$ are called the \emph{sides} of the lozenge, the intersection of two sides is called a {\em corner}. A \emph{closed lozenge} is a lozenge together with its sides and corners, and a \emph{chain of lozenges} is a union of closed lozenges that satisfies the following connectedness property: for any two lozenges $L,L'$ in the chain, there exist lozenges $L_0, \dots,L_n$ in the chain such that $L=L_0$, $L'=L_n$, and, for all $i$, $L_i$ and $L_{i+1}$ share a corner (and may or may not share a side).
\end{definition}

\begin{definition}[Quadrants]
For a point $x \in P$, a {\em quadrant} of $x$ is a connected component of $P \smallsetminus \cF^{\pm}(x)$ that contains $x$ in its boundary.  
If $x$ is a regular point, it has 4 quadrants; an $n$-pronged singularity has $2n$ quadrants. 
\end{definition} 

Note that with this definition, the union of the closed quadrants at $x$ may not be the whole plane: if $x$ is a regular point on a singular leaf, parts of the plane are not contained in any quadrants of $x$.  See figure \ref{fig:quadrants}.  
However, we will typically use quadrants in the case where $x$ is a point fixed by some element of $G$, in which case the quadrants do partition the plane.

  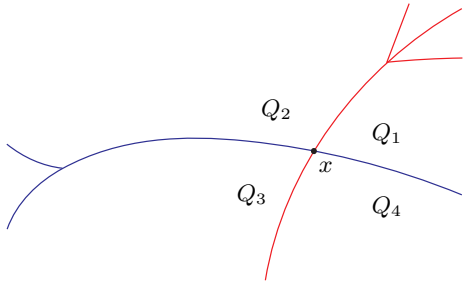
\begin{figure}[h]
   \labellist 
   \small\hair 2pt
      \pinlabel $x$ at 260 93 
     \pinlabel $Q_1$ at 310 120 
     \pinlabel $Q_2$ at 220 140 
    \pinlabel $Q_3$ at 200 70  
    \pinlabel $Q_4$ at 310 60 
  \endlabellist
     \centerline{ \mbox{
\includegraphics[width=6cm]{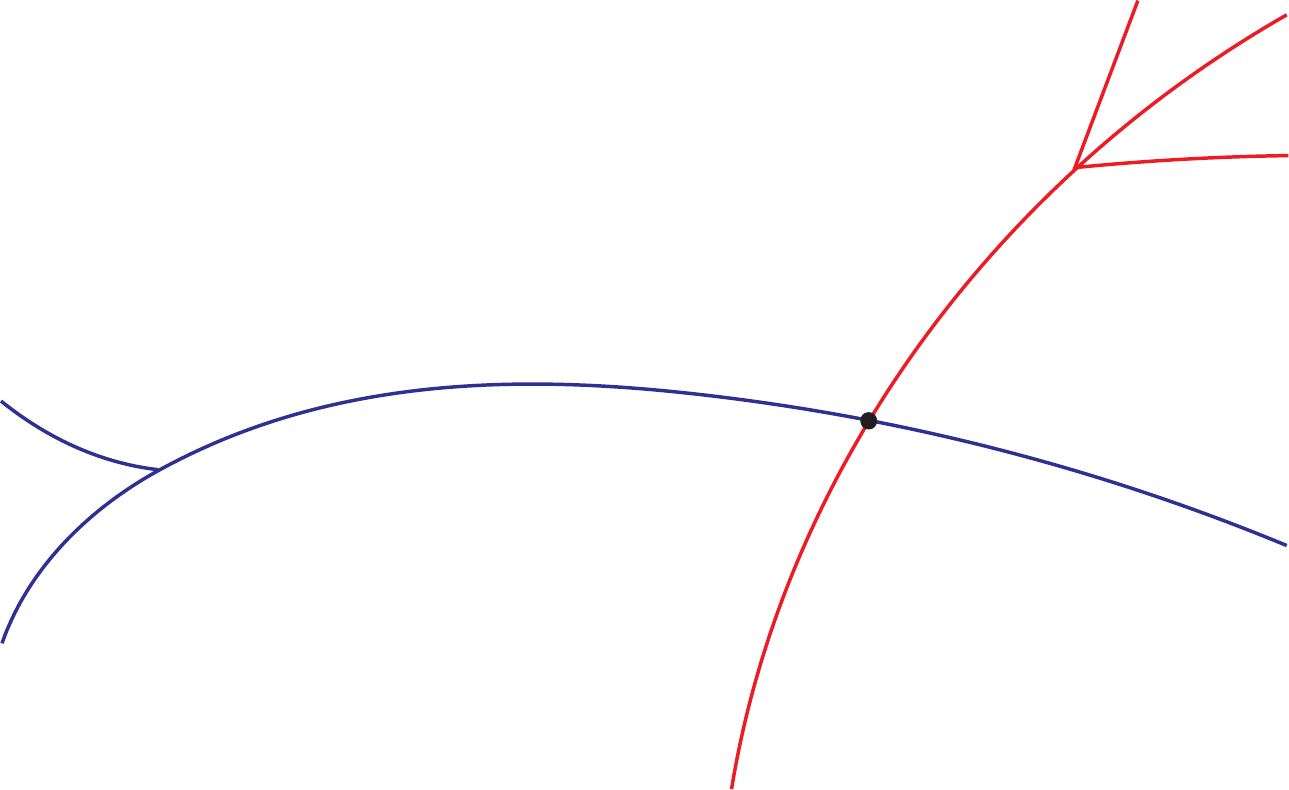}} }
\caption{The four quadrants of a regular point on singular leaves.}
 \label{fig:quadrants} 
\end{figure}

\subsection{A dynamical trichotomy} 
 
 As well as the trivial case, there is another special possible structure to a bifoliated plane.  
 \begin{definition} \label{def:skewed}
A bifoliated plane $(P,\cF^+, \cF^-)$ is called \emph{skewed} if it is homeomorphic to the strip $\{(x,y)\in \R^2 \mid x-1<y<x+1\}$ foliated by vertical and horizontal lines.  
 \end{definition}
 
Barbot and Fenley showed that Anosov flows on $3$-manifolds satisfy a trichotomy:  a flow is either orbit equivalent to a suspension, or has orbit space homeomorphic to a skewed plane, or both leaf spaces are non-Hausdorff.  
We will show that this trichotomy also holds in the setting of Anosov-like actions on bifoliated planes.  

 Before proving the trichotomy, we note the following complementary result to Remark \ref{rem_affine_counterexample}.  
 \begin{rem}  \label{rem:skew_counterexample}   
 Skewed planes also admit Anosov-like actions that do not arise from pseudo-Anosov flows (necessarily Anosov in this case, since the foliations are non-singular).  One may even take these to be faithful actions of 3-manifold groups.  As one such example, let $G \subset \mathrm{PSL}(2,\bC)$ be the fundamental group of a closed, hyperbolic 3-manifold group.  If the {\em trace field} of $G$ admits a real place, then this group embeds (non-discretely) into $\mathrm{PSL}(2,\bR)$, and its action on $S^1 \cong \R \cup \infty$ will lift to the universal cover $\R \cong \wt{S^1}$, commuting with the group of deck transformations.  In the language of \cite{BMB}, this is a \emph{hyperbolic-like} action on the line.  Considering the line as the boundary of the diagonal strip (and hence the leaf space of both $\cF^+$ and $\cF^-$), such an action on the line induces an Anosov-like action on the diagonal strip in $\bR^2$ preserving each foliation.   
Thurston \cite{Thurston:3MFC} classified the possible actions on skewed planes which come from Anosov flows, which he calls {\em extended convergence groups}, and this example is easily seen to fail Thurston's condition of acting properly discontinuously on the associated space of triples, so it does not arise from an Anosov flow.   
 \end{rem}

 \begin{theorem}[Dynamical trichotomy for Anosov-like actions] \label{thm_trivial_skewed_or_nowheredense}
 Let $(P,\cF^+, \cF^-)$ be a bifoliated plane with an Anosov-like action of a group $G$. Then exactly one of the following holds:
  \begin{enumerate}[label=(\roman*)]
   \item The bifoliated plane $(P,\cF^+, \cF^-)$ is trivial.
   \item The bifoliated plane $(P,\cF^+, \cF^-)$ is skewed.
   \item There is either a singular point in $P$, or the leaf spaces of $\cF^+$ and $\cF^-$ are both non-Hausdorff.
  \end{enumerate}
 \end{theorem}
 
 \begin{proof}[Proof of Theorem \ref{thm_trivial_skewed_or_nowheredense}]
  Suppose first that $(P,\cF^+, \cF^-)$ does not have singular points and that the leaf space of $\cF^+$ is Hausdorff, and hence homeomorphic to $\bR$. A simple rewriting of \cite[Theorem 3.4]{Fen_Anosov_flow_3_manifolds} or of \cite[Th\'eor\`eme 4.1]{Bar_caracterisation} using the properties of Anosov-like actions will then show that $\cF^-$ is also homeomorphic to $\bR$ and that $(P,\cF^+, \cF^-)$ is either trivial or skewed. For the sake of the reader, we quickly recall the arguments, following \cite{Bar_caracterisation} whose proof is more readily seen to only depend on the action on the bifoliated plane (Fenley's proof in \cite{Fen_Anosov_flow_3_manifolds} has some arguments using the flow).
  
  Let $\cL^+\simeq \bR$ be the leaf space of $\cF^+$. For any leaf $l^-\in \cF^-$, we consider the set $I(l^-) \subset \cL^+$ defined by 
  \[
   I(l^-) :=\lbrace l^+\in \cL^+ \mid l^- \cap l^+ \neq \emptyset\rbrace.
  \]
Then $I(l^-)$ is open and connected, that is, an interval. So there exist two functions $\alpha,\omega\colon \cL^- \to \bR \cup\{-\infty,+\infty\}$ such that $I(l^-) = (\alpha(l^-),\omega(l^-))$.

The first thing to note is that $\alpha$ (and $\omega$) is either always finite or always infinite:
\begin{claim}
 If there exists $l_0^-$ such that $\alpha(l_0^-)= -\infty$, then $\alpha \equiv -\infty$. Similarly, if $\omega$ takes the value $\infty$ at any point, then $\omega \cong \infty$.
\end{claim}

\begin{proof}
The set $A_{-\infty}:= \{ x\in P : \alpha(\cF^-(x))= -\infty\}$ is invariant under the action of $G$ and saturated by $\cF^-$. Since the action of $G$ is assumed to be topologically transitive, it suffices to show that $A_{-\infty}$ is open.

Let $x \in A_{-\infty}$. Since the set of fixed points of elements of $G$ is dense in $P$ (Axiom \ref{Anosov_like_dense_fixed_points}), we can find $g_1,g_2\in G$ such that $g_1\cdot \cF^-(x)$ 
 and $g_2\cdot \cF^-(x)$ intersect $\cF^+(x)$ in points $y_1,y_2$ which lie on either side of $x$ in $\cF^+(x)$.
Since $\alpha(g_1\cdot \cF^-(x)) = \alpha(g_2\cdot \cF^-(x)) = -\infty$, we conclude that for any point $z\in \cF^+(x)$ between $y_1$ and $y_2$, we also have $\alpha(\cF^-(z))=-\infty$. Thus $A_{-\infty}$ contains the saturation of the interval between $y_1$ and $y_2$ by $\cF^-$, which is an open neighborhood of $x$, proving the claim about $\alpha$.  The case for $\omega$ is completely analogous.
 \end{proof}

 \begin{claim}
  $\alpha \equiv -\infty$ if and only if $\omega \equiv +\infty$
 \end{claim}

 \begin{proof}
 Suppose that $\omega$ is always finite.  Consider some point $x\in P$ fixed by an element $g \in G$. Then $\omega(\cF^-(x))$ is also fixed by $g$, so (by Axiom \ref{Anosov_like_A1}) $g$ has a fixed point on this leaf. If we were to have $\alpha \equiv -\infty$, then $\cF^+(y) \cap \cF^-(x) \neq \emptyset$, contradicting the uniqueness of fixed points on a given leaf (Axiom \ref{Anosov_like_A1}).  Thus, $\alpha$ is always finite when $\omega$ is.  The reverse implication is symmetric.  
 \end{proof}
If both $\alpha$ and $\omega$ are infinite, then the plane is trivial, so we can assume now that both $\alpha$ and $\omega$ are finite.
 To finish the proof, we will show that in this case, any leaf $l^-$ of $\cF^-$ separates the leaves $\alpha(l^-)$ and $\omega(l^-)$ in $P$. This will imply that the leaf space $\cL^-$ of $\cF^-$ is Hausdorff, hence homeomorphic to $\bR$. 

 Let $C\subset P$ be the set of points $x$ such that $\cF^-(x)$ does \emph{not} separate $\alpha(\cF^-(x))$ and $\omega(\cF^-(x))$.  We assume $C$ nonempty for a contradiction. 
 For any $x\in C$, denote by $D_x$ the connected component of $P\smallsetminus \cF^-(x)$ that does not contain $\alpha(\cF^-(x))$ and $\omega(\cF^-(x))$. Since $\cL^+\simeq \bR$, we have$D_x \subset C$. Hence $C$ is open and, by Axiom \ref{Anosov_like_dense_fixed_points} contains a point, call it $x$ again, fixed by an element $g$. Up to replacing $g$ with $g^2$, we have that $\alpha(\cF^-(x))$ and $\omega(\cF^-(x))$ are both fixed by $g$. Thus there exists $y\in \alpha(\cF^-(x))$ and $z\in \omega(\cF^-(x))$ both fixed by $g$, and we must have that $x$, $y$ and $z$ are the corners of two lozenges with a shared side (one half-leaf of $\cF^+(x)$).
 Now, picking another fixed point $x'$ of some element $g'$ close to $x$ in the half-space $D_x$, we get a contradiction.

We conclude that $\cL^-$ is homeomorphic to $\bR$.  Now the fact that $(P,\cF^+, \cF^-)$ is skewed follows readily: The plane $P$ can be seen as a subset of the plane obtained as the product of the leaf spaces $\cL^+\times\cL^- \simeq  \bR^2$ bounded by the graphs of the function $\alpha$ and $\omega$. Up to an isotopy, we can make the graphs of $\alpha$ and $\omega$ be the two lines $y=x\pm 1$ (see \cite[Theorem 3.4]{Fen_Anosov_flow_3_manifolds}).
\end{proof}

Next we will show that a non-affine, Anosov-like action does not admit infinite strips on which the foliations are trivial. To make this precise, we introduce the following definition.  

\begin{definition}[Infinite product regions]\label{def_product_region}
	An \emph{infinite product region} is an open region $U \subset P$ such that (up to reversing the roles of $+$ and $-$), we have
	\begin{enumerate}[label=(\roman*)] 
		\item the boundary, $\partial U$, consists of a compact segment $I^+$ of a leaf of $\cF^+$ and two half-leaves $r_1^-, r_2^-$ of $\cF^-$,
		\item for each $x\in U$, the leaf $\cF^+(x)$
		intersects both $r_1^-$ and $r_2^-$, as in  Figure \ref{fig:product_region}.
	\end{enumerate}
\end{definition}

\begin{figure}[h]
	\labellist 
	\small\hair 2pt
	\pinlabel $r_1^-$ at 160 130
	\pinlabel $r_2^-$ at 160 20 
	\pinlabel $I^+$ at 50 70 
	 \pinlabel $\ldots$ at 310 75 
	\endlabellist
	\centerline{ \mbox{
			\includegraphics[width=6cm]{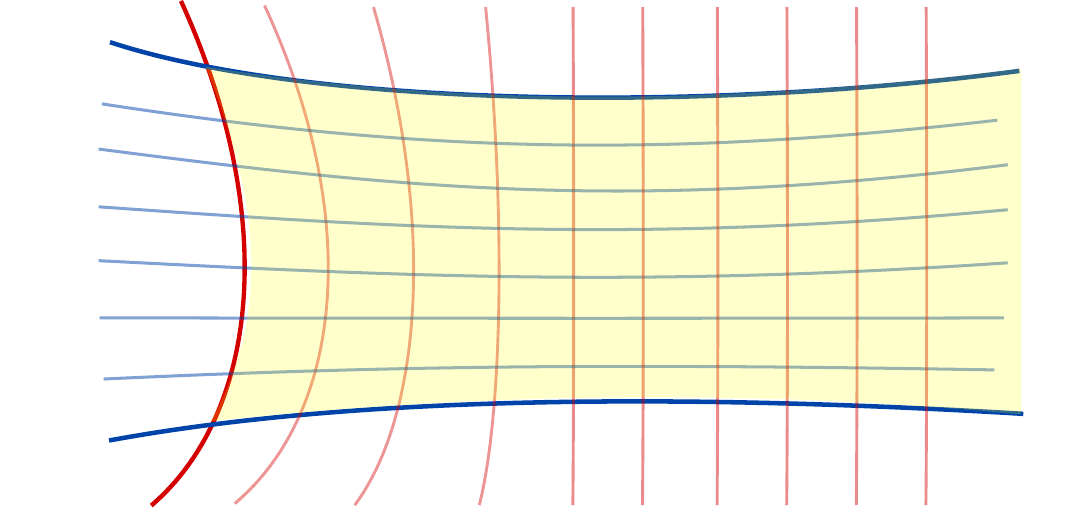}}}
	\caption{An infinite product region with base $I^+$} 
		\label{fig:product_region} 
\end{figure}
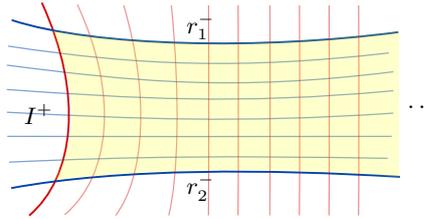

 \begin{proposition}\label{prop:no_product}
  Let $(P,\cF^+, \cF^-)$ be a bifoliated plane with Anosov like action of $G$.  The plane $P$ has an infinite product region if an only if the action is affine and $(P,\cF^+, \cF^-)$ is trivial.
 \end{proposition}
 
 To simplify the proof of this proposition, we first give a easy consequence of Axioms \ref{Anosov_like_topologically_transitive} and \ref{Anosov_like_dense_fixed_points}:
 \begin{lemma}\label{lem_orbit_of_leaves_are_dense}
  Let $(P,\cF^+, \cF^-)$ be a bifoliated plane with Anosov like action of $G$. Then the image of any leaf of $\cF^+$ or $\cF^-$ under $G$ is dense in $P$.
 \end{lemma}
 
 \begin{proof}
  Let $U$ be an open set in $P$ and $l$ a leaf in $\cF^{+}$. (The $\cF^-$ case is analogous.)
  By Axiom \ref{Anosov_like_topologically_transitive}, there exists $x\in P$ and $g\in G$ such that $\cF^-(x)\cap l\neq \emptyset$ and $gx \in U$.  Furthermore, we can choose $x$ close enough to $l$ to ensure that there are no singular points on $\cF^-(x)$ between $x$ and $l$. Now by Axiom \ref{Anosov_like_dense_fixed_points}, we can choose a point $y$, fixed by some element $h\in G$, very close to $gx$ so that $y\in U$ and $\cF^-(y)\cap gl\neq \emptyset$. Then, for $n$ large enough, either $h^{n} g l\cap U \neq \emptyset$, or $h^{-n} g l\cap U \neq \emptyset$, as sought.
 \end{proof}

With this lemma it is easy to prove the proposition.
 \begin{proof}[Proof of Proposition \ref{prop:no_product}]
 The reverse direction is immediate from the definition of trivial plane, so we need only prove the forwards direction.   This is established for Anosov flows in \cite[Theorem 5.1]{Fenley_structure_branching}, and more to the point for us, for \emph{transitive} Anosov flow in \cite[Proposition 3.5]{Fenley_one_sided}. It easily translates to the setting of Anosov-like actions. For completeness, we give the proof:
 
 Suppose $U$ is an infinite product region bounded by interval $I^+$, and half-leaves $r_1^-$ and $r_2^-$ (switching the roles of $+$ and $-$ has no impact on the proof)  and suppose that there exists two leaves $l_1^-, l_2^-\in \cF^-$ that are non-separated. By Lemma \ref{lem_orbit_of_leaves_are_dense}, the orbit of $l_1^-$ is dense so after translating by some $g \in G$, we can assume $l_1^-$ intersects $U$. Since $l_2^-$ is non-separated with $l_1^-$, it cannot intersect the product region $U$.
 
 By density of fixed points (Axiom \ref{Anosov_like_dense_fixed_points}) and the fact that since $l_1^-,l_2^-$ are not separated, we can choose a point $x$ near $l_2^-$, fixed by some non-trivial $g\in G$ such that $\cF^+(x)\cap l_2^-\neq \emptyset$ and $\cF^-(x)\cap U\neq \emptyset$.  After passing to a power of $g$ if needed, we will have that $g$ preserves each half-leaf through $x$ and contracts $\cF^+(x)$.  Then for $n$ large enough, $g^n(l_2^-)$ will intersect $U$ and $g^n(l_1^-)$ will be entirely contained in $U$, a contradiction with the fact that $U$ is an infinite product region.
 
 Thus the leaf space of $\cF^-$ is Hausdorff. Similarly, we can show that there are no singular points. If $y$ was a singular point, as above, acting by the group $G$, we may assume that $\cF^-(y)$ intersects $U$. But then, we may again find a point $x$ close to $y$ and fixed by some non trivial element $g\in G$ such that $g^n(y)\in U$. This is impossible as infinite product regions cannot contain singular points.
 
By Theorem \ref{thm_trivial_skewed_or_nowheredense}, we deduce that the plane is either skewed or trivial. As skewed planes have no infinite product regions, we conclude that the plane is trivial, and the action affine.
\end{proof}

The following lemma is the only result that uses Axiom \ref{Anosov_like_totallyideal} 
(note that the conclusion does not hold if $l_1$ and $l_2$ are on opposite sides of a totally ideal quadrilateral). It will only be referenced in an essential way in the proof of Proposition \ref{prop_unique_limit_for_NC} at the end of Section \ref{sec:nonskewed_proof}, but is also used as a shortcut in the proof of Lemma \ref{lem:accumulation_for_non_corner_leaves}.

\begin{lemma} 
 \label{lem_distinct_leaves_distinct_saturations}
  Let $G$ be a group with an Anosov-like action on a nontrivial bifoliated plane $(P,\cF^+, \cF^-)$.  Then for any two distinct leaves $l_1, l_2 \in \cF^+$  there exists a leaf $l\in \cF^-$ that intersects one but not the other, i.e. (up to switching 1 and 2)  
 we have 
  \[
   l\cap l_1 \neq \emptyset  \text{ and } l\cap l_2 = \emptyset.
  \]
  The same holds reversing the roles of $+$ and $-$. 
 \end{lemma}
  
 \begin{proof}
  Consider the set $I = \left\{ l^-\in \cF^- \mid l^-\cap l_1 \neq\emptyset  \text{ and } l^-\cap l_2 \neq\emptyset\right\}$.  If it is empty we are done.  If not, 
 consider the region $R$ formed by the union of the segments of leaves of $I$ between $l_1$ and $l_2$.  
By Proposition \ref{prop:no_product} we have no infinite product regions, so $R$ is bounded on either side by either a segment of a leaf, a leaf, or a union of non-separated leaves, i.e.  
 $ \partial R  \subset \{l^-_i\}\cup \{f^-_j\}$ where $\{l^-_i\}$ and $\{f^-_j\}$ are some unions of pairwise non-separated leaves.  Each union can consist of a unique element, a finite union or an infinite union.  
First note that, if some leaf $l^-$ in $\partial I$ intersects $l_i$, then it cannot intersect $l_{3-i}$, and we are done.

So we assume that no leaf in $\partial R$ intersects either $l_1$ or $l_2$.  
If both $\{l^-_i\}$ and $\{f^-_j\}$ each consisted of a unique leaf, then $l_1$, $l_2$ and $\partial R$ would make up the boundary of a totally ideal quadrilateral, contradicting Axiom \ref{Anosov_like_totallyideal}.
So we must have that at least one (say $\{l^-_i\}$) consists of more than one leaf. 
Thus, by Axiom \ref{Anosov_like_periodic_non-separated} there is some element $g$ fixing each leaf of $\{l^-_i\}$.  Note also that by Lemma \ref{lemma_scalloped_ideal_corner} $\{l^-_i\}$ is a finite union.
We deduce that both $l_1$ and $l_2$ are two sides of a line of lozenges fixed by $g$.  
 This implies that there is in fact some leaf $f^-_j$ of  $\{f^-_j\}$ intersecting $l_1$ and not $l_2$.   
 \end{proof}

\subsection{Lozenges, corners and non-corners}
We give here some further preliminary results involving lozenges, corners and perfect fits. Most of the results in this section are extensions to Anosov-like actions of known results on (pseudo)-Anosov flows, the one exception being Lemma \ref{lem_nowhere_dense}, which is new.

\begin{lemma} \label{lem:periodic_pf_corner} 
 Let $(P,\cF^+, \cF^-)$ be a bifoliated plane with an Anosov-like action of $G$.  If some nontrivial element $g\in G$ fixes a point $x$, and some half-leaf through $x$ makes a perfect fit, then $x$ is the corner of a lozenge.
\end{lemma} 
\begin{proof} 
	This is standard (it is for instance used in the proof of \cite[Theorem 3.3]{Fenley_QGAF}), we sketch a proof for completeness.   By passing to a power, we may assume that $g$ preserves all half-leaves through $x$.  Let $r_x$ be a half-leaf from $x$ making a perfect fit with a leaf $l$, and assume without loss of generality that $r_x = r_x^+ \in \cF^+$ and $l \in \cF^-$.  By Observation \ref{obs:fix_perf_fit} $g(l)=l$, so there exists $y \in l$ such that $g(y)=y$. Let $r_y^- \subset l$ be the half-leaf based at $y$ that makes a perfect fit with $r_x^+$.
	
	Take $r_x^- \in \cF^-(x)$ to be the half-leaf such that $r_x^+$ and $r_x^-$ bound the quadrant at $x$ that contains $y$, and take $r_y^+ \in \cF^+(y)$ to be the half-leaf such that $r_y^-$ and $r_y^+$ bound the quadrant at $y$ that contains $x$. We will show that $r_x^-$ makes a perfect fit with $r_y^+$, completing the proof. 
	
	Let $Z$ be the set of points $z \in r_x^+$ such that $\cF^-(z)$ intersects $r_y^+$. This set cannot contain $x$, since that would mean that $r_x^-$ intersects $r_y^+$ and the (necessarily unique) intersection point would therefore be fixed by $g$, contradicting \ref{Anosov_like_A1}. Since $Z$ is nontrivial, connected and $g$-invariant, it must be all of $r^+_x \setminus \{x\}$. Similarly, the $\cF^+$-leaf through every point in $r_y^-$ intersects $r_x^-$. Thus $r_x^-$ makes a perfect fit with $r_y^+$, forming the other sides of the desired lozenge.
\end{proof} 
An essential fact in the topological study of (pseudo)-Anosov flows in $3$-manifolds is that if two orbits are freely homotopic, then there are coherent lifts to the orbit space that are corners of a chain of lozenges (see \cite[Theorem 3.3]{Fenley_QGAF}\footnote{The theorem is stated for Anosov flows, but the same proof holds for pseudo-Anosov, see \cite[Theorem 4.8]{Fen_foliation_good_geo}.}). In our setup, this translates into the following proposition.

\begin{proposition}\label{proposition:cofixed_lozenge}
 Let $(P,\cF^+, \cF^-)$ be a bifoliated plane with an Anosov-like action of $G$. If some nontrivial $g \in G$ fixes distinct points $x, y \in P$ then there is a chain of lozenges in $P$ containing both $x$ and $y$ as corners.  
\end{proposition}

\begin{proof}
The proof of \cite[Theorem 3.3]{Fenley_QGAF} applies here with only cosmetic changes, so we just sketch the argument.

Given $x$ and $y$ as above, let $Q$ be the quadrant of $x$ containing $y$, and consider the set $S$ of leaves of $\cF^{-}$ that intersect both 
$\cF^{+}(x)$ and $Q$.  
This set is $g$-invariant, indexed by points along the half-leaf of $\cF^{+}(x)$ bounding $Q$ and bounded on one side by $\cF^-(x)$. 
It must also have some additional boundary because $\cF^{\mp}(x) \cap \cF^{\pm}(y) = \emptyset$. Hence, there is a (unique) leaf $l$ of the boundary such that either $y\in l$ or $l$ separates $x$ from $y$. This leaf $l$ is necessarily preserved by $g$, since $g$ fixes $x$ and $y$ and preserves $S$.  By \ref{Anosov_like_A1} $g$ fixes some point $x_1$ on $l$.  The set $\{l^- \in S : \cF^+(x_1) \cap l^- \neq \emptyset \}$ is $g$-invariant and contains an open neighborhood of $l$ in the leaf space of $S$, so by the hyperbolicity of the action of $g$ on $\cF^{+}(x)$ given by  \ref{Anosov_like_A1}, the set is equal to $S$. Since there are no infinite product regions (Proposition \ref{prop:no_product}), we deduce that $\cF^+(x_1)$ makes a perfect fit with either $\cF^-(x)$ or a leaf non-separated with it. By Lemma \ref{lem:periodic_pf_corner} above, in the first case $x$ and $x_1$ are the two corners of a lozenge, and in the latter case, they are corners of a chain of adjacent lozenges.

If $x_1=y$, then we are done. Otherwise we iterate the construction and find a sequence $x_n$ of corners, all fixed by $g$, in a chain of lozenges starting at $x$ and disjoint from $\cF^{\pm}(y)$.  
The final step is to show that this process must terminate.  If it does {\em not} terminate, one obtains an infinite sequence of leaves $\cF^+(x_n)$, disjoint from $\cF^{\pm}(y)$, which are forced to accumulate on some leaf or leaves. Exactly one of these limit leaves, call it $l'$, separates $x_n$ from $y$.  Since $\cF^+(x_n)$ and $\cF^+(y)$ are all fixed by $g$, the leaf $l'$ is as well, and thus contains a fixed point $x_\infty$ of $g$.  But then $\cF^-(x_\infty)$ will intersect the fixed leaf $\cF^+(x_n)$ for sufficiently large $n$, contradicting the uniqueness of fixed points.
\end{proof}

\begin{lemma}\label{lem_power_fixes_max_chain}
 Let $(P,\cF^+, \cF^-)$ be a bifoliated plane with an Anosov-like action of $G$.  Suppose $g \in G$ preserves all half-leaves through a corner point $x$.  Then $g$ preserves each corner, and each lozenge in the maximal chain containing $x$.  
In particular, if $g$ fixes both corners of a lozenge $L$, then $g$ fixes all corners and lozenges in the chain containing $L$. 
\end{lemma}

\begin{proof}
First note that $g$ fixes all half-leaves through a corner $x$ of a lozenge $L$ if and only if $g$ fixes both corners of $L$.  
Let $x$ be a corner point fixed by $g$, and suppose $g$ fixes all half leaves through $x$. Let $\mathcal C$  be the maximal chain of lozenges containing the lozenge $L$ with $x$ as corner.  Then $g$ fixes each other corner of every lozenge in $\mathcal C$ containing $x$.  
Since chains are by definition connected, we conclude inductively that $g$ fixes each corner of each lozenge in $\mathcal \cC$, as claimed.  
\end{proof}

In section \ref{sec:ideal_circle}, we will make frequent use of the following 
 consequence of Proposition \ref{proposition:cofixed_lozenge}.   

\begin{lemma}\label{lem:nonsep_family}  
Suppose that $S$ is a maximal set of pairwise nonseparated leaves in $\cF^+$ (maximal meaning that no leaf outside of $S$ is non-separated with a leaf of $S$).  Then the leaves of $S$ may be indexed $l_k$, for $k$ in 
some set (finite, infinite or bi-infinite) of consecutive integers, so that 
for each $k$, there exists a unique leaf of $\cF^{-}$ making a perfect fit with both $l_k$ and $l_{k+1}$.   Consequently, the limit leaves are all fixed by a common nontrivial element $g \in G$.  The same holds with the roles of $+$ and $-$ reversed. 
\end{lemma} 

\begin{proof}
Let $l, l'$ be two leaves in $S$.  By Axiom \ref{Anosov_like_periodic_non-separated}, both are fixed by the same element $g\in G$. \emph{A priori}, $g$ may depend on $l$ and $l'$.  Let $x$ and $x'$ be the points of $l$ and $l'$ fixed by $g$.  Applying the procedure from Proposition \ref{proposition:cofixed_lozenge} produces a chain of adjacent lozenges between $x$ and $x'$, configured like those with sides labelled $l^{1,+}_i$ in Figure \ref{fig_scalloped}.  
Thus, leaves of $S$ may be {\em locally} indexed by consecutive integers and form sides of adjacent lozenges -- maximality implies that all sides $l^{1,+}_i$ actually are elements of $S$.  Patching this procedure together, we conclude that 
globally, $S$ consists of sides of adjacent lozenges, and is either finite, or indexed by $\bN$ or by $\bZ$ as described, and leave of $S$ are all fixed by any element that fixes both corners of one of these lozenges.
\end{proof}

\begin{rem}
 We will see in section \ref{sec:ideal_circle}, that, in fact, a maximal union cannot be indexed in in the manner above by $\bN$, as an infinite union of non-separated leaves forces the existence of a scalloped region (see Definition \ref{def_scalloped}) and thus is contained in a bi-infinite collection of pairwise nonseparated leaves.  
\end{rem}

In the proof of the main theorem, we often need to distinguish corners of lozenges from other points.   For convenience, we make the following definition.  
 
 \begin{definition}
A point $x\in P$ is called a \emph{corner} if it is a corner of some lozenge. It is called a \emph{non-corner} if it is not the corner of any lozenge.   Thus, when we say {\em non-corner fixed point} we mean a point fixed by some element of $G$, which is not the corner of any lozenge. 
\end{definition}
  
  \begin{lemma}[Non-corner criterion] \label{lem_no_corner_criterion}  
  Let $x\in P$, and let $r^{\pm}_x$ be two half-leaves bounding a quadrant $Q$. 
  If $r^+_x$ and $r^-_x$ intersect a each of pair of leaves making a perfect fit, or if they each intersect leaves of a singular point, then there are no lozenges in $Q$ with $x$ as a corner.
  \end{lemma} 
  
  \begin{proof}
   Assuming the set-up of the lemma, let $l^{\pm}$ denote the leaves of $\cF^{\pm}$ that intersect $r^{\pm}_x$ and either make a perfect fit or are the leaves of a prong.  Suppose that there exists some leaf $l_1^+$ making a perfect fit with $r^-_x$ and a leaf $l_1^-$ making a perfect fit with $r^+_x$ in $Q$, as shown in Figure \ref{fig:no_corner_lem}.   
   
The leaf $l_1^+$ must be inside the half-space $P\smallsetminus l^+$ that does not contain $x$, and $l_1^-$ is inside the half-space $P\smallsetminus l^-$ that does not contain $x$.  By assumption, these two half-spaces are disjoint, so these leaves cannot intersect, as would be required to form the other corner of a lozenge in $Q$.  We conclude that $x$ cannot be a corner of a lozenge in $Q$. 
  \end{proof}

  \begin{figure}
   \labellist 
   \small\hair 2pt
      \pinlabel $x$ at 80 93 
     \pinlabel $x$ at 350 90 
     \pinlabel $l^+$ at 160 45 
     \pinlabel $l^-$ at 30 140 
          \pinlabel $l^-$ at 330 135 
     \pinlabel $l^+$ at 444 45 
  \endlabellist
     \centerline{ \mbox{
\includegraphics[width=8cm]{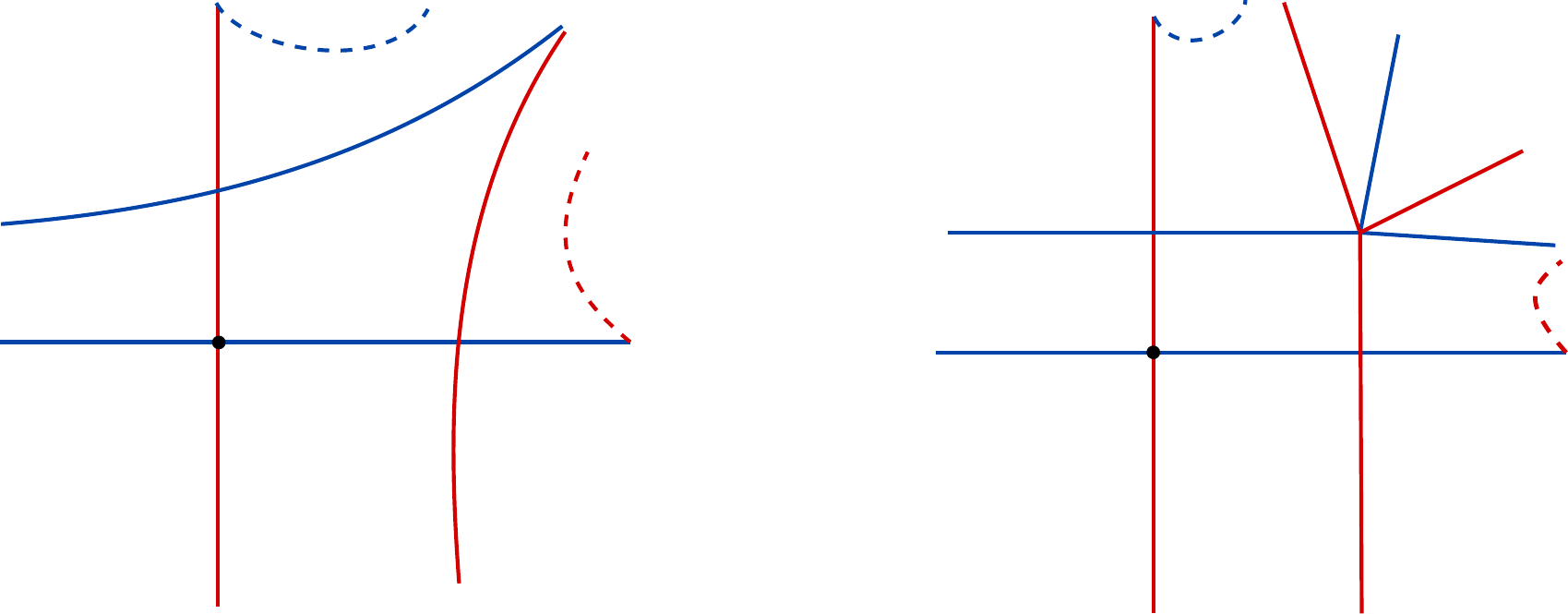}} }
\caption{The dotted leaves making perfect fits cannot intersect to form a lozenge with $x$ as a corner.}
 \label{fig:no_corner_lem} 
\end{figure}
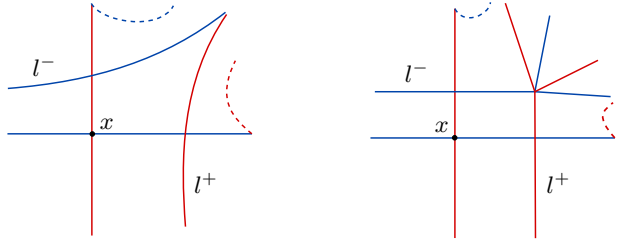

We will use extensively the following fact that, when the plane is neither skewed nor trivial, we can find non corner points everywhere.

\begin{lemma}\label{lem_nowhere_dense}
 Let $(P,\cF^+, \cF^-)$ be a bifoliated plane with an Anosov-like action of $G$ and assume $P$ either has non-separated leaves in $\cF^+$ or a singular point. Then the set of corner points in $(P,\cF^+, \cF^-)$ is nowhere dense.
\end{lemma}

 \begin{proof}
Suppose first that there exist two non-separated leaves in $\cF$.  By \ref{Anosov_like_periodic_non-separated}, some element of $G$ fixes both leaves, so by Proposition \ref{proposition:cofixed_lozenge}, these leaves belong to a chain of lozenges.  Note that at least two lozenges in the chain must share a side, for if they only share corners, then each leaf of $\cF$ in each lozenge is separated from each other by leaves that meet the interior of the lozenges.   
So let $L_1$ and $L_2$ be two lozenges sharing a side.  Suppose for concreteness this side is in $\cF^-$. 

By Axiom \ref{Anosov_like_dense_fixed_points} we can pick a point $x$ in $L_1$ fixed by some  $g\in G$ (note that $x$ is automatically non-singular since it is inside a lozenge).  Up to replacing $g$ with $g^{-1}$, it will contract $\cF^-(x)$ and so $g(L_1)$ will intersect both $L_1$ and $L_2$, as shown in Figure \ref{fig:non-corner-dense}. 

If $p \in g (L_1) \cap L_2$ is a fixed point of some element (again, it is necessarily non-singular), then each of its quadrants will contain a perfect fit involving a leaf of $L_1$, $g(L_1)$, or $L_2$.  See Figure \ref{fig:non-corner-dense}.  
Thus, $p$ satisfies the conditions of Lemma \ref{lem_no_corner_criterion} so cannot be a corner.  This gives an open set without corners, and by topological transitivity of the action, it follows that there is an open, dense set of non-corners.

 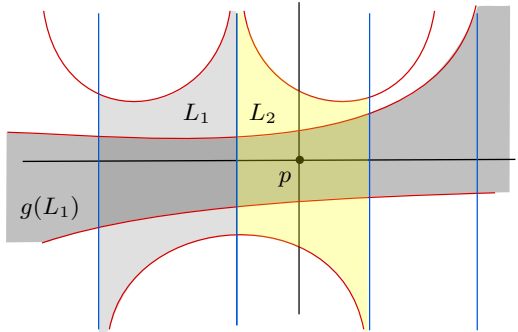
\begin{figure}
   \labellist 
  \small\hair 2pt
     \pinlabel $L_1$ at 135 165 
    \pinlabel $L_2$ at 180 165 
    \pinlabel $g(L_1)$ at 35 100 
    \pinlabel $p$ at 195 120 
 \endlabellist
     \centerline{ \mbox{
\includegraphics[width=7cm]{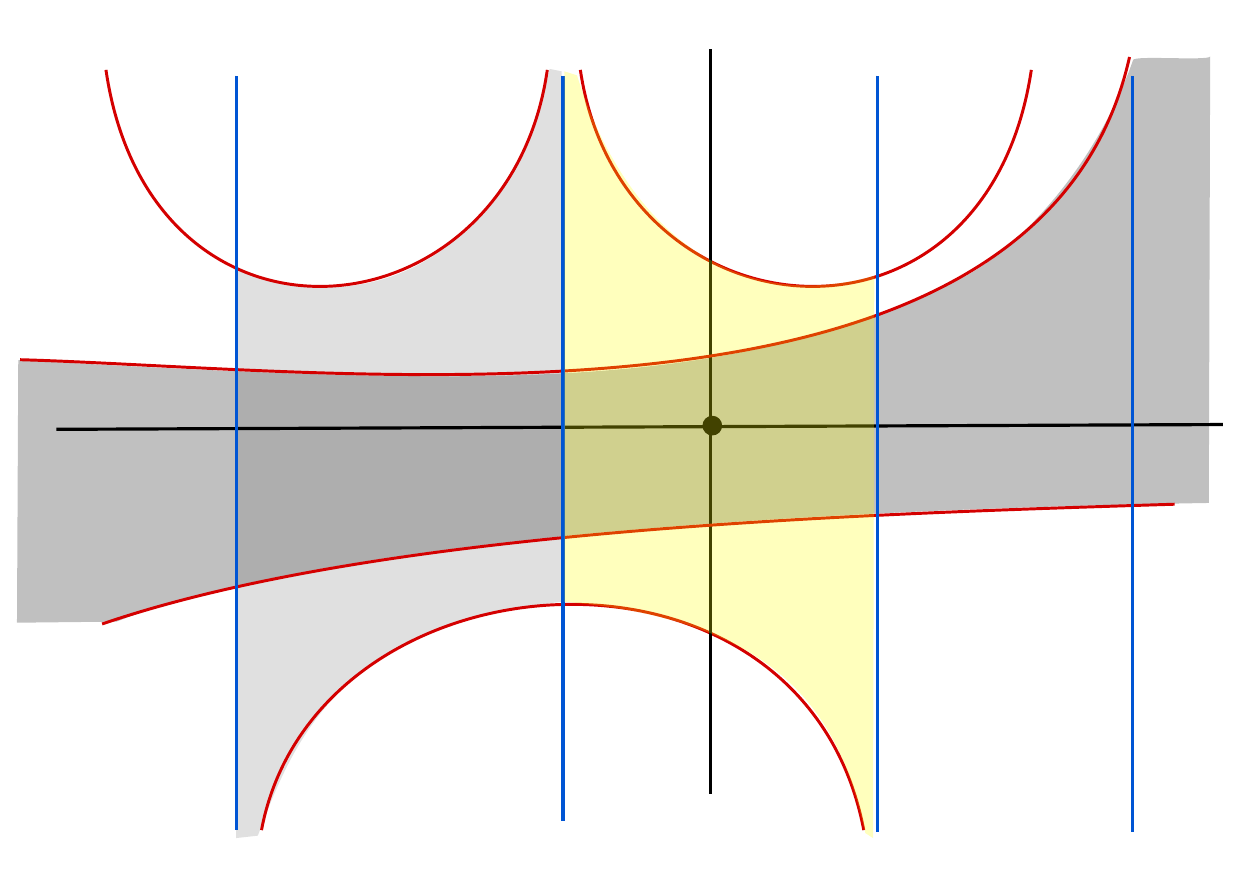}}}
\caption{$L_2 \cap g( L_1)$ contains no corner points}
 \label{fig:non-corner-dense} 
\end{figure}

Now, we treat the second case, assuming that no leaves are non-separated, and there is a singular point $x$ in $P$.
By density of fixed points (Axiom \ref{Anosov_like_dense_fixed_points}), we can find $g \in G$ such that $\cF^+(g x)\cap\cF^-(x)\neq \emptyset$. Let $y = g x$, which is also a singular point, and  let $Q$ be the quadrant of $x$ containing $y$.
If $x$ is a corner of a lozenge $L$, with opposite corner $p$ in $Q$, then a neighborhood of $\cF^+(y)\cap L$ on the side containing $x$ cannot contain any corner, as illustrated in Figure \ref{fig_proof_non-corners} on the left.

Thus, we are left to consider the case where $Q$ does not contain any lozenge with corner $x$.  We claim that in this case, neither of the half-leaves of $\cF^{\pm}(x)$ bounding $Q$ make a perfect fit with another leaf in $Q$.   Indeed, if there where such a perfect fit with a leaf $l$, then the element $f \in G$ fixing the singular point $x$ will fix $l$ also.  Lemma \ref{lem:periodic_pf_corner} then says that $x$ would be a corner of a lozenge with its other corner in $l\subset Q$, a contradiction.   This proves the claim.  

To finish the proof, let $z= \cF^+(y)\cap \cF^-(x)$, and let $[z,y]^+$ be the segment in $\cF^+(y)$ between $z$ and $y$. Consider the set of leaves of $\cF^{-}$ that pass through $[z,y]^+$. By the above argument, none of these leaves can make a perfect fit with $\cF^+(x)$. Thus, either some such leaf has another leaf of $\cF^-$ with which it is non-separated, or all leaves through $[z,y]^+$ intersect $\cF^+(x)$.
The former case does not arise as we assumed that there are no non-separated leaves. So we are in the latter case and we conclude that $\cF^-(y)$ intersects $\cF^+(x)$, as in Figure \ref{fig_proof_non-corners} right.

Let $f \in G$ be the element fixing $x$ and contracting $\cF^+(x)$.  
Then $\cF^{-1}(f(y))$ and $\cF^{+}(f^{-1}(y))$ intersect at a point $q$ in the rectangle with corners $x$ and $y$.  Points in the quadrant of $q$ containing $y$ that are sufficiently close to $q$ will all satisfy the conditions of Lemma \ref{lem_no_corner_criterion}, for each of their quadrants, since every pair of rays will intersect two prongs of a singularity. Thus we obtain an open set consisting of points which are not corners of a lozenge, as in Figure \ref{fig_proof_non-corners}.  
As in the previous case, topological transitivity now implies that corners are nowhere dense.\qedhere

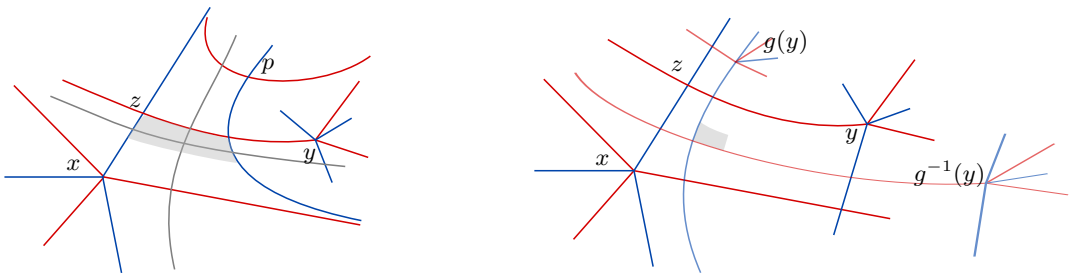
\begin{figure}[h] 
   \labellist 
  \small\hair 2pt
     \pinlabel $x$ at 38 60 
  \pinlabel $y$ at 168 65 
    \pinlabel $z$ at 72 95
    \pinlabel $p$ at 145 115 
      \pinlabel $x$ at 329 63 
  \pinlabel $y$ at 466 76
  \pinlabel $g(y)$ at 430 128  
  \pinlabel $z$ at 370 115
    \pinlabel $g^{-1}(y)$ at 520 55  
 \endlabellist
     \centerline{ \mbox{
\includegraphics[width=14cm]{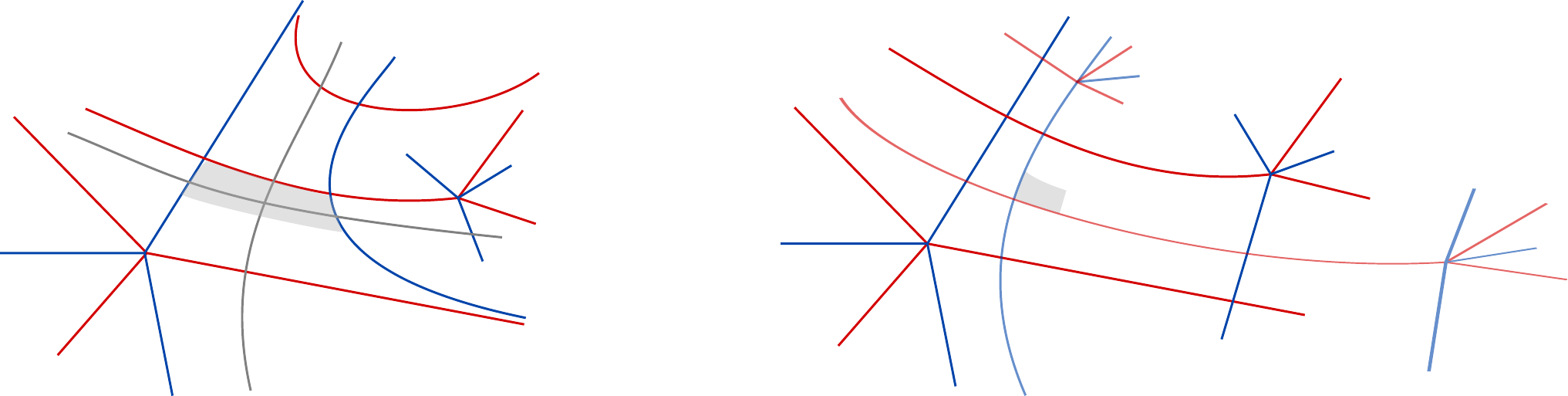} }}
\caption{The shaded regions cannot contain any corners of lozenges.}
\label{fig_proof_non-corners} 
\end{figure}
\end{proof}

\subsection{Scalloped regions and trees} \label{sec_scalloped_and_trees}

The following definition comes from \cite{Fenley_structure_branching}, \cite{BarbFen_pA_toroidal}.  
While it appears technical, it is simply describing the configuration depicted in Figure \ref{fig_scalloped} right.

 \begin{definition}\label{def_scalloped}
  A \emph{scalloped region} is an open, unbounded set $U \subset P$ with the following properties:
  \begin{enumerate}[label=(\roman*)]
   \item The boundary $\partial U$ consists of the union of four families of leaves $l_k^{1,+}, l_k^{2,+}$ in $\cF^+$ and $l_k^{1,-},l_k^{2,-}$ in $\cF^-$, indexed by $k\in \bZ$.
   \item The leaves of each family $l_k^{i,\pm}$, $k \in \bZ$ are pairwise non-separated.
   \item The boundary leaves are ordered so that there exists a (unique) leaf $f_k^{1,-}$ that makes a perfect fit with $l_k^{1,+}$ and $l_{k+1}^{1,+}$. Moreover, $f_k^{1,-}$ accumulates on the leaves $\cup_{i\in \bZ} l_i^{1,-}$ as $k \to \infty$, and  on $\cup_{i\in \bZ} l_i^{2,-}$ as $k \to -\infty$.  The analogous statement holds for leaves making perfect fits with the other families $l_k^{i,\pm}$.
   \item The bifoliation is trivial inside $U$, i.e., for all $x \neq y\in U$, $\cF^+(x)\cap \cF^-(y) \neq \emptyset $ and  $\cF^+(y)\cap \cF^-(x) \neq \emptyset $ and $U$ contains no singular points.
  \end{enumerate}
 \end{definition}

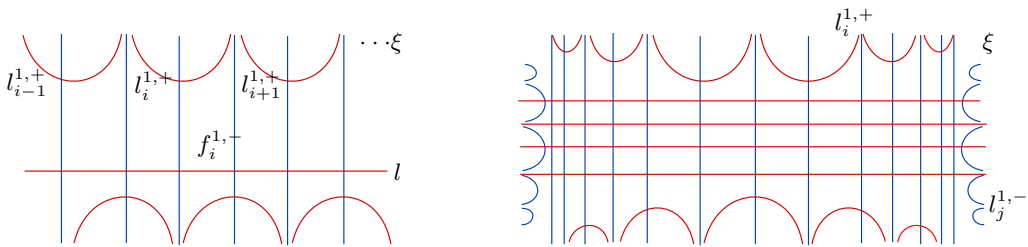
\begin{figure}
   \labellist 
  \small\hair 2pt
     \pinlabel $l_{i-1}^{1,+}$ at 8 85 
     \pinlabel $l_{i}^{1,+}$ at 70 84 
   \pinlabel $l_{i+1}^{1,+}$ at 122 84
    \pinlabel $f_i^{1,-}$ at 102 55
      \pinlabel $l$ at 188 42 
   \pinlabel $\cdot \cdot \cdot \xi$ at 180 105  
   \pinlabel $\xi$ at 475 105
  \pinlabel $l^{1,-}_j$ at 485 25  
  \pinlabel $l_{i}^{1,+}$ at 410 115 
 \endlabellist
     \centerline{ \mbox{
\includegraphics[width=13cm]{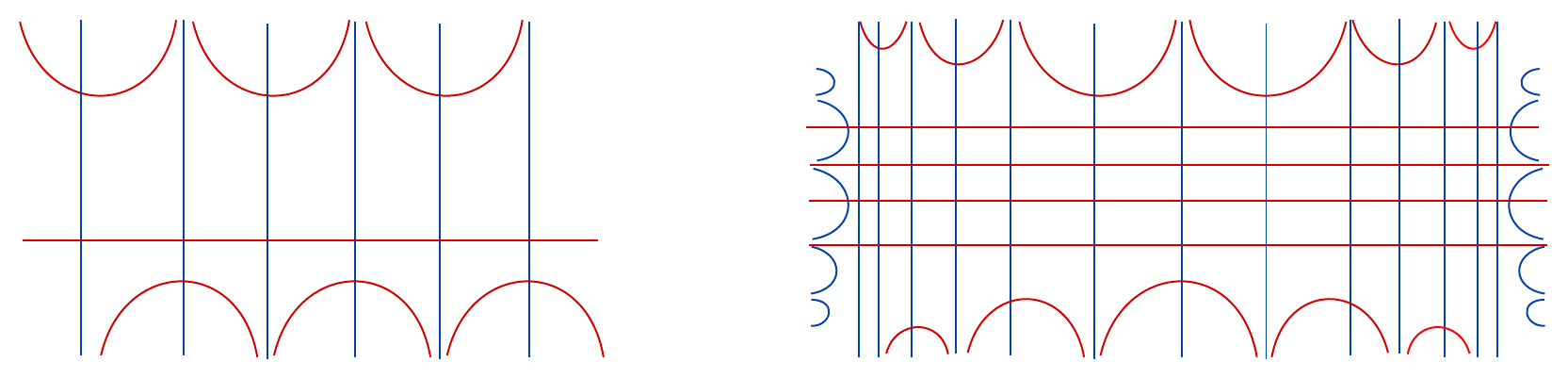} }}
\caption{Adjacent lozenges and a scalloped region.  The leaves $f_i^{j,-}$ are vertical, $f_i^{j,+}$ are horizontal.}
\label{fig_scalloped}
\end{figure}
 
 \begin{lemma} \label{lemma_scalloped_ideal_corner}
Let $(P,\cF^+, \cF^-)$ be a nontrivial bifoliated plane with Anosov-like action of a group $G$.
Suppose $P$ contains an infinite set of pairwise non-separated leaves.  Then this set is contained in the boundary of a single scalloped region in $P$.  
 \end{lemma}

 \begin{proof}
 Suppose $\{l_k^{1,+}\}$ is an infinite set of pairwise non-separated leaves of $\cF^+$ (the proof is the same for leaves of $\cF^-$). Without loss of generality we may assume that $\{l_k^{1,+}\}$ is a maximal (with respect to inclusion) family of pairwise nonseparated leaves.    
 By Lemma \ref{lem:nonsep_family} we can take the indexing to be by consecutive integers $k$ and so that $l_k^{1,+}$ and $l_{k+1}^{1,+}$ make a perfect fit with a common leaf $f_k^{1,-}$ (forming a shared side of adjacent lozenges).    Since $\{l_k^{1,+}\}$ is infinite, without loss of generality we assume the index set contains all integers $k>0$.

  Consider the sequence of leaves $(f_k^{1,-})$.  As $k\to+\infty$, there are {\em a priori} three possibilities:
  \begin{enumerate}
   \item $(f_k^{1,-})$ does not accumulate onto anything, or 
   \item  $(f_k^{1,-})$ accumulates onto a union of leaves that is finite or indexed by $\bN$, or
   \item  $(f_k^{1,-})$ accumulates onto a bi-infinite union of leaves.  
  \end{enumerate}

We first eliminate case (1), showing this leads to an infinite product region.  Let $l, l' \in \cF^+$ be leaves intersecting $f_1^{1,-}$ and passing through the chain of lozenges.  Let $r$ and $r'$ be the half leaves of $l$ and $l'$, respectively, that intersect $f_2^{1,-}$.  For all $i \geq 1$, the region between $f_i^{1,-}$ and $f_{i+1}^{1,-}$ bounded by segments of $r$ and $r'$ is trivially foliated.  In the case where $(f_k^{1,-})$ does not accumulate, the union of these trivially foliated regions produces an infinite product region, contradicting Proposition \ref{prop:no_product} since the bifoliated plane has non-separated leaves so is nontrivial.

Now suppose that we are in case (2) and $(f_k^{1,-})$ accumulates onto a union of leaves, say
$l_1^{1,-}, l_2^{1,-} \dots$, either finite or indexed by $\mathbb{N}$.   
 Let $g \in G$ be an element that fixes each $l_k^{1,+}$.  Then $g$ fixes each $f_k^{1,-}$, and so fixes $l_1^{1,-}$.  Let $x$ be the fixed point of $g$ on $l_1^{1,-}$. Then, for $k$ large enough, $f_k^{1,-}\cap \cF^+(x)\neq \emptyset$. But then $\cF^+(x)$ must contain several distinct fixed points of $g$, contradicting Axiom \ref{Anosov_like_A1}.

Therefore, we must be in case (3), where we have a bi-infinite family of leaves accumulated by $(f_k^{1,-})$.  Denote this family by $l_j^{1,-}$, indexed by $j\in \bZ$.  It remains to show that this forces the structure of a scalloped region, as in Figure \ref{fig_scalloped} (right). 

To do this, we may repeat the entire argument above using $l_j^{1,-}$, taking leaves $f_j^{1,+}$ forming perfect fits with $l_j^{1,-}$ and $l_{j+1}^{1,-}$ and intersecting the leaves $f_k^{1,-}$ already considered.   For concreteness, reverse the indexing of $l_j^{1,-}$ and $f_j^{1,+}$ if needed so that $f_1^{1,+}$ separates the original leaves $l_k^{1,+}$ from $f_0^{1,+}$, i.e the indexing increases as the $f_j^{1,+}$ move in the direction of $\{l_k^{1,+}\}$
As $k \to \infty$, the sequence $f_j^{1,+}$ limits to a bi-infinite union of leaves, we claim that this contains $\{l_k^{1,+}\}$ (and hence is equal to $\{l_k^{1,+}\}$ by maximality).   To prove the claim, suppose for contradiction that this does not hold, so some leaf $l$ and nearby leaf $l'$ both separate the limit leaves from $\{l_k^{1,+}\}$.  Then $l$ and $l'$ both intersect each $f_k^{1,-}$, but do not intersect any $l^{1,-}_j$.  Since the limit of the $f_k^{1,-}$, as $k \to \infty$ is the union $l^{1,-}_j$, we see an infinite product region bounded by $l$ and $l'$, again contradicting that the plane is nontrivial.  This proves the claim, and we conclude additionally that $\{l_k^{1,+}\}$ is indexed by $\bZ$. 

Thus, we can run the argument again taking $k$ to $-\infty$ and conclude $(f_k^{1,-})$ accumulates as $k \to -\infty$ onto a bi-infinite union of leaves, $l^{2,-}_j$, with leaves $f_j^{2,+}$ making perfect fits between them; these are forced to intersect each $f_k^{1,-}$ and hence together with $f_j^{1,-}$ form the sides of lozenges bridging between the leave of $\{l^{1,-}_j\}$ and $\{l^{2,-}_j\}$.  

We may also run the argument with $f_j^{1,+}$ taking $j \to -\infty$ to conclude these limit onto a fourth bi-infinite family of leaves, which we call $l_k^{2,+}$.  As in the previous case, taking leaves $f_k^{2,-}$ that make perfect fits with $l_k^{2,+}$ and $l_{k+1}^{2,+}$, these form the remaining sides of an infinite family of lozenges with sides from $f_k^{i,+}$ and $l_k^{i, -}$.   Let $U$ denote the union of these lozenges.  Since $U$ is a union of adjacent lozenges, the bifoliation is trivial in $U$.  By construction the boundary of $U$ is the union of the leaves of $\{l_k^{1,+}\}$, $\{l_k^{2,+}\}$ $\{l_k^{1,-}\}$ and $\{l_k^{2,-}\}$, with families of leaves limiting on to the desired sides, as claimed in the statement of the Lemma.  
\end{proof}

\begin{lemma}  \label{lem:Z2stabilizer}
If $G$ has an Anosov-like action on a bifoliated plane with a scalloped region $U$, then the stabilizer of $U$ in $G$ is virtually isomorphic to $\bZ^2$. 
\end{lemma}
We remark that this was proved for orbit spaces of (pseudo)-Anosov flows by Barbot and Fenley, see e.g. \cite[Corollary 4.8]{Fenley_structure_branching}.  We give an elementary proof here for Anosov-like actions.  

\begin{proof} 
Let $G_U$ denote the stabilizer of the scalloped region $U$. A subgroup $G'_U$ of index at most $4$ will preserve (setwise) each of the four bi-infinite families of leaves forming the boundary of $U$.   The action of $G'_U$ on each family preserves the linear ordering of the indices, i.e. can be considered a translation of $\bZ$.  Thus, there is a homomorphism  $G'_U \to \bZ \times \bZ$ by considering the translation actions of $G'_U$ on the families $l_k^{1,+}$ and  $l_k^{1,-}$ (keeping the notation from above).  We claim that the kernel of this homomorphism is trivial, and the image contains nontrivial elements of $\{0\} \times \bZ$ and $\bZ \times \{0\}$.  For the first statement, if $h$ is in the kernel, then $h$ preserves each leaf $f_k^{1,-}$ and $f_j^{1,+}$.  Each leaf $f_k^{1,-}$  intersects all of the infinitely many $f_j^{1,+}$, giving infinitely many fixed points for $h$.  Thus, $h$ is identity.  

For the image, we note that there is some nontrivial element $g\in G$ that fixes each $l_k^{1,+}$ (since these are nonseparated), so $g^2$ fixes each half leaf through the corners of the lozenges with sides in $l_K^{1,+}$ so lies in $G_U$.  Since $g^2$ is nontrivial, it cannot act trivially on $U$, so must have nontrivial image in $\bZ \times \bZ$, i.e. is a nontrivial element of $\bZ \times \{0\}$.  The argument for the other factor is completely analogous.  
\end{proof}

\begin{definition}\label{def_tree_scalloped}
A tree of scalloped regions $T \subset P$ is a chain of lozenges such that each lozenge in $T$ shares each of its sides with some other lozenge in $T$.
\end{definition}

We record here some immediate consequences of the above definition: 
\begin{rem}[Properties of trees of scalloped regions] \label{rem:scalloped}
Let $T$ be a tree of scalloped regions.  Then:
\begin{itemize}
\item $T$  is a maximal chain of lozenges.  
\item Each lozenge in $T$ is contained in exactly two distinct scalloped regions; one with the $\cF^+$ sides of the lozenge in its boundary and one with the $\cF^-$ sides in the boundary.  
\item Two distinct scalloped regions in $T$ are either disjoint or intersect in a unique lozenge.
\item Each $p$-prong corner of a lozenge in $T$ is the corner of exactly $2p$ distinct lozenges in $T$ (where $p=2$ when the corner is regular).  
\item Finally, Axiom \ref{Anosov_like_periodic_non-separated} implies that some nontrivial element $g\in G$ fixes every corner of the tree.
\end{itemize}  
\end{rem}

We conclude this section by describing what feature of flows gives rise to a tree of scalloped regions in the orbit space.

\begin{definition}[Barbot, Fenley] 
A {\em periodic Seifert piece} of a pseudo-Anosov flow on $M$ is a Seifert fibered piece of the JSJ decomposition of $M$ such that, up to finite
powers, a regular fiber of some Seifert fibration is freely homotopic to a closed orbit of the flow.   If such a piece exists, the bounding tori can always be taken transverse to the flow.  
\end{definition}

The reason for saying ``some Seifert fibration" is that there do exist Anosov flows on manifolds with a JSJ piece obtained as a neighborhood of an embedded one-sided Klein  
bottle. Such pieces admit two distinct Seifert fibrations; we say the piece is periodic if either one of the fibers is freely homotopic to a closed orbit.

\begin{proposition} \label{prop:tree_iff_periodic} 
The orbit space of a pseudo-Anosov flow contains a tree of scalloped regions if and only if the flow has a periodic Seifert piece $M'$ such that, for each boundary torus $T_i$ of $M'$, $T_i$ can be isotoped to be transverse to the flow and $\pi_1(T_i)$ is generated by elements freely homotopic to periodic orbits.  
\end{proposition}

We postpone the proof of this proposition to our forthcoming paper with Fenley \cite{BFM_in_preparation} as we do not need the full strength of it here.  Instead, we will just prove the weaker statement given in Proposition \ref{prop_conditions_no_tree}, i.e, we will show that if an orbit space for a flow $\varphi$ contains a tree of scalloped regions, then there must be at least one torus transverse to $\varphi$ in $M$ and at least one periodic Seifert piece. 

\begin{proof}[Proof of Proposition \ref{prop_conditions_no_tree}] 
By Lemma \ref{lem:Z2stabilizer}, the stabilizer of a scalloped region $U$ is virtually isomorphic to $\bZ^2$.  Thus, we may apply 
work of Barbot and Fenley (\cite[Th\'eor\`eme B]{Barbot_MPOT} in the Anosov case, \cite{BarbFen_pA_toroidal} for the generalization), saying that $U$ corresponds to a torus transverse to the flow.  Precisely, the scalloped region is the projection to the orbit space of a lift of a transverse torus with fundamental group the associated $\bZ^2$.  Thus, condition (3), and hence also condition (1) of  Proposition \ref{prop_conditions_no_tree} prohibits the existence of scalloped regions in the orbit space.  

To see that (2) and (4) each prohibit trees of scalloped regions, we now analyze how the existence of a tree of scalloped regions in the leaf space constrains the structure of the fundamental group of $M$.  Suppose that $T$ is a tree of scalloped regions and fix one such scalloped region $U_0$ in $T$. As noted above, its stabilizer is virtually isomorphic to $\bZ^2$, generated by some $g$ and some $h$ in $\pi_1(M)$, and we can take $g$ to be the element fixing every corner of the tree.  
Let $U_1$ be a scalloped region intersecting $U_0$ in $T$ along a lozenge fixed by $g$.  Then $g$ stabilizes $U_1$ (but $h$ does not) and so there exists $h_1$ stabilizing $U_1$ that commutes with $g$.  Moreover, it is easily verified that $h_1$ does not commute with $h$.  Thus, the centralizer of $g$ in $\pi_1(M)$ 
is neither cyclic nor isomorphic to $\bZ^2$.  
Therefore, by \cite[Theorem 3.1]{AFW}, the centralizer $C(g)$ of $g$ is the fundamental group of a Seifert piece of the JSJ   decomposition in $M$, and thus $g$ represents a regular fiber and is freely homotopic to a periodic orbit of the flow.  By definition, this means the flow is periodic in that Seifert piece.
\end{proof}

\subsection{Reduction of Theorem \ref{thm_main_Anosov} to \ref{thm_main_general}}
We end this section by the proof that Theorem \ref{thm_main_Anosov} is implied by Theorem \ref{thm_main_general}.
The first step is to show that having one of the flow be transitive is enough to know that both flows are transitive, hence both flows give a transitive Anosov-like action. 

\begin{proposition}\label{prop_1transitive_both_transitive}
 Let $\varphi$ and $\psi$ be two Anosov flows on a closed $3$-manifold $M$. Suppose that $\varphi$ is transitive and that $\fix(\varphi)= \fix(\psi)$. Then $\psi$ is also transitive. 
\end{proposition}

\begin{proof}
Let $\psi$ be a non-transitive flow.  
 Then $\psi$ has a (finite) number of basic sets, with at least one attractor and one repeller (see, e.g., \cite{FH_book}). 
 Then, by Brunella \cite{Brunella} (see also the proof of Proposition 2.7 of Mosher in \cite{Mosher_homologynormI} for the pseudo-Anosov case), the basics sets are separated by disjoint tori transverse to $\psi$. In particular, there exists a set $T$, comprised of a disjoint union of embedded, essential tori transverse to $\psi$, such that $M\smallsetminus T = M_1 \sqcup M_2$, where $M_1$ contains an attractor and $M_2$ contains a repeller (for instance we can take $M_1$ containing only one of the attractors of $\psi$ and $M_2$ containing all the other basic pieces).
 As a consequence, no periodic orbits intersect $T$.  

To prove the proposition, we need to show that no transitive flow $\varphi$ can have the same free homotopy data as $\phi$.  Let $\varphi$ be any transitive flow on $M$.  We will show there exists $\gamma \in \pi_1(M)$ whose free homotopy class is represented by a periodic orbit of $\varphi$, but that is \emph{not} freely homotopic to a curve inside either $M_1$ or $M_2$.  
 
Since each of $M_1$ and $M_2$ contains an attractor or repeller of $\psi$, they contain 
periodic orbits that are not freely homotopic to a curve in $T$. Choose such orbits $\beta_1 \in M_1$ and $\beta_2 \in M_2$.  
Since we assumed $\fix(\varphi)= \fix(\psi)$, there exists orbits $\alpha_1$, $\alpha_2$ of $\varphi$ representing the same free homotopy classes (up to reversing orientation) in $M$ as $\beta_1$ and $\beta_2$, respectively.   
 Since $\varphi$ is transitive, the specification property together with the Anosov closing lemma (see, e.g., \cite{FH_book}) allows us to build a periodic orbit $\alpha$ of $\varphi$ that shadows both $\alpha_1$ and $\alpha_2$ for some time.
 Since neither $\alpha_1$ nor $\alpha_2$ are freely homotopic into $T$, such an orbit $\alpha$ cannot be freely homotopic to an element of $\pi_1(M_1)$ or $\pi_1(M_2)$, as desired.   
\end{proof}

In the introduction, we used the notation $\fix(\varphi)$ for the set of conjugacy classes $[\gamma ]$ of elements $\gamma \in \pi_1(M)$ such that either $\gamma$ or $\gamma^{-1}$ is represented by a periodic orbit of $\varphi$.  We extend this to the context of Anosov-like actions.  
 \begin{notation}
  Given an Anosov-like action $\rho$ of a group $G$ on a bifoliated plane $(P,\cF^+, \cF^-)$, denote by $\fix(\rho)$ the set of elements $g\in G$ such that $\rho(g)$ has at least one fixed point in $P$.
 \end{notation}
Thus, if $\rho$ is the action on the orbit space of a transitive pseudo-Anosov flow $\varphi$, then $\fix(\rho) = \fix(\varphi)$.   We now can proceed with the main reduction.

\begin{proof}[Proof of Theorem \ref{thm_main_Anosov} assuming Theorem \ref{thm_main_general}]  
Suppose that $\varphi$ and $\psi$ are two pseudo-Anosov flows on a given compact manifold $M$ such that $\Phi(\cP(\varphi)) = \cP(\psi)$ for some automorphism $\Phi$ of $\pi_1(M)$.  We assume at least one flow is transitive. 
Since $M$ is a $K(\pi, 1)$ space, $\Phi = f_\ast$ for some homotopy equivalence $f\colon M \to M$. Since we are working in dimension 3, $f$ can be upgraded to a homeomorphism, and even a diffeomorphism, thanks to work of Waldhausen \cite{Waldhausen} and Gabai--Meyerhoff--Thurston \cite{GMT}.  After conjugating $\varphi$ by $f$, we can assume that $\cP(\varphi) = \cP(\psi)$ and we need to prove that these flows are now isotopically equivalent.  

Proposition \ref{prop_1transitive_both_transitive} implies that both flows are transitive, provided one is.
Thus, by Theorem \ref{thm:pA_is_anosov_like}, the actions of $\pi_1(M)$ on the orbit spaces of $\varphi$ and $\psi$ are Anosov-like actions on bifoliated planes.
The set of conjugacy classes of elements of $\pi_1(M)$ fixing a point in the orbit space of $\varphi$ is exactly $\cP(\varphi)$. Since $\cP(\varphi) = \cP(\psi)$, we deduce that the two actions have the same set of elements with fixed points.    Thus, the notion of {\em signs} of trees of scalloped regions in their orbit spaces is well-defined (see Remark \ref{rem:def_sign_tree_flow}), and these signs agree if and only if the actions on the orbit spaces are conjugate.

Thus, provided that neither of the actions are affine actions on a trivial bifoliated plane, then the conclusion follow directly from the statement of Theorem \ref{thm_main_general} together with the fact, proven by Barbot \cite[Theorem 3.4]{Bar_caracterisation} that conjugacy of the actions on the orbit space is equivalent to isotopy equivalence of the flow.

If, on the other hand, the orbit space of one of the flows, say $\varphi$, is a trivial bifoliated plane, then $\varphi$ is a suspension flow. Hence the manifold $M$ is the mapping torus of an Anosov diffeomorphism, and any other Anosov flow on $M$ is orbit equivalent to $\varphi$ by \cite{Plante81}.
\end{proof}

 \section{Nontrivial bifoliated planes and the ideal circle} \label{sec:ideal_circle}
 
 In \cite{Fen_ideal_boundaries}, Fenley defined a circle at infinity associated to the orbit space of a pseudo-Anosov flow. Frankel generalized this construction in \cite{Fra_Mobiuslike} to any \emph{generalized unbounded decomposition} of a topological plane.  As a special case of this, to any bifoliated plane $(P,\cF^+, \cF^-)$, one can associate an ideal circle $\Pbound$ by taking the \emph{end compactification} of the union of the foliations $\cF^+\cup \cF^-$ as in \cite{Fra_Mobiuslike}.   
 
 Following \cite{Fen_ideal_boundaries}, the ideal circle of a bifoliated plane has a topology that can be specified in terms of the foliations.  Specifically, a neighborhood basis of the topology at a point $\eta \in \Pbound$ can be taken to consist of all {\em convex regions bounded by polygonal paths}, where a {\em polygonal path} is a properly embedded, bi-infinite path made up of a finite collection of segments alternately in $\cF^+$ and $\cF^-$, with rays of $\cF^\pm$ at the ends, and convex means any leaf intersecting the boundary of the region does so in at most two points.   In particular, endpoints of leaves represent a dense subset of $\Pbound$. Homeomorphisms of $(P,\cF^+, \cF^-)$ preserving the foliations extend to homeomorphisms of the compactification.  
  
 We recall here for future reference a few basic results about this end compactification taken from section 3 of \cite{Fen_ideal_boundaries} (see in particular Lemma 3.6 and Lemma 3.14 therein).  
 \begin{proposition}\label{prop_ideal_circle} \cite{Fen_ideal_boundaries}
  Let $(P,\cF^+, \cF^-)$ be a bifoliated plane and $\Pbound$ its associated ideal circle. Then the following hold: 
  \begin{enumerate}[label=(\roman*)]
   \item The distinct ends of any leaf determine distinct points in $\Pbound$.
   \item If $r^+,r^-$ are two half leaves of $\cF^+$ and $\cF^-$ (respectively) that make a perfect fit, then their ends determine the same point on $\Pbound$.
   \item If $r$ and $r'$ are two half-leaves of either $\cF^+$ or $\cF^-$ that determine the same point on $\Pbound$, then there exist half-leaves $r_1, \dots r_n$ such that $r=r_1$, $r_n= r'$, and each consecutive pair $(r_i, r_{i+1})$ makes a perfect fit.
   \item If $r_n$ is a sequence of half-leaves converging to a single half-leaf $r$, then the endpoints of $r_n$ converge to the endpoint of $r$.  
  \end{enumerate}
 \end{proposition}

 In the case where $(P,\cF^+, \cF^-)$ is trivial, the ideal circle contains exactly four points which do not correspond to endpoints of leaves; and each complementary interval can be identified either with the stable or unstable leaf space.  In the skew case, there are exactly two points that are not endpoints of leaves (the two ``ends" of the infinite strip), and the sides of the infinite strip are naturally identified with the remaining ideal points.    We will show later in Proposition~\ref{prop_4_2_0_global_fixed_points},  for planes with an Anosov-like action, having a $G$-invariant finite set in $\Pbound$ exactly characterizes the trivial and skew case.   For now, we establish some necessary results on the structure of an Anosov-like action on the boundary of a {\em nontrival} bifoliated plane.

 \begin{observation} \label{obs:disjoint_ideal_closures}
Suppose $a \in P$ is a non-corner fixed point of $\alpha$, and $b \in P$ a fixed point of $\beta$.  If $\xi$ is a common fixed point of $\alpha$ and $\beta$ on $\Pbound$, then $b = a$ is the unique fixed point of $\beta$ in $P$. 
\end{observation} 

\begin{proof} 
Since $a$ is assumed to be a non-corner, one of the half-leaves through $a$ is the only leaf ending at $\xi$, by Proposition \ref{prop_ideal_circle}. Thus $\beta$ fixes this leaf through $a$, so by Lemma \ref{lem_same_fixed_leaf_same_fixed_point}, $a$ itself is fixed by $\beta$. 
\end{proof} 

 \begin{observation} \label{obs:escape_endpoints}
If $(P,\cF^+, \cF^-)$  is a nontrivial bifoliated plane and $l_n$ is a sequence of leaves in $\cF^+$ that escapes to infinity, then the endpoints of $l_n$ converge to a single point. 
\end{observation} 

Since this follows from \cite{Fen_ideal_boundaries}, we only sketch the proof. 
\begin{proof} 
Suppose $l_n$ has endpoints $\eta_n, \xi_n$ and $l_n$ escapes to infinity.  Pass to a subsequence so $\eta_n \to \eta$ and $\xi_n \to \xi$ in $\Pbound$.   Since $l_n$ escapes to infinity, one interval between $\eta$ and $\xi$ on $\Pbound$ contains no endpoints of leaves of $\cF^+$.  Thus, it consists of endpoints of half-leaves of $\cF^-$, which necessarily intersect $l_n$ for all $n$ sufficiently large, producing an infinite product region. 
\end{proof} 

Using this, Proposition \ref{prop_ideal_circle}, and the definition of scalloped region, we obtain the following.   See \cite[Lemma 3.27]{Fen_ideal_boundaries}. 

\begin{corollary}  \label{cor:ideal_scalloped} 
Let $l^{i,\pm}_k$ be any one of the four families of leaves comprising the boundary of a scalloped region.  Then, as $k \to \infty$, the endpoints of $l^{i,\pm}_k$ converge to a point $\xi$, and as $k \to -\infty$ to another point $\nu \neq \xi$.  Varying the families, exactly four points are obtained this way, each point the limit in one direction of two different families.  These four points are called the {\em corners} of the scalloped region. 
\end{corollary}

We will make frequent use of the following short lemma on convergence of ideal boundary points in nontrivial planes.  
 
 \begin{lemma}\label{lem_convergence_or_trivially_product} 
 Let $(P,\cF^+, \cF^-)$ be a nontrivial bifoliated plane with an Anosov like action of a group $G$.
 Let $(l_n)_{n \in \mathbf{N}}$ be a sequence of leaves of $\cF^+$.  Suppose that there exist ideal points $\xi_n, \eta_n\in \Pbound$ of $l_n$ such that $\xi_n$ converges to $\xi$ and $\eta_n$ converges to some point $\eta \neq \xi$.  Then the sequence $(l_n)$ converges to a union (finite or infinite) of leaves in $\cF^+$.   The same holds for a sequence of leaves of $\cF^-$. 
 \end{lemma}

 \begin{proof}
  If $l_n$ escapes to infinity in the leaf space (even along some subsequence) then the endpoints cannot converge to two distinct points, by Observation \ref{obs:escape_endpoints}.   Thus, for some subsequence we have that $l_{n_j}$ converges to a finite or infinite union of leaves.   If finite, the extreme leaves of this union have endpoints $\xi$ and $\eta$ by Proposition \ref{prop_ideal_circle}.  We claim $l_n$ converges to this finite union as well.  For indeed, that the endpoints converge to $\xi$ and $\eta$ forces all leaves $l_n$ to eventually lie in the same connected component of the finite union and converge to the limit.  
  
In the case where the union is infinite, it forms one side of a scalloped region, say for concreteness $l^{1,-}_k$; where as $k \to \pm\infty$ the endpoints of the leaves converge to corners of the scalloped region, necessarily $\xi$ and $\eta$. 
This also forces all leaves $l_{n_j}$ (for $j$ sufficiently large) to intersect the scalloped region.  Since the endpoints of $l_n$ also converge to the corners, $\xi$ and $\eta$, these leaves also intersect the scalloped region and converge to the family  $l^{1,-}_k$ as well. 
\end{proof}

 \subsection{Fixed points on the ideal circle}
In this section, we collect some results about the structure of fixed points on the ideal circle, primarily on nontrivial planes.  
 
 \begin{lemma}\label{lem_leaves_to_ideal_are_fixed}
 Let $(P,\cF^+, \cF^-)$ be a nontrivial bifoliated plane with Anosov like action of $G$.
 Suppose $g\in G$ has some fixed point in $P$ and preserves the orientation of $\Pbound$, as well as fixing $\xi \in \Pbound$.  If  $\xi$ is the endpoint of a leaf $l$, then $g(l) = l$.   
 \end{lemma}

 \begin{proof}
 Suppose $g$ has some fixed point $x \in P$.  Since $g$ preserves orientation of $\Pbound$ and has a fixed point on $\Pbound$, it cannot nontrivially permute the ends of the half-leaves through $x$.  Thus $g$ has at least {\em four} fixed points on $\Pbound$.  
Assume for contradiction $g(l) \neq l$.  Then the leaves $g^n(l)$ are all distinct, but share a common endpoint $\xi$.  Either as $n \to \infty$ or as $n \to -\infty$, the other endpoints of $g^n(l)$ will accumulate on some other fixed point $\eta \neq \xi$ for $g$.  

Thus, by Lemma \ref{lem_convergence_or_trivially_product}, $g^n(l)$ must accumulate to a finite or infinite union of pairwise non-separated leaves $\{f_k\}$, the closure of whose endpoints contain $\xi$. If this union is infinite it is contained in the boundary of a scalloped region with ideal corner $\xi$, and $g^n(l)$ will intersect the scalloped region, but this is incompatible with these leaves having endpoint $\xi$.  

If the union is finite, then the $f_k$ must all be fixed by $g$ and one of them, say $f_{0}$, has endpoint $\xi$. But then item \emph{(iii)} of Proposition \ref{prop_ideal_circle} implies that $f_{0}$ and $l$ are joined by a sequence of perfect fits, and, as $g$ fixes $f_{0}$ it must also fix $l$, contradicting our starting assumption. 
 \end{proof}

The following lemma gives a trichotomy for the dynamics of the action of an element of $G$ on the boundary. It will be useful in the later parts of this work, especially in the proof of Proposition \ref{prop_charac_tot_linked}, and the proof that a conjugacy between induced boundary actions can be extended to a conjugacy between actions on bifoliated planes.  

\begin{proposition}[Boundary action of elements on a nontrivial bifoliated plane] \label{prop:boundary_action_general}
Let $(P,\cF^+, \cF^-)$ be a nontrivial bifoliated plane with Anosov like action of $G$, and $g \neq 1 \in G$.  
\begin{enumerate} 
\item If $g$ fixes a non-corner point $x$, then the only points of $\Pbound$ fixed by $g$ are endpoints of $\cF^\pm(x)$. 

\item If $g$ fixes all corners in a maximal chain of lozenges $\cC$, then the set of fixed points of $g$ on $\Pbound$ is the closure of the set of endpoints of the sides of lozenges in $\cC$.  

\item If $g$ acts freely on $P$, then it either has at most two fixed points on $\Pbound$, or has exactly four fixed points and preserves a scalloped region.  
\end{enumerate} 
\end{proposition} 

\begin{proof} 
Assume first that $g$ fixes some point in $P$, so we are in the setting of cases 1 or 2 of the statement.   If this is a unique fixed point, by Lemma \ref{lem_power_fixes_max_chain} we may pass to a power to assume that $g$ fixes all half-leaves through the point. (This is automatically satisfied in the setting of case 2.)  Note that this does not affect the conclusion in case 1, since the set of fixed points of an element is a subset of the fixed set of any of its powers.   Let $\partial(\cF_g)$ denote the set of endpoints of leaves fixed by $g$. 

To establish the conclusion in these cases, it suffices to prove the following. 

\begin{lemma} If $\xi, \eta \in \partial(\cF_g)$ are adjacent points, i.e. we have $(\xi, \eta) \cap  \partial(\cF_g) = \emptyset$, then the interval $(\xi, \eta)$ contains no points fixed by $g$.  
\end{lemma}
\begin{proof}
Let $\xi$ and $\eta$ be two such points.  Consider first the case where $g$ has a unique fixed point $x$ in $P$.  Then $\xi$ and $\eta$ are the endpoints of leaves of $\cF^\pm(x)$, necessarily bounding a quadrant.  Let $l^+ \in \cF^+$ be a leaf intersecting this quadrant and $\cF^-(x)$, and let $\zeta$ denote the endpoint of the ray of $l^+$ in the quadrant.   Then, we have that $g^n(\zeta)$ must converge to a point that is fixed by $g$ and in the interval $[\xi, \eta]$, i.e., it must converge to either $\xi$ or $\eta$. Thus, up to replacing $g$ with $g^{-1}$, we may assume that $g^n(\zeta) \to \xi$.
Now choose any leaf $l^-$ that intersects both $l^+$ and  $\cF^+(x)$.  Since there are no infinite product regions, we must have that $g^{-n}(l^+) \cap l^- = \emptyset$ for sufficiently large $n$.  Since $l^-$ can be chosen as close as desired to a line of $\cF^-(x)$ with $\eta$ as an endpoint, this shows that $g^{-n}(\zeta)$ approaches $\eta$ as $n \to \infty$, so $g$ has no fixed points on $(\xi, \eta) \subset \Pbound$.

Now consider the case where $g$ does not have a unique fixed point, but instead preserves each corner of a maximal chain $\cC$ of lozenges.  Consider the set of leaves of $\cC$ which have $\xi$ as an endpoint.  We claim that among this set, there is one whose other endpoint is closest to $\eta$.  
To see this, note that $\eta$ cannot be the endpoint of a leaf from $\xi$ (by Lemma \ref{lem_leaves_to_ideal_are_fixed}, this leaf would contain a fixed point $y$ for $g$, and thus $\cF^\pm(y)$ would have an endpoint in $(\xi, \eta)$, contradicting our assumption). If our claim was false, i.e. 
if we had a sequence of leaves $l_i$ with endpoints $\eta_i$ approaching some closest point $\eta_\infty \neq \eta$, then the $l_i$ would limit onto a $g$-invariant union of leaves by Lemma \ref{lem_convergence_or_trivially_product}.  If the union were infinite, it would be contained in the boundary of a $g$-invariant scalloped region with $\xi$ as a corner, contradicting the fact that $\xi$ was an endpoint of a fixed leaf of $g$.  If it were finite, then the leaf with endpoint $\xi$ would have other endpoint equal to $\eta_\infty$ or closer to $\eta$ than $\eta_\infty$.   Thus, the claim holds and we have a closest leaf $l$.  

Let $x$ be the unique fixed point of $g$ on $l$.  Then the quadrant of $x$ containing $\eta$ in its boundary has no other lozenges of $\cC$ with corner $x$.  Since $\eta$ is the endpoint of a leaf fixed by $g$, we conclude that $\eta$ is the endpoint of a leaf through $x$.  
Thus, we are in the previous setting, and can use exactly the same argument to show that the dynamics of $g$ on $(\xi, \eta)$ is fixed point free.  This concludes the proof of the lemma. 
\end{proof}

We return to prove the third statement of the proposition, where $g$ acts freely.
Suppose $g$ does not fix any point in $P$.  Then it acts freely on the leaf spaces of $\cF^+$ and $\cF^-$ by \ref{Anosov_like_A1}.   Since $g$ acts freely on $\cF^+$, it has an {\em axis} $A$ consisting of leaves $l$ such that $l$ separates $g^{-1}(l)$ from $g(l)$.  Let $l \in A$ be a nonsingular leaf.  (See e.g.  \cite{Barb_ActionGroupe,BarbFen_pA_toroidal} for details on axes). 

 Let $\xi, \zeta$ be the endpoints of $l$.   If $g^n(\xi)$ and $g^n(\zeta)$ approach each other as $n \to \infty$, and also approach each other as $n \to -\infty$, then $g$ can have at most two fixed points on $\Pbound$.  So we can now assume that $g$ has more than two fixed points, and thus conclude that in at least one case (say, without loss of generality as $n \to \infty$) the endpoints of  $g^n(l)$ converge to distinct points.  Call  the two distinct limits
 $\xi_\infty,\zeta_\infty \in \Pbound$.   

 By Lemma \ref{lem_convergence_or_trivially_product}, the sequence $g^n(l)$ converges to a union of leaves, either finite or infinite.  
  If the union is finite, then a unique leaf, say $l_1$ has one of its endpoint equal to $\xi_\infty$, which is fixed by $g$. Thus $l_1$ must itself be fixed by $g$ and so $g$ has a fixed point in $P$, contrary to our assumption.  
Otherwise, the limiting union is infinite. In particular, $\xi_\infty$ is a fixed point of $g$ which is accumulated by endpoints of pairwise non-separated leaves.  By Corollary \ref{cor:ideal_scalloped} it is the corner of a scalloped region, which is by construction $g$-invariant.  

It remains to note that $g$ cannot have additional fixed points other than the corners of the scalloped region.
For this, take $\eta \in \Pbound$ not a corner of the scalloped region $U$ preserved by $g$.  Then $\eta$ is separated from the scalloped region by a unique leaf $l'$ in the boundary of $U$.  Since $g$ acts freely on $P$, $g(U) \neq U$, i.e. it is a different leaf in the boundary of $U$, and thus $g(\eta) \neq \eta$.  
 \end{proof} 
 
 The following immediate consequence will be an essential tool in the proof of Proposition \ref{prop_charac_tot_linked}.  
 \begin{corollary} \label{cor_fixed_points_are_joined}
Suppose some nontrivial $g\in G$ has a fixed point in $P$ and fixes two points $\xi,\eta\in \Pbound$.
  Then, for any neighborhoods $U$, $V$ of $\xi$ and $\eta$ respectively, there exists a sequence of leaves $l_1, \dots l_n$, all fixed by $g$, such that one end of $l_1$ is in $U$, one end of $l_n$ is in $V$, and for all $i$, either $l_i$ and $l_{i+1}$ intersect at a fixed point of $g$ or make a perfect fit.
\end{corollary}

 \begin{proof}
  By Proposition \ref{prop:boundary_action_general}, the neighborhoods $U$ and $V$ contain the endpoints of leaves fixed by $g$. The conclusion then follows directly from Proposition \ref{proposition:cofixed_lozenge} .
 \end{proof}

We conclude this section with a characterization of trivial and skewed planes in terms of the boundary action.  

\begin{proposition}\label{prop_4_2_0_global_fixed_points} Let $P$ be a bifoliated plane (possibly trivial) with an Anosov-like action of a group $G$.  
\begin{enumerate} 
\item 
If $P$ is skewed, then exactly two points on $\Pbound$ have finite orbits under $G$.
\item
If $P$ is trivial, then exactly four points $\Pbound$ have finite orbits under $G$.
\item
Otherwise, every point has an infinite orbit. 
\end{enumerate}  
\end{proposition} 

We note that this statement is equivalent to the following:  $P$ is trivial iff a finite index subgroup of $G$ fixes four boundary points, $P$ is skewed iff it is nontrivial and a finite index subgroup fixes two boundary points; and otherwise no finite index subgroup of $G$ fixes any point of $P$.  

\begin{proof} 
Suppose $P$ is skewed.  There are exactly two points of $\Pbound$ that are not the endpoints of any leaf, giving an invariant set of cardinality $2$ that is invariant by $G$.  We claim every other point has an infinite orbit: Indeed each other point $\xi$ is the endpoint of exactly two leaves, one of $\cF^+$ and one of $\cF^-$. Taking $g\in \fix(G)$ to be an element that preserves orientation and does not fix $\xi$, hyperbolicity implies that the orbit of $\xi$ by $g$ is infinite.

For a trivial bifoliated plane, there are exactly four points of $\Pbound$ that are not the endpoints of any leaf, and the argument is almost exactly as above (in this case, any element $\xi\in \Pbound$ that is not among the four non-endpoint leaves is the endpoint of a unique leaf, not two). 

Finally, for a nontrivial, non-skewed plane, consider two elements $g_1,g_2 \in G$ that fix distinct non-corner, non-singular points $x_1,x_2 \in P$ and their half-leaves (existence of such $g_i$ follows from Theorem \ref{thm_trivial_skewed_or_nowheredense}). Then Proposition \ref{prop:boundary_action_general} implies that each $g_i$ fixes exactly $4$ points on $\Pbound$, which must all be distinct. So the group generated by $g_1$ and $g_2$ has no global fixed points.
\end{proof}

 
 \section{Linking of fixed points and consequences} \label{sec:linking} 

Recall that our goal is to recover a conjugacy class of an action of $G$ from the set of elements acting with fixed points.  In this section, we do the first step towards the proof, by showing that the relative position of the fixed points of two elements $\alpha, \beta \in G$ can be understood in terms of which elements of the form $\alpha^n \beta^m$ act with fixed points in $P$.  This allows us to promote knowledge of which elements have fixed points to information about the position of fixed points in $P$.  We will later leverage this information to essentially reconstruct the plane $P$.   

\subsection{Total and partial linking} 
 
 \begin{definition}
Let $a,b\in P$.  We say that $a$ and $b$ are:
\begin{itemize}
 \item \emph{totally linked} if $\cF^+(a)\cap \cF^-(b) \neq \emptyset$ and $\cF^+(b)\cap \cF^-(a) \neq \emptyset$;
 \item \emph{unlinked} if $\cF^+(a)\cap \cF^-(b) = \emptyset$ and $\cF^+(b)\cap \cF^-(a) = \emptyset$ and 
 \item \emph{partially linked} otherwise. 
\end{itemize}
 \end{definition}
 
 The linkage of points can easily be read on the ideal circle, as follows: given two points $a$ and $b$, let $\idealpts(a)$ and $\idealpts(b)$ be the sets of endpoints on $\Pbound$ of the leaves $\cF^{\pm}(a)$ and $\cF^{\pm}(b)$, respectively.  Note that a point $a$ is totally linked with itself and any point that shares a leaf with $a$, but the much more interesting situation is for points $a$ and $b$ such that the ideal points of their leaves on the ideal circle are disjoint. When $\idealpts(a) \cap \idealpts(b) = \emptyset$, then it is immediate from the definition that the points $a$ and $b$ are:
 \begin{itemize}
  \item totally linked if and only $\idealpts(a)$ intersects exactly three of the connected components of $\Pbound\smallsetminus \idealpts(b)$,
  \item partially linked if and only if $\idealpts(a)$ intersects exactly two of the connected components of $\Pbound\smallsetminus \idealpts(b)$, and
 \item unlinked if and only if $\idealpts(a)$ is contained in a single connected component of $\Pbound\smallsetminus \idealpts(b)$.
 \end{itemize}
 
 For Affine planes, all points are totally linked with all others.  Since this case is special and will be treated separately, we make the following standing assumption: 

\putinbox{
 \begin{convention}
For the remainder of this section, $P$ denotes a \textbf{nontrivial} bifoliated plane, and $G$ a group with a (non-affine) Anosov-like action on $P$.   
 \end{convention}
}

The following lemma characterizes when a point is linked with some corner of a chain of lozenges.  Here and in what follows, an {\em endpoint of a chain} $\cC$ means any ideal point of a leaf of a corner of $\cC$ and the {\em endpoints of $\cC$}, denoted $\partial( \cC)$, is the union of all such points.   
 \begin{lemma}\label{lem_3_connected_components}
  Let $a \in P$, and let $\cC$ be a chain of lozenges. 
  $\cC$ has a corner $b$ that is totally linked with $a$ if and only $\partial( \cC)$ meets at least three connected components of $\Pbound\smallsetminus \idealpts(a)$.  
   \end{lemma}

 \begin{proof}
First suppose $b$ a corner in $\cC$ is totally linked with $a$. Then $\idealpts(b) \subset \partial( \cC)$, so $\partial( \cC)$ meets at least three connected components of $\Pbound\smallsetminus \idealpts(a)$ by the preceding characterization of totally linked points.

For the other direction, suppose that $\partial( \cC)$ meets at least three connected components of $\Pbound\smallsetminus \idealpts(a)$. Choose a half-leaf $r$ based at $a$ that intersects $\cC$, and let $l$ denote the leaf containing a side of lozenge in $\cC$ that intersects $r$ closest to $a$.  Let $b$ be the corner of $\cC$ on $r$. We claim that $b$ and $a$ are totally linked. To see this, let $l'$ be the other leaf through $b$. Suppose that $b$ and $a$ are not totally linked. Then $l'$ does not cross the other leaf through $a$. Since the ends of $\cC$ intersect at least three complementary components of $a$, there must be a lozenge of $\cC$ with one corner equal to $b$ and the other corner some point $b'$ that lies on the same quadrant of $b$ as $a$, as illustrated in Figure~\ref{fig_linked_lozenge}(a). The fact that positive and negative leaves cannot cross restricts which quadrant of $a$ can contain $b$, eliminating all but one option. The remaining quadrant is eliminated by the fact that $l$ was chosen closest to $a$ along $r$. See Figure~\ref{fig_linked_lozenge}(b). 
\end{proof} 

 \begin{figure}[h]  
   \labellist 
   \pinlabel (a) at 150 0
   \pinlabel (b) at 440 0
  \small\hair 2pt
     \pinlabel $a$ at 77 87 
    \pinlabel $b$ at 124 170
    \pinlabel $r$ at 120 70 
 \pinlabel $l$ at 142 120 
 \pinlabel $l'$ at 90 155 
 \pinlabel $b'$ at 350 55
   \endlabellist
     \centerline{ \mbox{
\includegraphics[scale=.5]{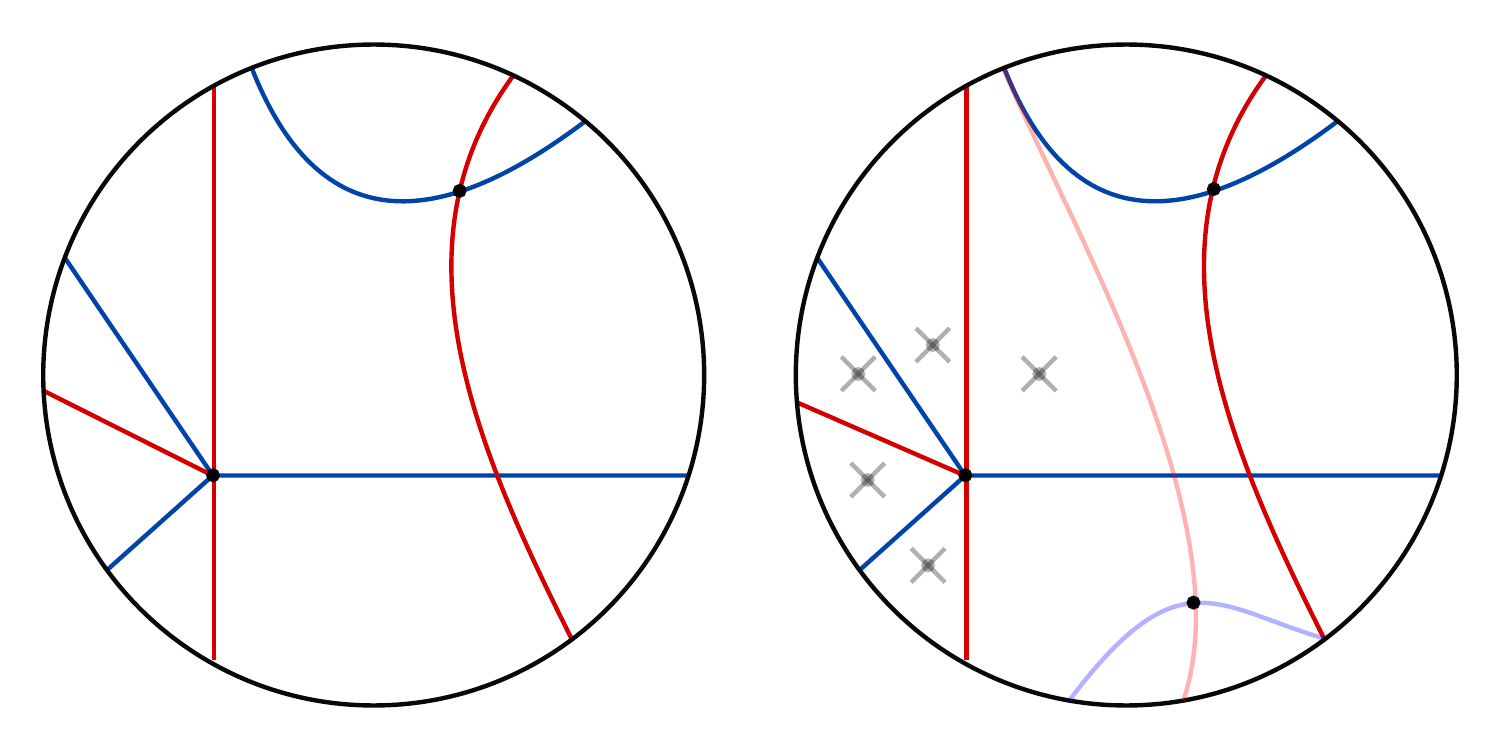} }}
\caption{ }
\label{fig_linked_lozenge}
\end{figure}

\begin{corollary} \label{cor:4_connected_comp}
Let $\cC$ be a chain of lozenges.  A point $a \in P$ is in the interior of a lozenge
of $\cC$ if and only if $\partial( \cC)$ is disjoint from $\idealpts(a)$ and meets four connected components of $\Pbound\smallsetminus \idealpts(a)$.
\end{corollary}
\begin{proof} 
Let $a \in P$.  By definition, $a$ is on a leaf forming the side of a lozenge in $\cC$ if and only if $\idealpts(a)$ contains some endpoint of $\cC$.  
So we assume this is not the case.  Suppose that the endpoints of $\cC$ meet four connected components of $\Pbound\smallsetminus \idealpts(a)$. 
Consider a corner $b$ of $\cC$ totally linked with $a$ as in Lemma \ref{lem_3_connected_components}, so $\idealpts(b)$ intersects three connected components $\Pbound\smallsetminus \idealpts(a)$, and $\cF^\pm(b)$ intersects two of the half-leaves of $\cF^\pm(a)$.   
Consider the quadrant of $b$ containing $a$. It must contain some lozenge in $\cC$, since $\partial( \cC)$ meets four connected components of $\Pbound\smallsetminus \idealpts(a)$.   Since $\cC$ is connected, $b$ must be the corner of a lozenge in this quadrant.  Such a lozenge necessarily contains the point $a$, as its other corner must lie in a quadrant of $a$ opposite to the quadrant containing $b$. 

The converse is immediate: if $a$ lies in a lozenge, then $a$ is nonsingular and the leaves of this lozenge meet all four connected components of $\Pbound\smallsetminus \idealpts(a)$.
\end{proof}

We now make a sign convention to distinguish between two kinds of fixed points. 
 
 \begin{definition}[Positive fixed points]
  Let $x$ be a fixed point of some nontrivial $g \in G$. We say that $x$ is a \emph{positive fixed point of $g$} if $g$ fixes each half-leaf through $x$ and is topologically expanding on $\cF^+(x)$ (hence topologically contracting on $\cF^-(x)$). It is a \emph{negative fixed point of $g$} if it is a positive fixed point for $g^{-1}$.
   \end{definition}

Note that if $g$ fixes both corners of a lozenge, then one corner of the lozenge is a positive fixed point for $g$ while the other is a negative fixed point.  While, in general, $g$ may fix a point $x$ but permute the half-leaves of $\cF^+(x)$, there is always a minimum $k>0$ such that $x$ is either a positive or negative fixed point of $g^k$.  

For the remainder of this section, we break with our usual convention of using $g$ to denote a group element, and instead use $\alpha$ and $\beta$ so as to give an easy mnemonic that $\alpha$ fixes point $a$, while $\beta$ fixes $b$ or $b'$, etc.  

 \begin{observation} \label{obs_easy_tot_linked}
If $\alpha$ and $\beta$ have totally linked positive fixed points, then $\beta^n\alpha^m$ has a positive fixed point for all $m,n>0$.  
 \end{observation} 
 
 \begin{proof}
 Suppose that $a$ and $b$ are totally linked positive fixed points of $\alpha$ and $\beta$, respectively. Let $I^- \subset \cF^{-}(b)$ be the $\cF^-$ leaf segment that runs from $b$ to $\cF^+(a)$. The set of positive leaves that intersect $I^-$ form a closed interval $I$ in the $\cF^+$ leaf space. For $n, m >0$, the map $\beta^n\alpha^m$ takes $I$ into itself, thus has a fixed leaf $l^+ \in I$. By Axiom \ref{Anosov_like_A1}, $\beta^n\alpha^m$ has a fixed point $x$ in $l^+$. Furthermore, since the action of $\beta^n\alpha^m$ on $I$ is by contraction, the action on $\cF^-(x)$ is by contraction, and hence $x$ is a positive fixed point of $\beta^n\alpha^m$. 
 \end{proof} 
 
 In Section \ref{sec:main_proof} we will make many arguments regarding sequences of non-corner points.  The following lemma allows us to produce some such sequences, and will be used crucially in Lemma \ref{lem:positive_corner}.

 \begin{lemma} \label{lem:make_NC_points}
  Suppose that $a$ and $b$ are totally linked positive fixed points of $\alpha$ and $\beta$, respectively, and assume $a$ is not a corner.  
  Then for all $m, n$ sufficiently large, $\beta^n \alpha^m$ has a unique non-corner fixed point in $P$.  Moreover, for fixed $m$, as $n \to \infty$ the $\cF^+$ leaf of this fixed point approaches $\cF^+(\beta)$.
 \end{lemma} 
 
 \begin{proof} 
 Let $a$, $b$ be as in the statement.
 Consider the action of $\beta$ on $\Pbound$ -- it either fixes the ideal points of two leaves, or fixes all ideal points of a chain $\cC_\beta$ of lozenges.  Abusing notation slightly, write $\partial(\cC_\beta) \subset \Pbound$ for the set of fixed points of $\beta$ on $\Pbound$, with the convention that $\cC_\beta = \cF^\pm(b)$ and $\partial(\cC_\beta) = \idealpts(b)$ if $\beta$ has a unique fixed point $b$ in $P$. 
 
Since $a$ is non-corner, $\alpha$ does not fix any corners, so by Observation \ref{obs:disjoint_ideal_closures} either 
$a = b$, or we have $\idealpts(a) \cap \overline{\partial \cC_\beta} = \emptyset$.  If $a=b$, we are already done, so we assume $\idealpts(a) \cap \overline{\partial \cC_\beta} = \emptyset$. 

Let $U^+_1$ and $U^+_2 \subset \Pbound$ be neighborhoods of the attracting fixed points of $\alpha$, disjoint from $\overline{\partial \cC_\beta}$. Because $a$ is non-corner, we may choose these neighborhoods to be bounded on either side by endpoints of leaves $l$ and $l'$ that pass through $\cF^-(a)$ very close to $a$.  
Similarly let $U^-_1$ and $U^-_2$ be neighborhoods of the repelling fixed points of $\alpha$, disjoint form $\overline{\partial \cC_\beta}$ corresponding to endpoints of leaves $f$ and $f'$ passing through $\cF^+(b)$ close to $b$.  See Figure \ref{new_fig} for an illustration. 
Choose $M$ large enough so that $\alpha^m(\Pbound \smallsetminus U^-) \subset U^+$ for all $m>M$, and fix any $m>M$ . 

The argument from Observation \ref{obs_easy_tot_linked} shows that $\beta^n\alpha^m$ has a fixed leaf of $\cF^+$ between $\cF^+(b)$ and $\cF^+(a)$.  Applying the same argument to its inverse $\alpha^{-n}\beta^{-m}$ shows that it also has a fixed leaf of $\cF^-$ between $\cF^-(a)$ and $\cF^-(b)$, and hence $\beta^n\alpha^m$ has a fixed point $x = x_{n,m}$ in the interior of the compact square of $P$ bounded by segments of $\cF^\pm(a)$ and $\cF^\pm(b)$. Moreover, all the half-leaves of that fixed point are also fixed by $\beta^n\alpha^m$.

We claim that, provided that $n$ is chosen sufficiently large, this point is the unique fixed point of $\beta^n\alpha^m$, which (since the half-leaves are preserved) implies that $x$ is indeed non-corner.
To show this claim, consider first a connected component $J$ of $\Pbound \setminus (U^-_1 \cup U^-_2)$.  By our choice of $m$, we have $\beta^m(J) \subset U_i^+$, and so $\beta^n\alpha^m(J) \subset \alpha^n(U_i^+)$, i.e. the only fixed endpoints of leaves in $J$ lie in $\beta^n(U_i^+)$.  
We require $n$ large enough so that both $\beta^n(l)$ and $\beta^n(l')$ will intersect $\cF^-(a)$ between $\cF^+(a)$ and $\cF^+(b)$.  

Now consider the action of $\alpha^{-m}\beta^{-n}$ on $U^-$.  For all $n$ sufficiently large, we will similarly have that 
$\beta^{-n}(f)$ and $\beta^{n}(f')$ intersect $\cF^{-}(a)$ between $\cF^+(a)$ and $\cF^+(b)$, and thus the same holds for $\alpha^{-m}\beta^{-n}(f)$ and $\alpha^{-m}\beta^{-n}(f')$.  Thus, any endpoint of a leaf in $U^-$ fixed by $\alpha^{-m}\beta^{-n}$ in fact lies between the pairs of consecutive endpoints of $\cF^{-}(a)$ and $\cF^{-}(b)$ on $\Pbound$.

Now recall that $x$ is a fixed point of $\beta^n\alpha^m$ in the square bounded by segments of $\cF^\pm(a)$ and $\cF^\pm(b)$. If it is not the unique fixed point then it is the corner of a lozenge fixed by $\beta^n\alpha^m$.  However, our restriction on fixed endpoints means that one of the other sides of this lozenge must cross from an interval of $\Pbound$ between consecutive points of $\cF^{-}(a)$ and $\cF^{-}(b)$ to an interval between consecutive points of $\cF^+(a)$ and $\cF^+(b)$, which is absurd. 

It remains only to remark that as $n \to \infty$, we have that $\cF^+(x_{n,m})$ approaches $\cF^+(b)$, but this follows from the fact that its positive endpoints lie in $\beta^n(U_i^+)$. 
 \end{proof}

  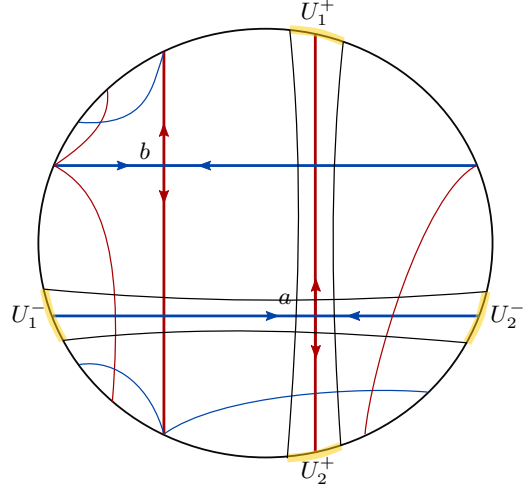
\begin{figure}  
   \labellist 
  \small\hair 2pt
     \pinlabel $b$ at 65 190
    \pinlabel $a$ at 150 100
    \pinlabel $U^+_1$ at 170 275 
  \pinlabel $U^+_2$ at 170 -5 
 \pinlabel $U^-_1$ at -5 90 
 \pinlabel $U^-_2$ at 285 90
   \endlabellist
     \centerline{ \mbox{
\includegraphics[width=6cm]{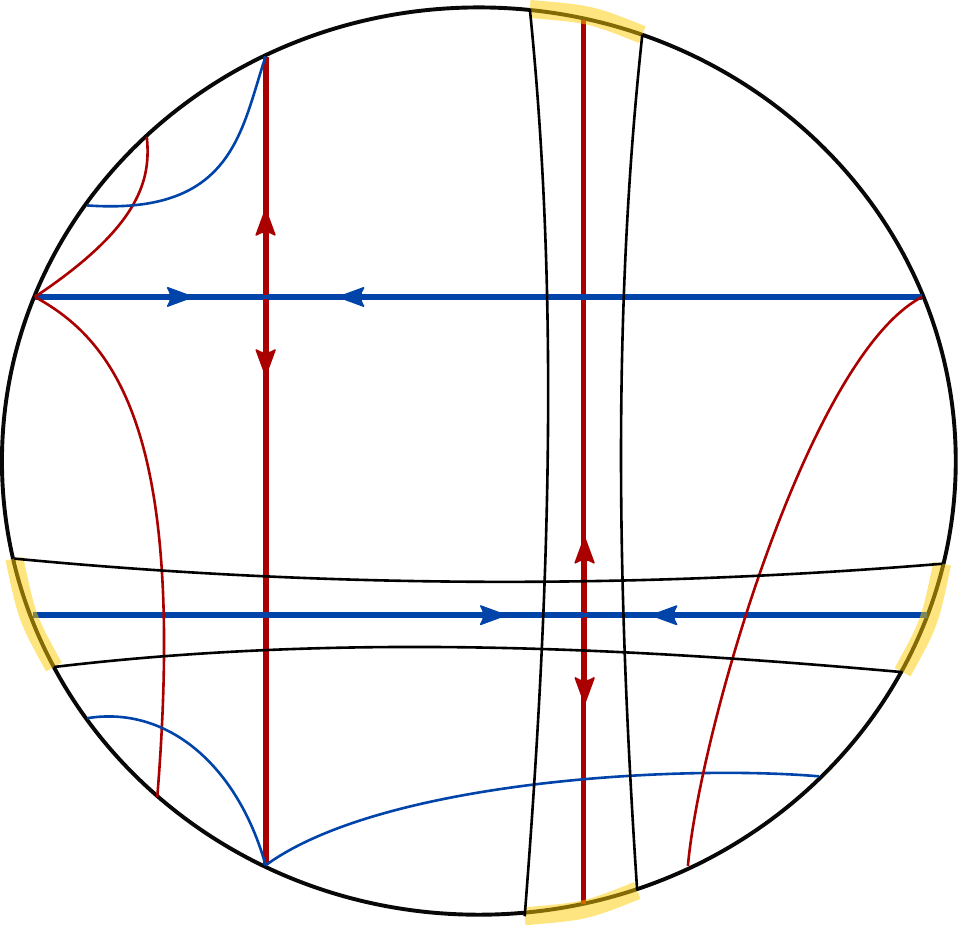} }}
\caption{Proof of Lemma \ref{lem:make_NC_points}. $\beta^n$ moves $U_i^+$ between the consecutive endpoints of $\cF^+$ (vertical) leaves of $a$ and $b$, and $\beta^{-n}$ moves $U_i^-$ between consecutive endpoints of $\cF^-$ leaves of $a$ and $b$.  As shown, $b$ is a corner, but the argument holds generally.}
\label{new_fig}
\end{figure}

 The following proposition gives a partial converse to Observation \ref{obs_easy_tot_linked}.
 
 \begin{proposition}[Characterization of totally linked points]\label{prop_charac_tot_linked}
 Let $\alpha, \beta \in G$ be elements that have either positive or negative fixed points (i.e. we assume only that they do not permute half-leaves through these fixed points), and assume $\alpha$ does not fix a corner.  
If there exist $n_i, m_i \to \infty$, such that $\alpha^{n_i} \beta^{m_i}$ has a fixed point for all $i$, then $\alpha$ and $\beta$ have fixed points that are totally linked. 
\end{proposition}

Note that the assumption that $\alpha$ does not fix a corner implies that $P$ is not skewed, and $\alpha$ has a unique fixed point.  

\begin{proof}
Suppose $a$ is a non-corner (positive or negative) fixed point of $\alpha$ and assume $\alpha^{m_i}\beta^{n_i}$ has a fixed point for some $m_i, n_i \to \infty$.  We will find a fixed point of $\beta$ that is totally linked with $a$. 

To first deal with a degenerate case, note that if $\alpha^j = \beta^k$ for some $j, k$ then $\beta$ also has $a$ as its unique fixed point.  Thus we will assume from now on that $\alpha^j \neq \beta^k$ for all $j, k$.

By assumption $\alpha$ fixed half-leaves through $a$, so pointwise fixes $\idealpts(a)$.  
As in the proof of Lemma \ref{lem:make_NC_points}, we abuse notation slightly and let $\overline{\partial(\cC_\beta)} \subset \Pbound$ denote the set of fixed points of $\beta$ on $\Pbound$. 
If $\idealpts(a) \cap \overline{\partial(\cC_\beta)} \neq \emptyset$, Observation \ref{obs:disjoint_ideal_closures} implies that $a$ is fixed by $\beta$ and we are done by the previous argument.

Thus, we need only analyze the situation where $\idealpts(a) \cap \overline{\partial(\cC_\beta)} = \emptyset$.   We will show that $\partial(\cC_\beta)$ meets at least three connected components of $\idealpts(a)$ by eliminating the other possibilities.  

\begin{case} Suppose that $\partial(\cC_\beta)$ (and therefore also its closure) is contained in a single complementary component of $\idealpts(a)$. Then we can partition the ideal circle $\Pbound$ into two closed intervals $I = [s, t], J = [t, s]$ that meet only at their endpoints in such a way that $\idealpts(a) \subset \mathring{I}$ and $\overline{\partial(\cC_\beta)} \subset \mathring{J}$. Let $a^+$ and $a^-$ be the points in $\idealpts(a)$ that are closest to $s$ and $t$, labeled so that $a^+$ (resp. $a^-$) is attracting (resp. repelling) for $\alpha$. Similarly, let $b^+$ and $b^-$ be the points in $\partial(\cC_\beta)$ that are closest to $s$ and $t$, labelled so that $b^+$ (resp. $b^-$) is attracting (repelling) for $\beta$. See Figure~\ref{fig:1component}. Note that $a^\pm$ may be reversed, as may $b^\pm$, but this does not affect the argument.  

Let $\gamma_i = \alpha^{m_i}\beta^{n_i}$. Observe that for $i$ sufficiently large $\beta^{n_i}(I)$ is contained in $J$, and lies arbitrarily close to $b^+$. After further increasing $i$, we also have $\alpha^{m_i}\beta^{n_i}(I)= \gamma_i(I) \subset I$, and will lie arbitrarily close to $a^+$. Similarly, for $i$ large $\gamma_i^{-1}(J)$ lies in $J$, arbitrarily close to $b^-$. It follows $\gamma_i$ has fixed points in both $\gamma_i(I)$ and $\gamma_i^{-1}(J)$, and that all of the fixed points of $\gamma_i$ are contained in these two intervals. By Corollary \ref{cor_fixed_points_are_joined}, (and since there are no fixed points of $\gamma_i$ outside of these intervals) it follows that $\gamma_i$ fixes some leaf $l_i$ that runs from $\gamma_i(I)$ to $\gamma_i^{-1}(J)$. Passing to a convergent subsequence, $l_i$ accumulates on a leaf $l_\infty$ with ideal endpoints $a^+$ and $b^-$. 
This again contradicts the fact, using Proposition \ref{prop_ideal_circle}, that there is no perfect fit at $a^+$.

 \begin{figure}
	\begin{overpic}{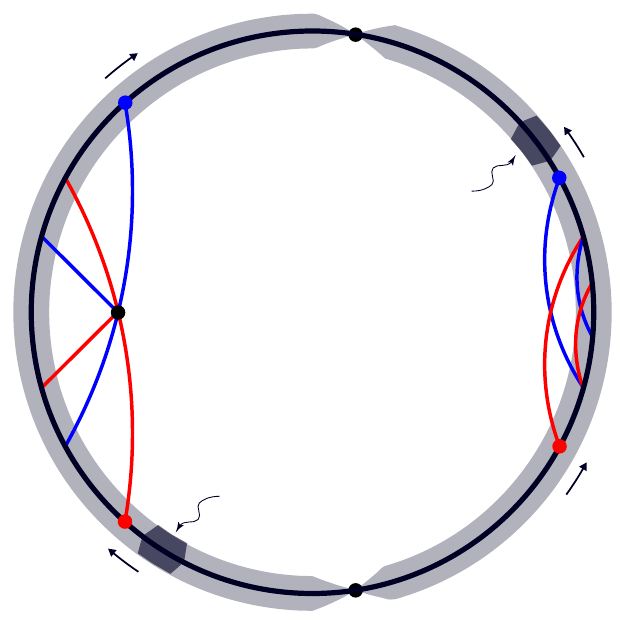} \small
		\put(56,1){$t$}
		\put(56,96){$s$}
		
		\put(27,82){$I = [s,t]$}
		\put(79,9){$J = [t,s]$}
		
		\put(21, 49){$a$}
		\put(14, 82){\textcolor{blue}{$a^-$}}
		\put(14, 12){\textcolor{red}{$a^+$}}
		
		\put(93, 70){\textcolor{blue}{$b^-$}}
		\put(93, 26){\textcolor{red}{$b^+$}}
		
		\put(35, 20){$\alpha^n\beta^n(I)$}
		\put(50, 68){$\beta^{-n}\alpha^{-n}(J)$}
	\end{overpic}
	\caption{Case 1}\label{fig:1component}
	
\end{figure} 
\end{case}

\begin{case}The case where $\partial(\cC_\beta)$ is contained in two connected components of $\Pbound \smallsetminus \idealpts(a)$ (and vice versa) is similar.  Figure \ref{fig:2components} illustrates all possible cyclic orderings (up to reversing the orientation of the circle) of $\partial(\cC_\beta)$ and points of $\idealpts(a)$, in the case where $a$ is a regular point.  The illustration for the case where $a$ is singular is easily obtained from this by inserting additional points of $\idealpts(a)$ in the intervals containing no points of $\partial(\cC_\beta)$, and the argument proceeds in exactly the same way.
   The colored intervals in the figure represent intervals, possibly degenerate (i.e. equal to a single point), containing points of $\partial(\cC_\beta)$, and a $\oplus$ or $\ominus$ sign indicates if the endpoint of the region is an attracting (respectively, repelling) fixed point for the action of $\beta$ on $\Pbound$.  

The first configuration cannot occur because a leaf of $\cC_\beta$ would need to cross both a leaf of $\cF^+(a)$ and of $\cF^-(a)$.  
For the other two, again using the fact that, for any set $U \subset \Pbound$ whose closure contains no repelling fixed points of $\alpha$ we have that   $\alpha^n(U)$ converges to an attracting fixed points of $\alpha$ as $n \to \infty$ (and the analogous statements for $\alpha^{-1}$ and $\beta^\pm$), we can argue as in the previous case that for $n, m$ sufficiently large, all fixed points of $\alpha^{m_i}\beta^{n_i}$ lie between the adjacent pairs of attracting and repelling points of $\alpha$ and $\beta$ on the boundary, as shaded in grey in the figure.   Thus, as in the previous argument we find leaves $l_{i}$ with endpoints approaching an attracting ideal point for $\alpha$ and a repelling ideal point for $\beta$ and can derive a contradiction as above.  

{\em Remark:} The reader may find it enlightening to attempt this style of argument when $\partial(\cC_\beta)$ meets three connected components of $\Pbound \smallsetminus \idealpts(a)$, in this case the cyclic orderings of fixed points allows leaves $l_i$ to approach a leaf {\em through} $a$, thus, giving no contradiction.    

 \begin{figure}[h] 
	\labellist 
	\small\hair 2pt
	\pinlabel $\partial(\cC_\beta)$ at 20 30 
	\pinlabel $a$ at 95 95
	\pinlabel $\oplus$ at 89 170 
	\pinlabel $\oplus$ at 89 6 
	\pinlabel $\ominus$ at 7 89
	\pinlabel $\ominus$ at 170 89 
	\pinlabel $l_i$ at 300 118 
	\pinlabel $a$ at 276 95
	\pinlabel $\oplus$ at 270 170 
	\pinlabel $\oplus$ at 270 6 
	\pinlabel $\ominus$ at 189 89
	\pinlabel $\ominus$ at 351 89 
	\pinlabel $\oplus$ at 308 161
	\pinlabel $\ominus$ at 343 128
	\pinlabel $\ominus$ at 308 16
	\pinlabel $\oplus$ at 343 50
	\pinlabel $l_i$ at 480 118 
	\pinlabel $a$ at 459 95
	\pinlabel $\oplus$ at 452 170 
	\pinlabel $\oplus$ at 452 6 
	\pinlabel $\ominus$ at 372 89
	\pinlabel $\ominus$ at 534 89 
	\pinlabel $\oplus$ at 490 161
	\pinlabel $\oplus$ at 525 128
	\pinlabel $\ominus$ at 490 16
	\pinlabel $\ominus$ at 525 50
	\endlabellist
	\centerline{ \mbox{
			\includegraphics[width=11cm]{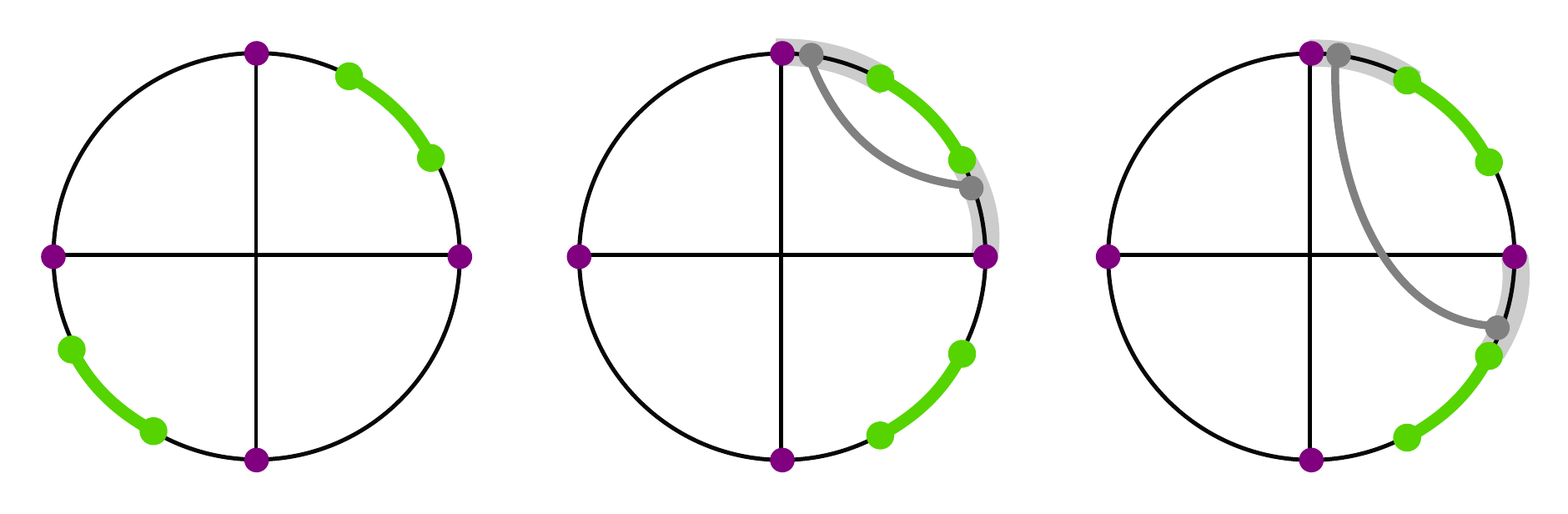} }}
	\caption{Configurations of fixed sets}
	\label{fig:2components}
\end{figure}
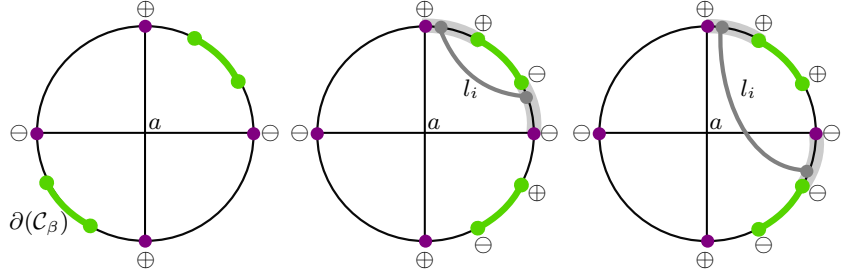
\end{case} 

\noindent \textit{Conclusion.} 
We conclude that $\partial(\cC_\beta)$ meets at least three connected components of $\idealpts(a)$. Thus, by Lemma \ref{lem_3_connected_components}, we conclude that $\beta$ has a fixed point that is totally linked with $a$, proving the proposition.  
\end{proof} 

The next proposition builds on the previous by taking positive and negative signs into account.  

\begin{proposition}[Characterization of totally linked positive points] \label{prop_charac_tot_linked_pos}
Suppose $a \neq b$ are totally linked, positive fixed points of $\alpha, \beta$, and assume $\alpha$ fixes no corner.  
Either 
\begin{enumerate}
\item $\beta$ has a negative fixed point totally linked with $a$ (in which case $\alpha^m \beta^{-n}$ has a fixed point for all $m, n>0$), \textbf{or} 
\item for all $n, m$ sufficiently large $\alpha^m \beta^{-n}$ has no fixed point in $P$.  
\end{enumerate} 
\end{proposition}  

\begin{proof} 
As in the proof of the previous proposition, we use $\partial(\cC_\beta)$ to denote the endpoints of leaves fixed by $\beta$, either endpoints of a chain of lozenges or the leaves through a single fixed point.   If $\overline{\partial(\cC_\beta)} \cap \cF^\pm(a) \neq \emptyset$, as observed at the start of the proof of Proposition \ref{prop_charac_tot_linked}, we have $a$ fixed by $\beta$, which is impossible.  If $\overline{\partial(\cC_\beta)}$ (and thus $\partial(\cC_\beta)$) meets all four connected components of  $\Pbound \smallsetminus \idealpts(a)$, then by Corollary \ref{cor:4_connected_comp} $a$ lies in a lozenge fixed by $\beta$, so $\beta$ has a negative fixed point totally linked with $a$, and the first condition holds. 
If $\Pbound \smallsetminus \overline{\partial(\cC_\beta)}$ meets one or two connected components of $\cF^\pm(a)$, then we can apply case 1 or 2 of the proof of  Proposition \ref{prop_charac_tot_linked} to $\alpha$ and $\beta^{-1}$ and conclude $\alpha^m \beta^{-n}$ has no fixed point in $P$.  

Thus, we need to treat the case where $\partial(\cC_\beta)$ meets exactly three connected components of $\Pbound \setminus \idealpts(a)$. 
This is illustrated in Figure \ref{fig:3components}, on the left where $\beta$ does not fix a lozenge, and on the right an example of a chain fixed by $\beta$.  As before, we have drawn $a$ as a nonsingular point; if singular, $\cF^\pm(a)$ would simply have additional half-leaves in the upper left quadrant, which does not affect the argument. The dynamics of $\alpha$ and $\beta^{-1}$ on the boundary are shown with attracting points labeled $\oplus$ and repelling labeled $\ominus$. 
 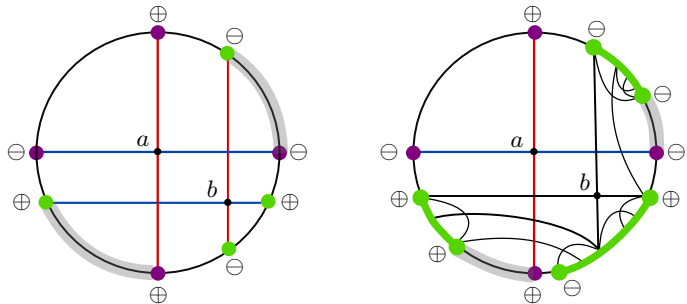
\begin{figure}[h]
   \labellist 
  \small\hair 2pt
\pinlabel $a$ at 80 96
\pinlabel $b$ at 120 68
    \pinlabel $\oplus$ at 89 170 
        \pinlabel $\oplus$ at 89 6 
    \pinlabel $\ominus$ at 7 89
        \pinlabel $\ominus$ at 170 89 
     \pinlabel $\ominus$ at 133 157 
        \pinlabel $\ominus$ at 133 22 
    \pinlabel $\oplus$ at 10 60
        \pinlabel $\oplus$ at 165 60     
\pinlabel $a$ at 296 95
\pinlabel $b$ at 335 72
    \pinlabel $\oplus$ at 305 170 
        \pinlabel $\oplus$ at 305 6 
    \pinlabel $\ominus$ at 222 89
        \pinlabel $\ominus$ at 388 89 
 \pinlabel $\ominus$ at 342 161
        \pinlabel $\ominus$ at 380 123
 \pinlabel $\ominus$ at 328 10
        \pinlabel $\oplus$ at 385 62
         \pinlabel $\oplus$ at 228 62 
         \pinlabel $\oplus$ at 250 30
 \endlabellist
     \centerline{ \mbox{
\includegraphics[width=9cm]{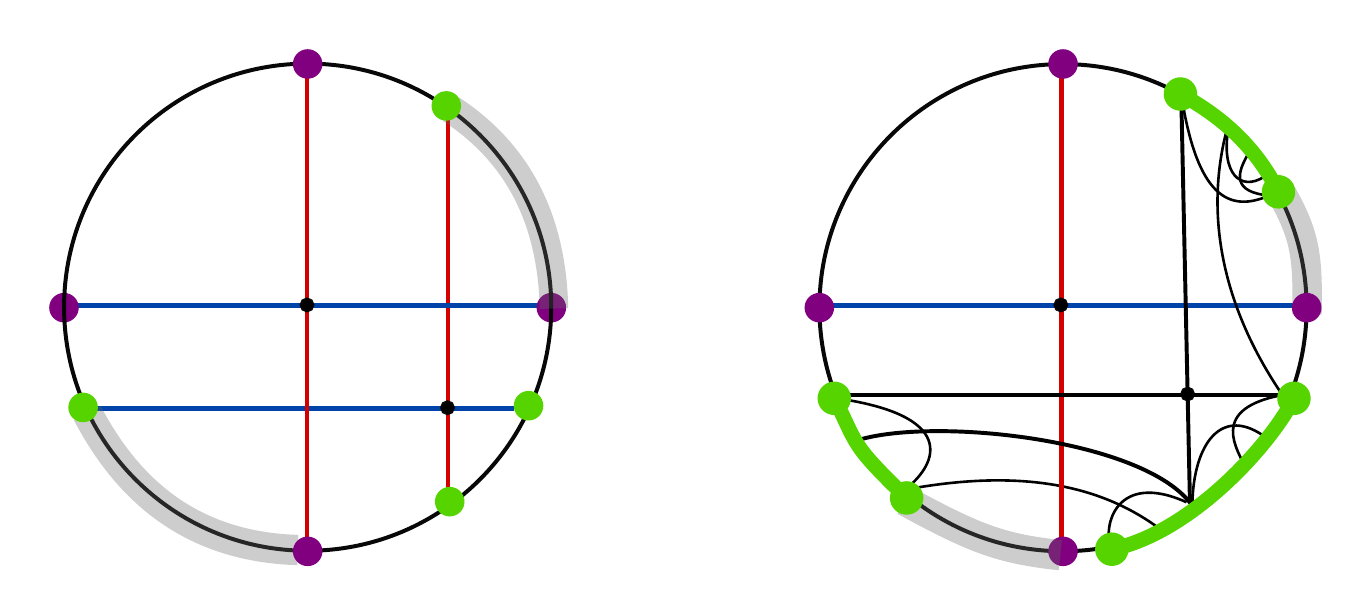} }}
\caption{Configurations of fixed sets for the action of $\alpha$ and $\beta^{-1}$}
\label{fig:3components} 
\end{figure}

In the non-lozenge case, for all $m,n > 0$, $\alpha^m\beta^{-n}$ contracts the interval between the adjacent attracting fixed points of $\alpha$ and $\beta^{-1}$ (and similarly for $(\alpha^m\beta^{-n})^{-1}$ with the interval between adjacent repelling fixed points), and for $m,n$ sufficiently large has no fixed points elsewhere.  Thus, if $\alpha^m\beta^{-n}$ had a fixed point in $P$, then
some leaf preserved by $\alpha^m\beta^{-n}$ must cross both $\cF^+(a)$ and of $\cF^-(a)$, a contradiction as in the previous Proposition.  

Now consider the case where $\beta$ preserves a chain of lozenges.  Let $b$ denote a corner that is totally linked with $a$. Up to simultaneously switching all $\oplus$ and $\ominus$ signs, the configuration of fixed points for $\alpha$ and $\beta^{-1}$ is illustrated in Figure \ref{fig:3components} on the right.  
Since no negative fixed point of $\beta$ is totally linked with $a$, the chain of lozenges stays in the bottom right three quadrants of $b$.  This forces $\beta^{-1}$ to have an attracting fixed point directly to the left of the lower attracting fixed point of $\alpha$ and thus a repelling fixed point directly to the right. (Indeed, $\beta^{-1}$ acts on $\cF^+(a)$ by pushing it to the left). Similarly,  $\beta^{-1}$ has a repelling fixed point on $\Pbound$ immediately above the repelling fixed point of $\alpha$ on the right, forcing the configuration shown. 
Again, we can conclude that all fixed points of $\alpha^m\beta^{-n}$ lie in the union of the region between the adjacent $\oplus$ points and the adjacent $\ominus$ points, forcing any leaf preserved by $\alpha^m\beta^{-n}$ to cross both $\cF^+(a)$ and of $\cF^-(a)$, and thus $\alpha^m\beta^{-n}$ cannot have a fixed point or fixed leaf.  
\end{proof}

We will also use the following variation on Proposition \ref{prop_charac_tot_linked}, which allows both $\alpha$ and $\beta$ to fix corners.  As before, we use the notation $\partial(\cC_\alpha)$ to denote the endpoints of a chain of lozenges fixed by $\alpha$.  

 \begin{proposition}[Linked chains of lozenges]\label{prop:linked_loz}
Let $\alpha, \beta \in G$ be elements such that $\alpha$ and $\beta$ have corner fixed points $a$ and $b$ in $P$, and assume that the points of $\overline{\partial(\cC_\alpha)}$ lie in exactly two connected components of $\Pbound  \smallsetminus \overline{\partial(\cC_\beta)}$.  
Then for all $n, m$ sufficiently large, $\alpha^n\beta^m$ and  $\alpha^{-n}\beta^m$ have no fixed points in $P$.  
 \end{proposition} 

\begin{proof} 
Our assumption means that exactly four connected componoents of $\Pbound - (\overline{\partial(\cC_\alpha)} \cup \overline{\partial(\cC_\beta)})$ have one endpoint fixed by $\alpha$ and one by $\beta$.  Recall as before these endpoints lie in $\partial(\cC_\alpha)$ and $\partial(\cC_\beta)$ respectively.  
Label these intervals $I_1, \ldots I_4$ in cyclic order. The possible configurations of attracting and repelling endpoints of the intervals $I_i$ in $\partial(\cC_\alpha)$ and $\partial(\cC_\beta)$ (up to symmetry) are depicted in Figure \ref{fig:configurations}.

 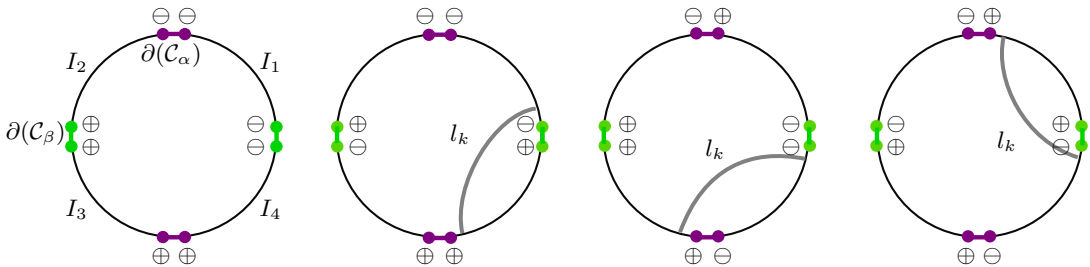
\begin{figure}[h]
	\labellist 
	\small\hair 2pt
	\pinlabel $\partial(\cC_\alpha)$ at 87 145 
	 \pinlabel $\partial(\cC_\beta)$ at -6 91 
	\pinlabel $I_1$ at 155 140 
	 \pinlabel $I_2$ at 22 140 
	\pinlabel $I_3$ at 22 38 
	 \pinlabel $I_4$ at 155 38 
	\pinlabel $\ominus$ at 80 172
	\pinlabel $\ominus$ at 98 172 
	\pinlabel $\oplus$ at 80 6
	\pinlabel $\oplus$ at 98 6 
	\pinlabel $\oplus$ at 32 81
	\pinlabel $\oplus$ at 32 97
	\pinlabel $\ominus$ at 145 81
	\pinlabel $\ominus$ at 145 97 
	\pinlabel $l_k$ at 285 90 
         \pinlabel $\ominus$ at 261 172
	\pinlabel $\ominus$ at 281 172 
	\pinlabel $\oplus$ at 261 6
	\pinlabel $\oplus$ at 281 6 
	\pinlabel $\ominus$ at 215 81
	\pinlabel $\oplus$ at 215 97
	\pinlabel $\oplus$ at 330 81
	\pinlabel $\ominus$ at 330 97
	\pinlabel $l_k$ at 460 80 
         \pinlabel $\ominus$ at 445 172
	\pinlabel $\oplus$ at 465 172 
	\pinlabel $\oplus$ at 445 6
	\pinlabel $\ominus$ at 465 6 
	\pinlabel $\oplus$ at 400 81
	\pinlabel $\oplus$ at 400 97
	\pinlabel $\ominus$ at 512 81
	\pinlabel $\ominus$ at 512 97
	\pinlabel $l_k$ at 660 85 
         \pinlabel $\ominus$ at 632 172
	\pinlabel $\oplus$ at 650 172 
	\pinlabel $\oplus$ at 632 6
	\pinlabel $\ominus$ at 650 6 
	\pinlabel $\oplus$ at 584 81
	\pinlabel $\ominus$ at 584 97
	\pinlabel $\ominus$ at 697 81
	\pinlabel $\oplus$ at 697 97
	\endlabellist
	\centerline{ \mbox{
			\includegraphics[width=14cm]{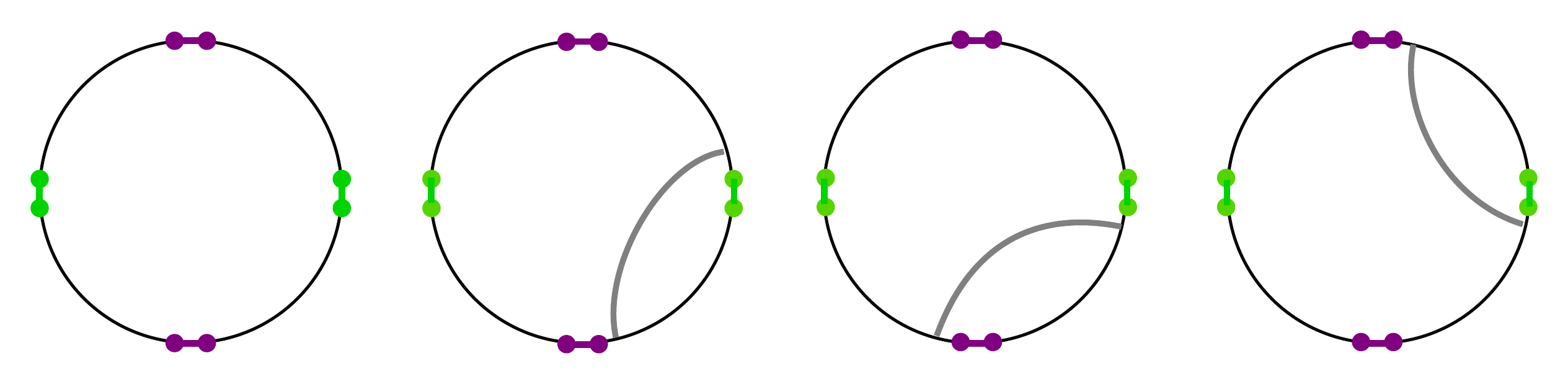} }}
	\caption{Configurations of fixed sets for linked lozenges: points of $\partial(\cC_\alpha)$ (resp. $\partial(\cC_\beta)$) lie in the small intervals between the adjacent purple (resp. green) points.}
	\label{fig:configurations}
\end{figure}

The proof follows the same arguments as in case 1 of the proof of Proposition \ref{prop_charac_tot_linked}, so we will be brief.  
In the leftmost case of Figure \ref{fig:configurations}, for $n, m$ large and positive, $\alpha^n\beta^m$ will have fixed points on the boundary in intervals $I_1$ and $I_3$ but no others. If $\alpha^n\beta^m$ had an interior fixed point, this would contradict Corollary \ref{cor_fixed_points_are_joined}. For $n$ negative, $m$ positive, the same argument applies with intervals $I_2$ and $I_4$. 

The remaining three cases are dealt with again exactly as in Proposition \ref{prop_charac_tot_linked}.  Supposing for some sequences $n_k$ and $m_k$ tending to $+\infty$, $\alpha^{n_k} \beta^{m_k}$ had a fixed point.  Then we would find a sequence of fixed leaves $l_k$ with endpoints approaching an attracting fixed point of $\alpha$ and repelling fixed point of $\beta$, in the positions indicated.  Then $l_k$ would accumulate on a leaf fixed by both $\alpha$ and $\beta$, giving a contradiction.   The case where $n_k$ is negative is symmetric.  
\end{proof}

\subsection{Determining skewedness and non-corners} 
As another important preliminary step for Theorem \ref{thm_main_general}, we now use the tools developed above to show that $\fix(\rho)$ distinguishes skewed from non-skewed planes.   We also show that if two actions have the same elements with fixed points, then they also have the same elements with fixed points that are noncorners.  

\begin{lemma}  \label{lem_skew_linked}
Suppose $(P,\cF^+,\cF^-)$ is a nontrivial  bifoliated plane with Anosov-like action of a group $G$.  If the plane is not skewed, 
then there exist a pair of non-corner points fixed by elements of $G$ that are partially but not totally linked.  
\end{lemma} 

\begin{proof} 
Suppose $P$ is not skewed.  By Theorem \ref{thm_trivial_skewed_or_nowheredense}, there exists some fixed point $c \in P$ which is not a corner and is non-singular.  
By Lemma \ref{lem:periodic_pf_corner}, no half-leaf through $c$ makes a perfect fit with any other leaf.  
Pick a non-singular leaf $l \in \cF^+$ that intersects $\cF^-(c)$ at a point $z$.
Proposition \ref{prop:no_product} implies that there are no infinite product regions, thus we can find some non-singular leaf $l_1^-\in \cF^-$ such that $l_1^-$ intersects $\cF^+(c)$ but not $l$.    
Let $y= l_1^-\cap \cF^+(c)$.  Then $y$ and $z$ are partially linked, and the fact that $c$ has no perfect fits means that 
the leaves of $y$ and $z$ have no common endpoints. 
Since all the leaves involved here are non-singular, this will remain true for any points sufficiently close to $y$ and $z$, i.e. any pairs of sufficiently nearby points are not totally linked. Since corners are nowhere dense  (Theorem \ref{thm_trivial_skewed_or_nowheredense}) and fixed points are dense, we can pick $a$ near $y$ and $b$ near $z$ that are fixed by some elements of the action, and such that $a$ and $b$ are partially, not totally, linked.   
\end{proof}

\begin{corollary}\label{cor_skewed_implies_skewed}
 Let $G$ be a group with Anosov-like actions $\rho_i$ on nontrivial bifoliated planes $(P_i,\cF^+_i,\cF^-_i)$, $i=1,2$  and suppose that $\fix(\rho_1)=\fix(\rho_2)$.    The plane $(P_1,\cF^+_1,\cF^-_1)$ is skewed if and only if $(P_2,\cF^+_2,\cF^-_2)$ is skewed.  
\end{corollary}

\begin{proof}
By symmetry, we need only prove one direction, so we assume $P_1$ is not skewed and show that this implies $P_2$ is not skewed.   
By Lemma \ref{lem_skew_linked}, there are non-corner fixed points $a$ and $b$ in $P_1$ that are not totally linked.  Let  $\alpha$ and $\beta$ be elements having $a$ and $b$, respectively, as their (unique) positive fixed points.  By the (contrapositive to) Proposition \ref{prop_charac_tot_linked}, for all sufficiently large $m, n>0$, both $\rho_1(\alpha^{\pm m}\beta^{\pm n})$ and $\rho_1(\alpha^{\pm m}\beta^{\mp n})$
have no fixed point.  

Since $\fix(\rho_1)=\fix(\rho_2)$, we have also that $\rho_2(\alpha)$ and $\rho_2(\beta)$ have fixed points, but $\rho_2(\alpha^{\pm m}\beta^{\pm n})$ and $\rho_2(\alpha^{\pm m}\beta^{\mp n})$ do not, for $m, n$ sufficiently large.   
Let $a'$ be a fixed point for $\rho_2(\alpha)$ and let $k \in \bZ$ be such that $a'$ is a {\em positive} fixed point for $\alpha^k$.  
Now if $P_2$ were skewed, then (again up to passing to powers) $\beta$ would have either a positive or negative (possibly both) fixed point totally linked with $a'$, which means that for $n, m$ large, either $\rho_2(\alpha^{km}\beta^{n})$ or $\rho_2(\alpha^{km}\beta^{-n})$ would have a fixed point by Observation \ref{obs_easy_tot_linked}.  Since this is not the case, we conclude $P_2$ is not skewed.  
\end{proof}

\begin{definition}
 Given an Anosov-like action $\rho$ of $G$ on a bifoliated plane, define $\fixnc(\rho)$ to be the set of elements $g\in G$ such that $\rho(g)$ fixes a (necessarily unique) non-corner point.  
\end{definition}

Note that it is possible for an element to have a unique corner fixed point, for instance if it rotates an invariant chain of lozenges about a corner.  However, such an element will have a nontrivial power that fixes more than one point in $P$, whereas all powers of elements in $\fixnc(\rho)$ have unique fixed points in $P$.  

The next proposition states that if two Anosov-like actions have the same elements with fixed points, then they have the same elements with non-corner fixed points. This property will allow us to start building a homeomorphism between two bifoliated planes with the same sets of elements with fixed points. 
\begin{proposition}\label{prop_non-corners_are_equal}
 Let $\rho_1$ and $\rho_2$ be Anosov-like actions on nontrivial bifoliated planes $(P_1,\cF^+_1,\cF^-_1)$ and $(P_2,\cF^+_2,\cF^-_2)$ respectively. If $\fix(\rho_1)=\fix(\rho_2)$, then $\fixnc(\rho_1)=\fixnc(\rho_2)$.
\end{proposition}

\begin{proof}
If one of the bifoliated planes is skewed, then there are no non-corner points and the result follows from Corollary \ref{cor_skewed_implies_skewed}.
So we assume that the bifoliated planes are not skewed.

 Let $\alpha \in \fixnc(\rho_1)$ and suppose for contradiction that $\alpha \notin \fixnc(\rho_2)$. Then, up to replacing $\alpha$ by a power of $\alpha$, there is some lozenge $L$ in $P_2$ whose two corners are fixed by $\rho_2(\alpha)$; call these corners $a_2^\pm$.
By Theorem \ref{thm_trivial_skewed_or_nowheredense} and Axiom \ref{Anosov_like_dense_fixed_points}, we can find a non-corner point $b_2$ inside $L$ that is a positive fixed point of $\rho_2(\beta)$ for some $\beta \in G$.
 Then $b_2$ is totally linked with $a_2^+$ and $a_2^-$. By Observation \ref{obs_easy_tot_linked}, the elements 
 $\rho_2(\beta^n\alpha^{n})$, and $\rho_2(\beta^n\alpha^{-n})$
have positive fixed points for all $n>0$, and hence their inverses $\rho_2(\alpha^n \beta^{-n})$ and $\rho_2(\alpha^{-n} \beta^{-n})$ also have fixed points.  
Thus, $\rho_1(\alpha^n \beta^{-n})$ and $\rho_1(\alpha^{-n} \beta^{-n})$ have fixed points as well (which may be negative or positive).  

Since $\rho_1(\alpha)$ has no corner fixed points, Proposition \ref{prop_charac_tot_linked} implies that $\rho_1(\alpha)$ and $\rho_1(\beta)$ have totally linked fixed points, say $a$ and $b$.  
For simplicity, we relabel $\cF^\pm_1$ if needed so that $b$ is a  positive fixed point of $\rho_1(\beta)$.  
Our first goal is to show that $a$ lies inside a lozenge fixed by $\rho_1(\beta)$.  

 Consider first the case where $a$ is a positive fixed point of $\rho_1(\alpha)$.  Then by Proposition \ref{prop_charac_tot_linked_pos}, either $\rho_1(\beta)$ has a negative fixed point $b'$ totally linked with $a$, or $\rho_1(\alpha^{n}\beta^{-n})$ has no fixed points for all $n$ large; but the latter option contradicts our observation above.  We conclude 
$\rho_1(\beta)$ has a negative fixed point $b'$ totally linked with $a$.  Since $b$ is also totally linked with $a$, it follows that $b$ and $b'$ are two corners of a lozenge. 
In the case where $a$ is a negative fixed point of $\rho_1(\alpha)$, it is a positive fixed point for $\rho_1(\alpha^{-1})$ and the same argument using $\rho_1(\alpha^{-n} \beta^{-n})$ in the place of $\rho_1(\alpha^{n}\beta^{-n})$ shows that $a$ is inside of a lozenge, with corner $b$ and another corner we may denote by $b'$.  
In either case, we let  $L'$ in $P_1$ denote the lozenge with corners $b, b'$ fixed by $\rho_1(\beta)$ with $a \in L'$.

We now consider the actions of conjugates of $\beta$ by powers of $\alpha$.   
Let $\gamma_n = \alpha^n \beta \alpha^{-n}$.    Since $\rho_2(\alpha)$ fixes the corners $a_2^+$ and $a_2^-$ of $L$, it fixes $L$, thus, for all $n\in \bZ$, $\rho_2(\alpha^n)b_2 \in L$.  In particular, $\rho_2(\alpha^n)b_2$ is totally linked with $b_2$.  
Since $\rho_2(\alpha^n)b_2$ is a positive fixed point of $\rho_2(\gamma_n)$ totally linked with $b_2$,  we conclude that $\rho_2(\gamma_n^m \beta^m)$ has a (positive) fixed point in $P_2$ for all $m, n>0$.  

Since  $\fix(\rho_1)=\fix(\rho_2)$, we have that $\rho_1(\gamma_n^m \beta^m)$ has a fixed point in $P_1$.  But, by choosing $n$ sufficiently large, the ideal fixed points of the conjugate $\rho_1(\gamma_n) = \rho_1(\alpha^n \beta \alpha^{-n})$ in $\Pbound_1$ can be taken to lie as close as we like to the attracting fixed points of $\rho_1(\alpha)$ on $\Pbound_1$.  In particular, they will lie in exactly two connected components of the complement of the set of ideal points fixed by $\rho_1(\beta)$.   Thus, by Proposition \ref{prop:linked_loz}, for $m$ large, $\rho_1(\gamma_n^m \beta^m)$, has no fixed points in $P_1$, a contradiction.  
\end{proof}

\section{Proof of Theorem \ref{thm_main_general}} \label{sec:main_proof} 

We are now ready for the proof of spectral rigidity for Anosov-like actions.  We first treat the skewed case, which is essentially covered by the work in \cite{BMB}.  In Section \ref{sec:tool_nonskew} we establish some general tools to use for the non-skewed case, then complete the proof in Section \ref{sec:nonskewed_proof}.

\subsection{Proof of Theorem \ref{thm_main_general} in the skew case.} 
Let $\rho_i$, $i=1,2$ be Anosov-like actions of a group $G$ on nontrivial bifoliated planes $(P_i,\cF^\pm_i)$.  We suppose that $\fix(\rho_1)=\fix(\rho_2)$, and we wish to produce a conjugacy between these actions.   
Suppose first that one of the bifoliated planes is skewed.  By Corollary \ref{cor_skewed_implies_skewed}, this means that the other plane is also skewed.  Now we can apply the main result in \cite{BMB}.  While stated for Anosov flows, the proof holds without modification for Anosov-like actions, as it only relies on the Anosov-like property of the action of $\pi_1(M)$ on the orbit space, there framed in terms of the ``hyperbolic-like" action on the leaf spaces of $\cF^\pm_i$.   We conclude that the two actions are conjugate.

\subsection{General toolkit for the non-skewed case.} \label{sec:tool_nonskew}
It remains to treat the case where one (and hence both) of the planes is non-skewed.  The general strategy is as follows:  by Theorem~\ref{thm_trivial_skewed_or_nowheredense}, the sets $\fixnc(\rho_1)$ and $\fixnc(\rho_2)$ of elements with non-corner fixed points are non-empty; by Proposition \ref{prop_non-corners_are_equal}, they are equal.  Thus to each non-corner point $x \in P_1$ fixed by some $\rho_1(g)$, we may associate the unique point of $P_2$ fixed by $\rho_2(g)$.  We wish to show that this map (defined on a dense subset) extends to a homeomorphism from $P_1$ to $P_2$.    
This is quite technical, so we first establish some preliminary tools, allowing us to detect lozenges and also detect convergence of certain sequences using fixed point data and the tools from Section \ref{sec:linking}.  

Since the material in this subsection is generally applicable to Anosov-like actions, we suppress $\rho$ from the notation as in the earlier sections.

\putinbox{
\begin{convention} 
We assume in this subsection that $(P, \cF^\pm)$ is a nontrivial, non-skewed, bifoliated plane with an Anosov-like action of a group $G$.  
\end{convention} 
}

\begin{observation}\label{obs_sequence_of_totally_linked}
Suppose $(x_n)$ is a sequence in $P$ that converges to $x$, and $y$ is totally linked with each point $x_n$.  Then $\cF^+(x)$ either 
 \begin{enumerate}
  \item  intersects $\cF^-(y)$, 
  \item makes a perfect fit with $\cF^-(y)$, or 
  \item is non-separated with a leaf $l^+$ that intersects or makes a perfect fit with $\cF^-(y)$.
 \end{enumerate}
 The same holds reversing the role of $+$ and $-$.  
\end{observation} 

\begin{proof} 
Suppose $x_n \to x$ and $y$ is totally linked with each $x_n$ i.e. $\cF^\pm(y)\cap \cF^\mp(x_n)\neq \emptyset$.  
The sequence of leaves $\cF^+(x_n)$ converges either to a unique leaf (which must be $\cF^+(x)$) or to a union of non-separated leaves containing $\cF^+(x)$. Since $\cF^-(y)$ intersects each $\cF^+(x_n)$, it either makes a perfect fit with or intersects some leaf in the limit. 
The same holds reversing the roles of $\cF^+$ and $\cF^-$. 
\end{proof}

We will frequently use sequences of pairwise totally linked points in our construction of maps between planes.  The following lemma lets us essentially ignore singular points in such arguments. 

\begin{lemma}[Pairwise linked points are eventually nonsingular]   \label{lem:sing_not_linked}
Suppose $x_1, \ldots, x_5$ are five distinct singular points in $P$.  Then there exist $i, j$ such that $x_i$ is not totally linked with $x_j$. 
\end{lemma} 

\begin{proof} 
We will prove the equivalent statement that, given any five pairwise totally linked points $x_1, \ldots, x_5$, at least one is non-singular.  

As a first step, we observe the following: for any three distinct, pairwise totally linked points $x_i, x_j, x_k$ (singular or not), after possibly re-indexing, we have that $\cF^+(x_j)$ separates $\cF^+(x_{i})$ from $\cF^+(x_{k})$.   To see this, note that if the conclusion is not immediately satisfied then $x_i$ and $x_k$ lie in the same connected component of $P \smallsetminus \cF^+(x_{j})$.  There is only one half-leaf, say $r$, of $\cF^-(x_{j})$ in this component, so $r$ must meet both $\cF^+(x_{i})$ and $\cF^+(x_{k})$ because all points are pairwise totally linked.  Whichever $\cF^+$ leaf is met first separates the other from $\cF^+(x_{j})$.

Now suppose $x_1, x_2, \ldots x_5$ are five distinct totally linked points in $P$.   
Note that the ``betweenness" relation $B$ defined by $(x_i, x_j, x_k) \in B$ if $\cF^+(x_j)$ separates $\cF^+(x_{i})$ from $\cF^+(x_{k})$ satisfies the usual axioms for a betweenness relation and contains some permutation of every triple, so we may re-index the points so that $\cF^+(x_j)$ separates $\cF^+(x_{i})$ from $\cF^+(x_{k})$ whenever $i < j< k$.  See Figure \ref{fig:separates} for a suggestive illustration, with $\cF^+$ leaves in red.  

\begin{figure}[h]
   \labellist 
  \small\hair 2pt
     \pinlabel $\cF^+(x_1)$ at 10 40
     \pinlabel $\cF^+(x_2)$ at 52 15 
     \pinlabel $\cF^+(x_3)$ at 110 5
 \pinlabel $\cF^+(x_4)$ at 165 15
 \pinlabel $\cF^+(x_5)$ at 210 40     
 \pinlabel $\cF^-(x_p)$ at -15 132
  \pinlabel $x_m$ at 70 150
 \pinlabel $x_p$ at 140 118
   \pinlabel $r_p$ at 175 183
  \pinlabel $x_m$ at 440 122
 \pinlabel $x_p$ at 395 92
 \endlabellist
     \centerline{ \mbox{
\includegraphics[width=12cm]{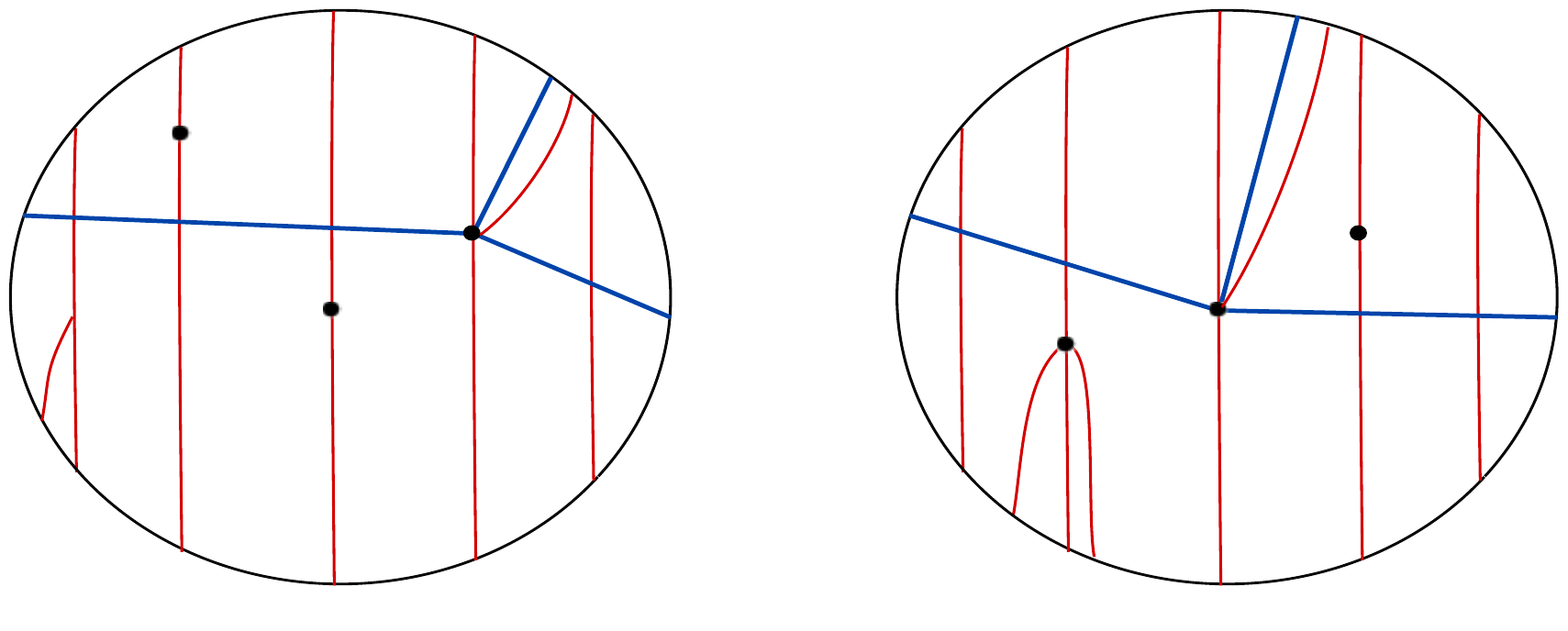} }}
\caption{Two possible configurations in the proof of Lemma \ref{lem:sing_not_linked}, with $x_p=x_4$ left, and $x_p=x_3$ right. In each, $x_m$ cannot be totally linked with $x_1$.}
\label{fig:separates}
\end{figure}
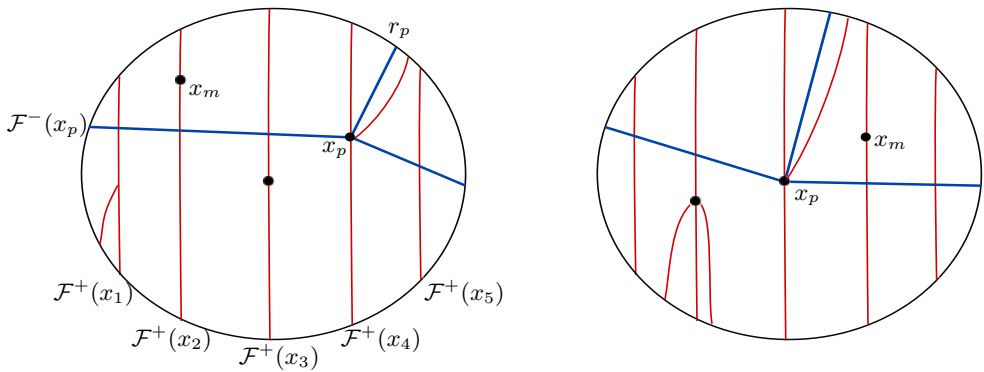

We will show at least one of $x_2, x_3$ or $x_4$ is nonsingular.  First note that our initial observation also holds 
 with $\cF^+$ replaced by $\cF^-$, so for some $p \in \{2, 3, 4\}$, we have that $\cF^-(x_p)$ separates the unstable leaves of the other points.  
Assume for contradiction that this $x_p$ is singular.   Since $x_p$ is totally linked with all the other points, there is some line (i.e. a union of two half-leaves) $l_p$ in $\cF^-(x_p)$ that intersects each leaf $\cF^+(x_j)$ and the other 
 half-leaves of $\cF^-(x_p)$ cannot intersect any $\cF^+(x_j)$.  Fix some such half leaf $r_p$ not intersecting any $\cF^+(x_j)$. Now, in the connected component of $\Pbound_1\smallsetminus \partial l_p$ that contains the ideal endpoint $\partial r_p$ of $r_p$, we notice that $\partial r_p$ separates $\partial\cF^+(x_1)$ and $\partial \cF^+(x_5)$.    Let $x_m \in \{x_2, x_3, x_4\}$ be the point on the same side of $l_p$ as $r_p$, which exists by our choice of $p$.  Then 
 $\cF^-(x_p)$ separates $x_m$ either from $x_1$ or $x_5$,  and thus $x_m$ cannot be totally linked with both $x_1$ or $x_5$.   
\end{proof}

\begin{lemma}[Behavior of pairwise totally linked points]\label{lem:TL_limit_trichotomy}
 Let $(z_n)$ be a sequence of pairwise totally linked points in $P$. Suppose the $z_n$ are not on the same $\cF^{\pm}$-leaf.
 Then $(z_n)$ admits a subsequence that either:
 \begin{enumerate}[label=(\roman*)]
  \item \label{item_accumulate_corner_scalloped} lies in a scalloped region $U$ and converges to an ideal corner of $U$,
  \item \label{item_accumulate_perfect_fit} converges to the ideal point of a perfect fit made by leaves $l^+, l^-$ and all elements of the subsequence satisfy $\cF^{\pm}(z_{n_k}) \cap l^{\mp}\neq \emptyset$, 
or
  \item \label{item_accumulate_point} converges to a point $z \in P$, and each $z_n$ in the subsequence is totally linked with $z$.   \end{enumerate}
\end{lemma}

Convergence here means in the topological closed disc $P \cup \Pbound$. 
\begin{rem}
 Note that if the $z_n$ are on the same leaf $l$, then we are either in case \ref{item_accumulate_point} or a subsequence converges to an ideal point of $l$.
\end{rem}

\begin{proof} 
Using the remark above, we may pass to a subsequence so all the $(z_n)$ are on distinct leaves.  By Lemma \ref{lem:sing_not_linked}, we may further assume all $(z_n)$ are nonsingular. 
Pass to a further subsequence so that $(z_n)$ converges in $P \cup \Pbound$.  If the limit point is in $P$, we may drop the first few terms to ensure that all the remaining ones are totally linked with the limit point, and we are done. So we assume the limit is on the boundary. 
 Let $\xi_n^+,\zeta_n^+$ and $\xi_n^-,\zeta_n^-$  be the endpoints of $\cF^+(z_n)$ and $\cF^-(z_n)$ respectively.  Pass again to a subsequence so  $(\xi_n^+,\zeta_n^+,\xi_n^-,\zeta_n^-)$  converges to some 4-tuple $(\xi^+,\zeta^+,\xi^-,\zeta^-)$.  We first show at least three of the points of this 4-tuple are distinct, as follows: since each point $z_n$ is totally linked with $z_1$, if three elements of the limiting 4-tuple agree, all three must agree with an endpoint of a leaf of $z_1$.  But each point $z_n$ is also totally linked with $z_2$, which has distinct endpoints from $z_1$, giving an immediate contradiction.  Thus, $\{\xi^+,\zeta^+,\xi^-,\zeta^-\}$ contains at least 3 points.  
 
 Now, since the $z_n$ are pairwise totally linked, the cyclic order of their endpoints on the boundary remains constant. Reading anti-clockwise we will see in order either 
 $\xi_n^+, \xi^-_n,\zeta_n^+,\zeta_n^-$ or the reverse of this, with the same ordering for all $n$.   Thus, $\xi^+ \neq \zeta^+$, and $\xi^- \neq \zeta^-$.  By Lemma~\ref{lem_convergence_or_trivially_product}, this means that $\cF^+(z_n)$ converges to a union, finite or infinite, of leaves.  (The same argument applies to the sequence $\cF^-(z_n)$ as well.)
 
Suppose first that  $\cF^+(z_n)$ converges to an infinite union of leaves.  Then by Lemma~\ref{lemma_scalloped_ideal_corner}, this infinite union forms one side of the boundary of a scalloped region.    
If the points $z_n$ do not eventually lie inside the scalloped region, the fact that $\cF^+(z_n)$ converges to a side of the region forces {\em both} endpoints of $\cF^-(z_n)$ to converge to a corner of the scalloped region, which contradicts our observation above that three of the limiting endpoints cannot coincide.  Thus, after dropping the first terms of the sequence, we may assume that all $z_n$ lie inside the scalloped region.  The scalloped region is tiled in two directions by bi-infinite families of lozenges indexed by $\bZ$.  Since $(z_n)$ converges to a point in $\Pbound$, the limit point is either a corner of the scalloped region as in case (i), or we may pass to a further subsequence such that all $(z_n)$ lie in a single lozenge, and converge to one of its corners, which puts us in case (ii). 

The same argument applies if $\cF^-(z_n)$ converges to an infinite union of leaves.  Thus, we assume that both converge to finite unions.  Let $l_1, l_2,\dots, l_k$ be the leaves in the limit of $\cF^+(z_n)$.  Since $(z_n)$ converges to a point, say $\eta$, on $\Pbound$, this limit point $\eta$ must be an endpoint of some leaf $l^+ \in \{l_1, \ldots l_k\}$.   Similarly for $\cF^-(z_n)$.  Thus, one of the limiting leaves of $\cF^-(z_n)$, say $l^-$, makes a perfect fit with $l^+$ at $\eta$. 
To finish the proof, we just need to argue that (after a further subsequence) all $z_n$ have leaves intersecting both $l^-$ and $l^+$. To see this, it suffices to show that   
$z_n$ all lie in the connected component of $P \smallsetminus (l^+ \cup l^-)$ bounded by both leaves, since the fact that $l^-$ and $l^+$ make a perfect fit means that any point $z_n$ with leaves sufficiently close to $l^+$ and $l^-$ will intersect both.  This fact is a direct consequence of the property that all $z_n$ are pairwise totally linked: if $z_k$ were in the component bounded only by $l^-$, say, then $\cF^{-}(z_k)$ shares no endpoint with $l^+$, and hence for $n$ sufficiently large we would have $\cF^{+}(z_n) \cap \cF^{-}(z_k) = \emptyset$.  
 \end{proof}

Our next goal is to characterize when a sequence converges to a point that shares a leaf with a non-corner fixed point.   We begin with the following observation.  Though the condition in the statement appears technical, it is exactly capturing the behavior shown in Figure~\ref{fig:powersg}.

\begin{observation} \label{obs:powers_g}
Suppose $g \in G$ has a non-corner, positive fixed point $w \in P$.  
Assume $(z_n)$ is a sequence of pairwise totally linked points, $z_n\notin \cF^-(w)$, and $(z_n)$ converges to a point $z \in \cF^-(w)$, $z\neq w$, then the following holds:
\begin{enumerate}[label=$(\star)$]
 \item\label{item_star} 
For each sufficiently large $i$, there exists $K$ such that:  if $k>K$, then $g^{-k}(z_i)$ is not totally linked with $z_j$ for $j\leq i$ and there exists $N$ such that $g^{-k}(z_i)$ is totally linked with $z_n$ for all $n>N$.
\end{enumerate}
 \end{observation}  
 
 See Figure \ref{fig:powersg} for an illustration of this condition.  

  \begin{figure}[h]
   \labellist 
  \small\hair 2pt
     \pinlabel $w$ at 20 20 
          \pinlabel $\cF^+(w)$ at 15 135 
    \pinlabel $\cF^-(w)$ at -15 27  
    \pinlabel $z$ at 55 20 
 \pinlabel $g^{-k}(z)$ at 270 20
\pinlabel $z_j$ at 75 75
 \pinlabel $g^{-k}(z_i)$ at 225 60
 \endlabellist
     \centerline{ \mbox{
\includegraphics[width=10cm]{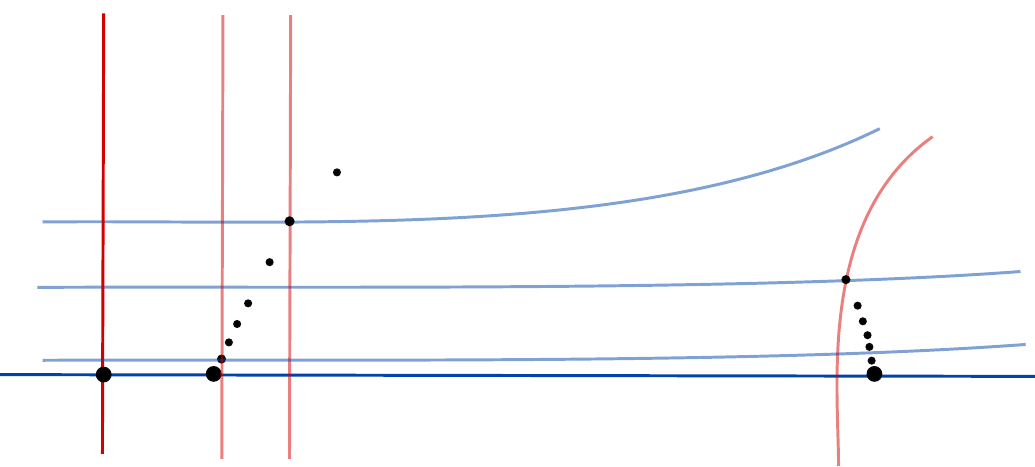} }}
\caption{The idea of Observation \ref{obs:powers_g}: $g^{-k}(z_i)$ will be totally linked with $z_n$ sufficiently close to $z$.  If $g^{-k}(z_i)$ were totally linked with $z_i$ for all $k$ (rather than partially linked as shown here), then we would have an infinite product region, as opposed to the behavior shown here.}
\label{fig:powersg} 
\end{figure}
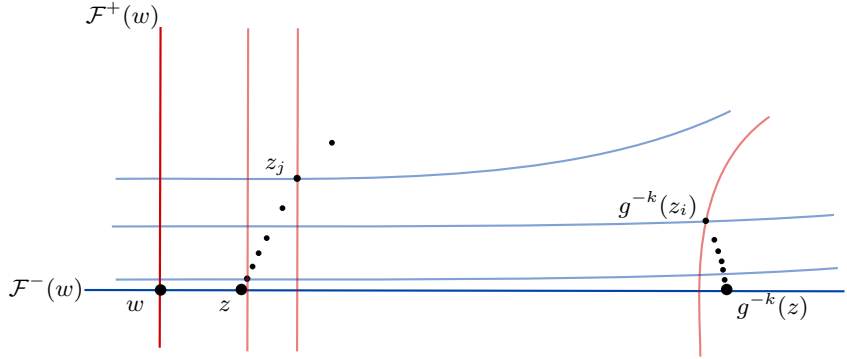

 \begin{proof}[Proof of Observation \ref{obs:powers_g}]
 Since $w$ is a positive fixed point of $g$, we have that, for any sufficiently large $i$, as $k \to \infty$ the sequence $g^{-k}(z_i)$ converges (in $P\cup \Pbound$) to the endpoint $\xi$ of the ray $r^-$ of $\cF^-(w)$ containing $z$. Since $w$ is non-corner, this ray does not make a perfect fit.  
Fix such $i$ and fix $j\leq i$.  Let $R$ denote the intersection of the quadrant of $z_j$ that meets the side of $\cF^-(w)$ with endpoint $\xi$, and the quadrant of $w$ containing $z_j$.  It is bounded by a bounded segment of $\cF^+(z_j)$, and infinite segments of $\cF^-(w)$ and 
$\cF^-(z_j)$.  

Consider the set 
$S = \{ y \in  r^- : \cF^+(y) \cap \cF^-(z_j) \neq \emptyset \}$.  Since there are no infinite product regions, and $w$ is non-corner, $S$ cannot be unbounded.
Consequently, any $\cF^+$ leaf through $r^-$ sufficiently close to $\xi$ cannot meet $\cF^-(z_j)$, and thus for all $k$ sufficiently large $g^{-k}(z_i)$ is not totally linked with $z_j$.  
 
For the second part of the statement, note that for sufficiently large $i$, we have that $z_i$ is totally linked with $z$ and for any fixed $k$, $g^{-k}(z_i)$ is also totally linked with $z$. Thus, there exists $N$ such that $g^{-k}(z_i)$ is totally linked with $z_n$ for all $n>N$.
 \end{proof}

 The following lemma says the condition above cannot happen if $z_n$ converges to a perfect fit or a corner of a scalloped region.  
 
\begin{lemma} \label{lem:new_simpler}
Suppose $(z_n)$ is a sequence of pairwise totally linked points that satisfies condition \ref{item_accumulate_corner_scalloped} or \ref{item_accumulate_perfect_fit} of Lemma \ref{lem:TL_limit_trichotomy} 
and suppose $w$ is a non-corner positive fixed point of $g$ that is totally linked with all $z_n$.   
Then \ref{item_star} is not satisfied by $g$ and $(z_n)$.  
\end{lemma}

\begin{proof} 
Suppose first that $(z_n)$ lies inside and converges to an ideal corner of a scalloped region $U$.  Since $w$ is totally linked with each $z_n$, we must have that $w \in U$.  But then for any given $i$, for all large enough $K$, $g^{K}(z_i)$ and $g^{-K}(z_i)$ will be outside of the scalloped region. 
In particular, $g^{\pm K}(z_i)$ cannot be totally linked with $z_n$ when $n$ is large, so condition \ref{item_star} fails.  

As a second case, suppose instead that $(z_n)$ lies inside a lozenge $L$ and converges to an ideal point $\eta$ of $L$, where leaves $l^+$ and $l^-$ make a perfect fit.  
Since $w$ is a non-corner, it does not lie on $l^\pm$. We deal separately with the different possible positions of $w$ with respect to $l^+$ and $l^-$.

Suppose first that $l^-$ separates $w$ from the $z_i$, thus $\cF^-(w)$ does not intersect $l^+$. 
As $w$ is totally linked with all $z_i$, the leaves $\cF^+(z_i)$ all intersect $\cF^-(w)$. Since $\cF^-(w)$ does not intersect $l^+$, the leaves $\cF^+(z_i)$ must converge to a union $\{l^+_j\}$ of non-separated leaves containing $l^+$, one of which intersects $\cF^-(w)$.  
Let $l_w^+ \in \{l^+_j\}$ be the leaf that intersects $\cF^-(w)$.  

If $l^+_w$ separated $w$ from $z_i$ then we would have $\cF^+(w) \cap \cF^-(z_i) = \emptyset$, contradicting that these points are totally linked. Thus, $z_i$ and $l^+_w$ are on the same side of $\cF^+(w)$,
as in Figure~\ref{fig:converge_pf}, left.  Since $w$ is a positive fixed point for $g$, we also have that $l_w^+$ separates all $z_n$ from $g^{-k}(z_i)$, so \ref{item_star} fails.

The case where $l^+$ separates $w$ from the $z_i$ is very similar: here $\cF^-(z_i)$ converges to a union $\{l^-_j\}$ of non-separated leaves containing $l^-$, with some leaf $l_w^- \neq l^-$ intersecting $\cF^+(w)$.  As before, $w$ cannot be separated from the $z_i$ by $l_w^-$, as it would contradict their total linking.
Then, for some large $i$ fixed, if the first part of  \ref{item_star} were to hold, we would have $z_i$ not totally linked with $g^{-k}(z_i)$, for $k$ sufficiently large.  But this implies that $\cF^+(g^{-k}(z_i)) \cap l^- = \emptyset$, and therefore $g^{-k}(z_i)$ is also not totally linked with $z_n$ for $n$ large, as shown in Figure \ref{fig:converge_pf} right.

  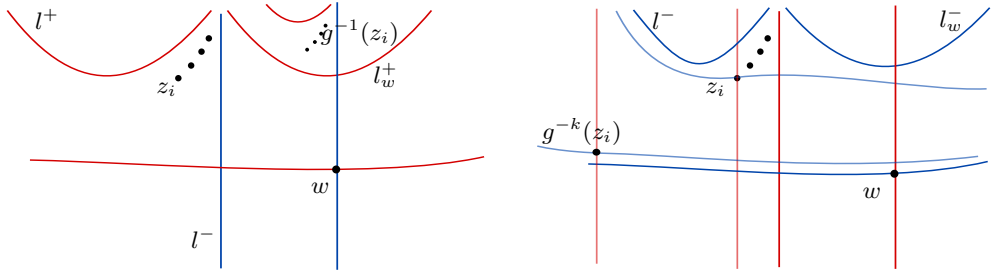
\begin{figure}[h]
   \labellist 
  \small\hair 2pt
    \pinlabel $w$ at 111 30 
   \pinlabel $z_i$ at 57 65 
    \pinlabel $l^+$ at 15 90 
   \pinlabel $l^-$ at 70 12    
   \pinlabel $l^+_w$ at 134 70
   \pinlabel $g^{-1}(z_i)$ at 125 85 
     \pinlabel $w$ at 305 28 
   \pinlabel $l^-$ at 232 90      
\pinlabel $l^-_w$ at 333 90
   \pinlabel $z_i$ at 250 65
      \pinlabel $g^{-k}(z_i)$ at 203 50     
 \endlabellist
     \centerline{ \mbox{
\includegraphics[width=13cm]{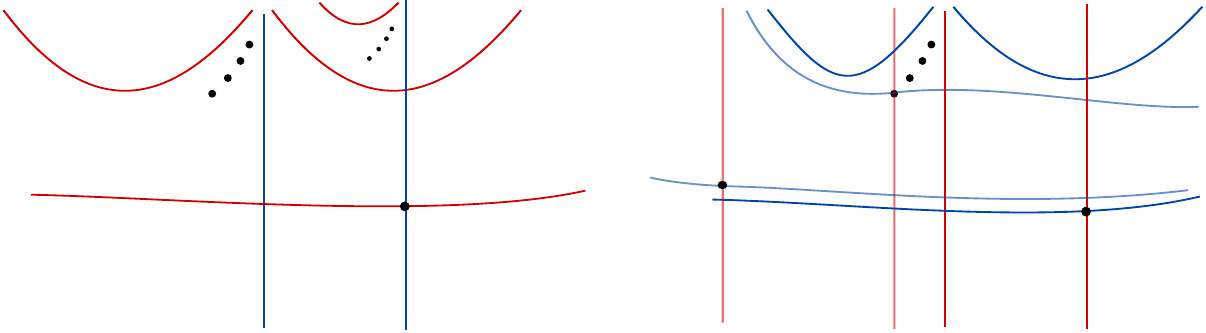} }}
\caption{When $z_i$ converges to the ideal point of a perfect fit, \ref{item_star} fails.}
\label{fig:converge_pf} 
\end{figure}

Thus, 
 we may now without loss of generality pass to a subsequence so that $l^+ \cup l^-$ does not separate $w$ from any $z_i$. 

As $w$ is a positive non-corner fixed point for $g$ and $P$ is nontrivial, both endpoints of $g^n(l^-)$ approach an endpoint of $\cF^-(w)$ as $n \to \infty$.  Thus, for $n$ sufficiently large, $g^n(l^-)$ will not intersect or make a perfect fit with any leaf nonseparated with $l^+$.  Fix such $n$.  Since $\cF^-(z_i) \to l^-$, for $i$ sufficiently large and $K>n$, we will have that $g^{K}(\cF^-(z_i))$ is separated from the union of leaves nonseparated with $l^+$ by $g^n(l^-)$.  Thus, for large $i$, we will have $g^{K}(\cF^+(z_i)) \cap \cF^+(g_i) = \emptyset$, and again condition \ref{item_star} fails.
\end{proof} 

\begin{corollary}[Convergence to non-corner leaves]  \label{cor:detect_non-corner_leaves}
Let $(z_n)$ be a sequence of pairwise totally linked points converging to some $z \in P \cup \Pbound$, and let $w$ be a non-corner positive fixed point of $g$ totally linked with all $z_n$.   Then \ref{item_star} holds if and only if $z \in \cF^-(w)$.  
\end{corollary}  
Note that by convention we do {\em not} consider the ideal boundary points of $\cF^-(w)$ to lie in $\cF^-(w)$.

\begin{proof} 
The reverse direction is Observation \ref{obs:powers_g}.  For the forward direction, assume $(z_n)$ converges to $z \in P \cup \Pbound$ and \ref{item_star} holds.  
Passing to a subsequence, we may assume that $(z_n)$ satisfies one of the conditions of Lemma \ref{lem:TL_limit_trichotomy}.  Note that \ref{item_star} still holds for the subsequence. 
By Lemma \ref{lem:new_simpler}, we have $z \in P$.  First we eliminate the possibility that $z \in \cF^+(w) \smallsetminus \{w\}$. If this is the case, then, for large enough $k$, $g^{-k}(z_j)$ will \emph{not} be totally linked with any $z_n$ for $n$ large, so condition \ref{item_star} fails.   Now, for the remaining case, note that if $z \notin \cF^\pm(w)$, then we may assume without loss of generality (by passing to a further subsequence) that $z_n \notin \cF^\pm(w)$ for all $n$.  Thus, in particular, $g^{-k}(z_2)$ will diverge to infinity as $k \to \infty$, approaching the endpoint of a fixed leaf of $g$.  Following the same argument as we used in the proof of Observation \ref{obs:powers_g}, since $w$ is a non-corner and there are no infinite product regions, $g^{-k}(z_2)$ cannot remain totally linked with $z_1$ for arbitrarily large $k$, so again the condition fails.  
\end{proof} 

\begin{rem}  \label{rem:pos_neg_star}
As an immediate consequence (allowing $g^{-1}$ to play the role of $g$ above), we also have the following: \\ 
Let $(z_n)$ be a sequence of pairwise totally linked points converging to some $z \in P \cup \Pbound$, and suppose is $w$ a non-corner {\em negative} fixed point of $g$ that is totally linked with all $(z_n)$.   Then \ref{item_star} holds if and only if $z \in \cF^+(w)$.  
\end{rem}

Our next goal is to give a characterization of lozenges and adjacent lozenges in terms of linking properties and the action.  This will occupy the remainder of this section.  
For convenience we introduce some terminology and notation.  

\begin{definition}[Extended side]
An {\em extended side} of a lozenge $L$ is a leaf containing a side of $L$.  The {\em extended sides} of a chain $C$ is the set of all extended sides of lozenges in $C$. 
 \end{definition}

\begin{notation} 
We call a lozenge whose corners are fixed by some nontrivial $g \in G$ a {\em $g$-lozenge}, and we abbreviate ``$x$ and $y$ are totally linked" by $x \TL y$. 
\end{notation} 

\begin{lemma}[Characterization of interior of lozenges]\label{lem_charac_in_lozenge}
 Let $g \neq \id$ be and assume that $g$ fixes each half-leaf through some point.  
\begin{enumerate}[label=(\roman*)]
\item If $x \in P$ lies in the interior of a $g$-lozenge, then $g^{n}x\TL g^{m}x$ for all $n,m\in \bZ$.
\item If $x$ is neither in a $g$-lozenge nor on a leaf fixed by $g$, then there exists a neighborhood $U$ of $x$ and an integer $N$ such that, for all $n\in \bZ$ with $|n|>N$, $g^{n}x$ is not totally linked with any point of $U$.
\end{enumerate}
\end{lemma}
\begin{rem}
One could additionally characterize points on $g$-invariant leaves, but the situation is more complicated: there are three possible behaviors for the linking of $g^{n}x$ with points in neighborhoods of $x$, depending on whether the invariant half-leaf containing $x$ is a side of $0$, $1$, or $2$ lozenges. Since we do not need such a precise statement for our main result, we do not pursue this here.
\end{rem}

\begin{proof}[Proof of Lemma \ref{lem_charac_in_lozenge}]
Note that if $x$ lies in the interior of a $g$-lozenge, then the first condition (and, vacuously, also the second) is immediately satisfied.
So from now on, we suppose that $x$ is neither in a $g$-lozenge nor on a leaf fixed by $g$.

Let $c$ be a fixed point of $g$.
Then, either $c$ is a non-corner fixed point or, $c$ is the corner of some lozenge in a maximal chain of lozenges $\cC$;  by Lemma \ref{lem_power_fixes_max_chain} $g$ preserves $\cC$ and fixes all of its corners.  To streamline the proof, we will slightly abuse notation here and say $\cC = c$ if $c$ is non-corner, the ``extended sides" of $\cC$ being then just the leaves of $c$.

By Proposition \ref{prop:boundary_action_general}, the set of fixed points of $g$ on $\Pbound$ is equal to the closure 
$\overline{\partial( \cC)}$ of endpoints of extended sides of $\cC$.  
Since we assumed that $x$ is not on any extended side of $\cC$, nor in a $g$-lozenge, Lemma \ref{lem_leaves_to_ideal_are_fixed} implies that $\partial \cF(x) \cap \overline{\partial (\cC)} = \emptyset$.  
By Corollary \ref{cor:4_connected_comp}, $\partial \cF(x)$ meets at most three distinct connected components of $\Pbound \smallsetminus \overline{\partial (\cC)}$.

Since $\partial \cF(x)$ has at least four points, by the pigeonhole principle, there exists a connected component $I$ of $\Pbound \smallsetminus \overline{\partial (\cC)}$ that contains at least two points of $\partial \cF(x)$.   Enumerate the points of $\partial \cF(x) \cap I$ in cyclic order $\eta_1, \ldots, \eta_k$.  

After possibly replacing $g$ with $g^{-1}$, for $n$ sufficiently large we will see the following points in $I$, in cyclic order 
\[ \eta_1, \ldots, \eta_k, g^n(\eta_1), \ldots, g^n(\eta_k) \]
and the point $g^n(x)$ lies in a quadrant of $x$ bounded by the half-leaf $r_1$ ending in $\eta_1$.  
See Figure \ref{fig:pointsonbound2} for an illustration.  

Note additionally that no other ideal point of $\cF^\pm(g^n (x))$ lies in $I$, since the interval $I$ is $g$-invariant.  
But this means that $g^n(x)$ and $x$ cannot be totally linked, for if they were, some leaf of $g^n(x)$ would need to intersect $r_1$.  Since it cannot also intersect the ray $r_2$ ending at $\eta_2$, its endpoint would lie between $\eta_1$ and $\eta_2$ in $I$, a contradiction.   
\end{proof}
  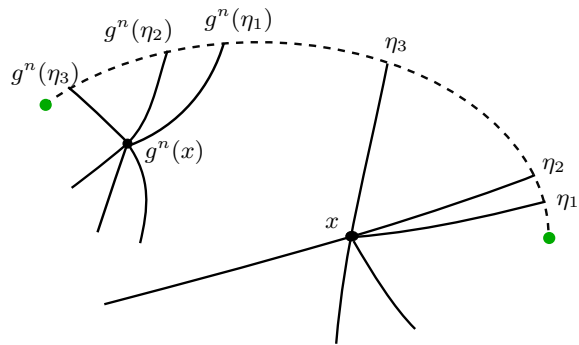
\begin{figure}[h]
   \labellist 
  \small\hair 2pt
   \pinlabel $\eta_3$ at 255 210 
 \pinlabel $\eta_2$ at 365 125
  \pinlabel $\eta_1$ at 375 100
 \pinlabel $g^{n}(\eta_1)$ at 145 228
  \pinlabel $g^{n}(\eta_2)$ at 75 220
  \pinlabel $g^{n}(\eta_3)$ at 10 190
       \pinlabel $x$ at 210 85 
  \pinlabel $g^{n}(x)$ at 100 135
\endlabellist 
    \centerline{ \mbox{
\includegraphics[width=7cm]{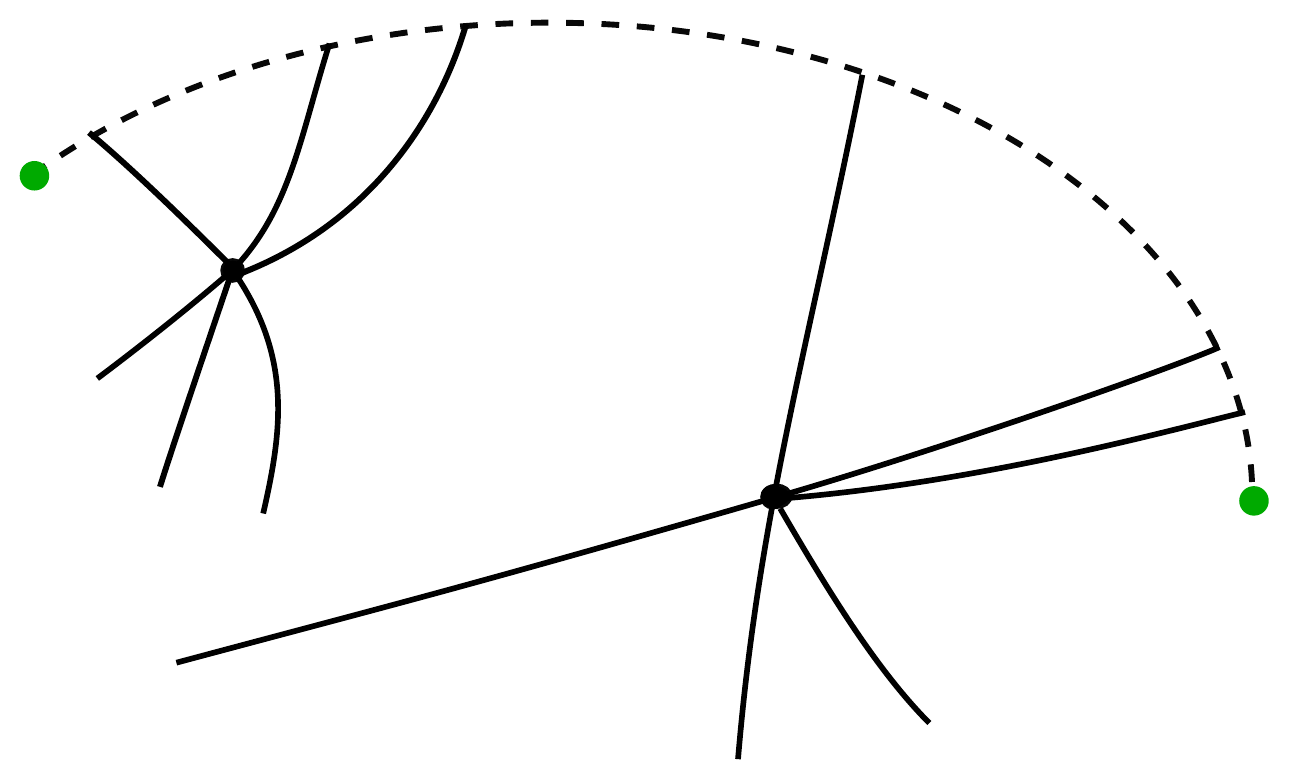} }}
\caption{Proof of Lemma \ref{lem_charac_in_lozenge}.  The interval $I \subset \Pbound$ is shown as a dotted line.  Whether $x$ is singular or regular here is irrelevant for the proof.} 
\label{fig:pointsonbound2} 
\end{figure}

Our next goal is to characterize when two points lie in the interior of the {\em same} lozenge.  For this we use the following definitions.

\begin{definition}[squares and lines]
A {\em square} of lozenges is a union of four lozenges with a common regular corner.   A {\em line} of lozenges is a chain (finite or infinite) 
of pairwise adjacent lozenges $L_i$ such that $L_i$ shares one side with $L_{i-1}$ and the opposite side with $L_{i+1}$.  
\end{definition} 

\begin{definition}[diagonal and adjacent] 
Let $L$ be a lozenge in a bifoliated plane $P$ bounded by leaves $l_1^\pm$ and $l_2^\pm$.  We say that a connected component of $P \smallsetminus \bigcup_{i=1,2} l_i^\pm$ is an {\em adjacent quadrant to} $L$ if its boundary shares a side with $L$, and a {\em diagonal quadrant} otherwise.  

For a chain, line, or square of lozenges, we say two complimentary regions of the union of their leaves are {\em adjacent} if their boundaries share a side, and \em{diagonal} if their boundaries share a corner.  
\end{definition}

\begin{lemma} \label{lem:line_or_square} 
Suppose $x_1$ and $x_2$ are in the interior of $g$-lozenges $L_1, L_2$ and we have $g^{n}(x_1) \TL g^{m}(x_2)$ for all $m, n$.  Then $L_1$ and $L_2$ lie in a line or square.  
\end{lemma}

\begin{proof} 
Let $z=\cF^+(x_1)\cap \cF^-(x_2)$. 
Then $\cF^+(g^n(z)) = \cF^+(g^n(x_1))$ and $\cF^-(g^m(z)) = \cF^-(g^m(x_2))$.  By assumption, $g^n(x_1)\TL g^m(x_2)$ for all $n,m$, so 
$\cF^+(g^n(z)) \cap \cF^-(g^m(z)) \neq \emptyset$.  The argument also holds after replacing $+$ with $-$, so we conclude  that $g^n(z) \TL g^m(z)$ for all $n,m$.   Since $x_i$ are in the interior of $g$-lozenges, $\cF^\pm(z)$ are not $g$-invariant, so 
by Lemma \ref{lem_charac_in_lozenge}, $z$ must be in the interior of a $g$-lozenge. The same argument applies reversing the role of $+$ and $-$, so we conclude 
both of the intersections $\cF^{\pm}(x_1)\cap \cF^{\mp}(x_2)$ are inside $g$-lozenges.  

If either point of $\cF^{\pm}(x_1)\cap \cF^{\mp}(x_2)$ lies in $L_1$, say for concreteness $\cF^{+}(x_1)\cap \cF^{-}(x_2) \in L_1$, then $L_1$ and $L_2$ lie in a line of lozenges all sharing common $\cF^{-}$ leaves (this follows easily from the fact that maximal chains are connected, as in Proposition \ref{proposition:cofixed_lozenge}.  Symmetrically, if some point of $\cF^{\pm}(x_1)\cap \cF^{\mp}(x_2)$ is in $L_2$, we conclude the lozenges are in a line.  If neither point of the intersection is in either $L_i$, then both intersection points must be in regions simultaneously adjacent to both $L_i$, 
 which implies that $L_1$ and $L_2$ are diagonal lozenges, and the $g$-lozenges containing the intersection points together with the $L_i$ form a square.  
 \end{proof}

We can now give the characterization of points in a common lozenge.  
By necessity, the statement is quite technical.  To help the reader parse this, we first give a statement of a special case where we assume the corners of lozenges are non-singular points, and then explain and prove the general version.

\begin{proposition}[Characterization of points in same nonsingular lozenge] \label{prop:same_lozenge_special}
Suppose $x_1$ and $x_2$ lie in the interior of $g$-lozenges $L_1$, $L_2$ with {\em non-singular} corners.  Assume $g^{n}(x_1) \TL g^{m}(x_2)$ for all $m, n$.
Then, $L_1= L_2$, if and only if the following condition holds: 
\\
There exists a non-corner fixed point $p$ such that $p \TL x_i$ for $i = 1, 2$, and for each $M \in \mathbb{N}$ there exists a non-corner fixed point $q_{M}$ unlinked with $p$ such that $q_M \TL g^{\pm m}(x_i)$ for $m< M$ and $i \in \{1, 2\}$.  
\end{proposition} 

For the forward direction of the proof, we will take points $p$ and $q_{M}$ in diagonal quadrants to the lozenge and close to opposite corners. See Figure \ref{fig:p_q} (left) for an illustration. The reader may find it helpful to verify that these points satisfy the condition, provided they are chosen close enough to the corners.   In the case where the lozenge has singular corners, we will need to replace the single points $p$ and $q_{M}$ with strings of points that ``go around" the corner, as shown in Figure \ref{fig:p_q}, right.  To formalize this, we make the following definition.

 \begin{figure}[h]
	\labellist 
	\small\hair 2pt
	\pinlabel $p$ at 12 115
	\pinlabel $x_2$ at 65 57
	\pinlabel $x_1$ at 42 67 
	\pinlabel $q_M$ at 97 20 
	\pinlabel $g^n(x_i)$ at 105 110
	\pinlabel $g^{-n}(x_i)$ at 0 10 
	\endlabellist
	\centerline{ \mbox{
			\includegraphics[width=9cm]{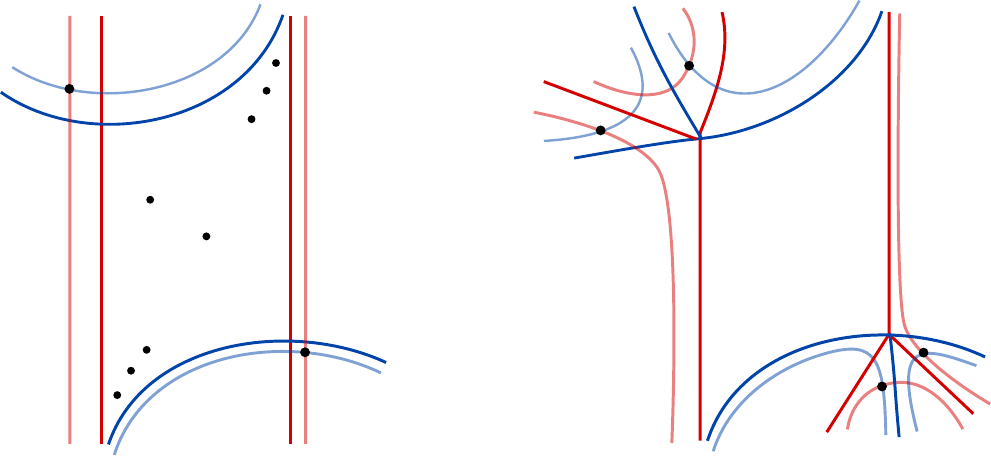}}}
	\caption{Points $p$ and $q_M$, and modification for singular lozenges}
	\label{fig:p_q}
\end{figure}
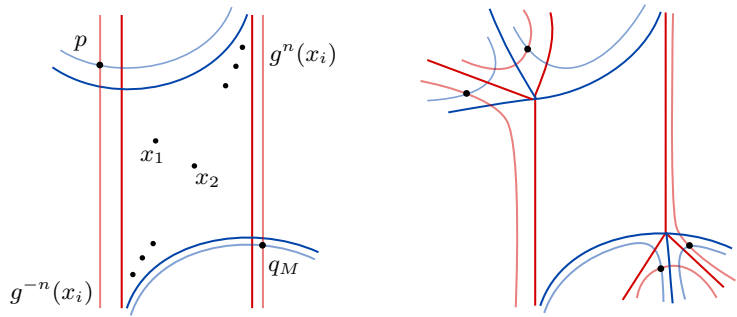

\begin{definition} 
A {\em $k$-web} is a collection of distinct noncorner points $x_0, x_2, \ldots, x_{k-1}$ in $P$ such that $\cF^+(x_i) \cap \cF^-(x_{i+1}) \neq \emptyset$ for all $i$ (with indices taken mod $k$), and $\cF^+(x_i) \cap \cF^-(x_{j}) = \emptyset$ for all $j \neq i+1$.  
\end{definition} 
Under this definition, a $2$-web is a pair of totally linked points.  
One example of a $k$-web can be obtained by taking points in distinct, alternating, quadrants surrounding a $k$-pronged singular point as shown in Figure \ref{fig:3web}.  

 \begin{figure}[h]
	\labellist 
	\small\hair 2pt
	\pinlabel $x_0$ at 120 25
	\pinlabel $x_2$ at 125 100
	\pinlabel $x_1$ at 42 63 
	\endlabellist
	\centerline{ \mbox{
			\includegraphics[width=6cm]{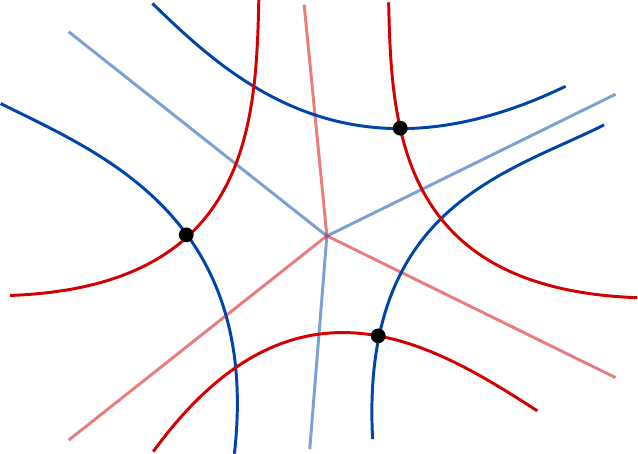} }}
	\caption{A 3-web of points}
	\label{fig:3web}
\end{figure}

The next lemma says that a $k$-web cannot cross both sides of a perfect fit.  
\begin{lemma} \label{lem:perf_web}
Suppose $l^+, l^-$ are two leaves making a perfect fit, dividing the plane into regions $R_1$, $R_2$ and $R_3$, with $R_2$ bounded by $l^+ \cup l^-$.  If $x_0, x_2, \ldots, x_{k-1}$ is a $k$-web with some point in $R_1$, then no point of the $k$-web lies in $R_3$. 
\end{lemma} 

\begin{proof} 
Suppose two points of the $k$-web lie in different regions, say without loss of generality these are separated by $l^-$.  Choose $i$ such that $x_i$ and $x_{i+1}$ are on opposite sides of $l^-$, then  $\cF^+(x_i) \cap l^- \neq \emptyset$.  Among all such points $x_i$, let $x_{i_0}$ be the one such that $\cF^+(x_{i_0})$ intersects $l^-$ at a point closest to the perfect fit.  Then any leaf $l^+$ making a perfect fit with $l^-$ is separated from all points $x_j$, $j \neq i_0$ by  $\cF^+(x_{i_0+1})$, so the points in the web lie in exactly two regions formed by the perfect fit. 
\end{proof} 

 \begin{figure}[h]
	\labellist 
	\small\hair 2pt
	\pinlabel $l^-$ at 155 100 
	\pinlabel $l^+$ at 160 125
	\pinlabel $x_{i_0}$ at 135 50
	\endlabellist
	\centerline{ \mbox{
			\includegraphics[width=6cm]{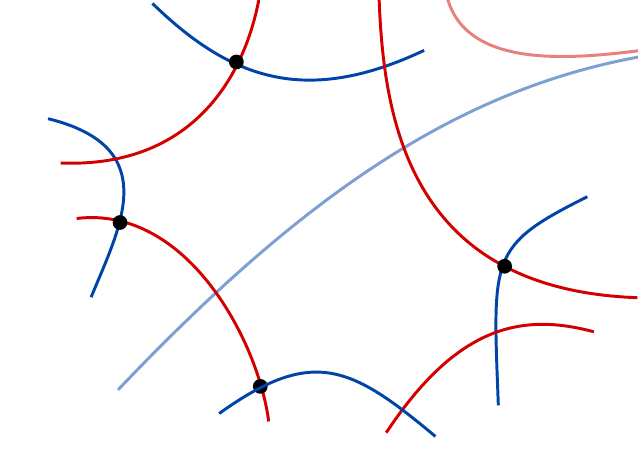} }}
	\caption{A $k$-web cannot cross both leaves of a perfect fit}
	\label{fig:kweb}
\end{figure}

Now we state and prove the general version of Proposition \ref{prop:same_lozenge_special}; which gives the special case above taking $k=l=2$, choosing $p_0$ and $q_0^{M}$ to be points in the lozenge, and taking $p_1 = p_{k-1} = p$, and $q_M = q_1^M = q_{l-1}^{M}$.

\begin{proposition}[Characterization of points in same lozenge] \label{prop:same_lozenge_II}
Suppose $x_1$ and $x_2$ lie in the interior of $g$-lozenges $L_1$, $L_2$, and assume $g^{n}(x_1) \TL g^{m}(x_2)$ for all $m, n$.
Then, $L_1= L_2$ if and only if the following condition holds: 
\\
There exists $k, l \in \bN$ and a $k-web$ $\{p_0, \ldots, p_{k-1} \}$ such that
\begin{itemize}
\item $\cF^-(p_1) \cap \cF^+(x_i) \neq \emptyset$ and $\cF^+(p_{k-1}) \cap \cF^-(x_i) \neq \emptyset$, for $i = 1, 2$ 
\item $p_j$ is unlinked with $x_i$ for all $1< j < k-1$ and $i=1,2$. 
\item For each $M \in \mathbb{N}$ there exists an $l$-web $\{q_0^{M}, \ldots, q_{l-1}^{M} \}$ where $p_i$ is unlinked with $q_j^{M}$ for all $i, j\neq 0$,
and $\cF^-(q_1^{M}) \cap \cF^+(g^{\pm m}(x_i)) \neq \emptyset$ and $\cF^+(q_{l-1}^{N}) \cap \cF^-(g^{\pm m}(x_i)) \neq \emptyset$, for all $m< M$ and $i \in \{1, 2\}$ and 
\item $q_j^{M}$ is unlinked with $x_i$ for all $1< j < l-1$ and $i=1,2$.  
\end{itemize} 
\end{proposition} 

Notice that we do not ask for any conditions on $p_0$ and $q_0^M$ in the characterization above, as both of them will end up having to be inside that one lozenge $L_1= L_2$.

\begin{proof}[Proof of Proposition \ref{prop:same_lozenge_II}]
Assume $x_i \in L_i$ and $g^{n}(x_1) \TL g^{m}(x_2)$ for all $m, n$.  
Suppose first that $x_1$ and $x_2$ lie in the same $g$-lozenge $L$.  
Let $c$ be a corner of $L$ which is a $k$-prong ($k=2$ being a regular point).  Choose a $k$-web $\{p_0 \ldots p_{k-1} \}$, with $p_0 \in L$, surrounding the corner $c$.  The points $p_j$, $j \neq 0$ are as shown in Figure \ref{fig:p_q}.  
By choosing these points sufficiently close to $c$, we can ensure that $\cF^+(p_{k}) \cap \cF^-(x_i) \neq \emptyset$, for $i = 1, 2$, and by construction $p_j$ is unlinked with $x_i$ for all $1< j < k-1$ and $i=1,2$. 

Similarly, if the other corner of $L$ is a $l$-prong, then given any $M$ we can choose an $l$-web $\{q_i^{M}\}$ around the corner, with $q_0^{M} \in L$, and chosen close enough to the corner point so that $\cF^-(q_1^{M}) \cap \cF^+(g^{\pm m}(x_i)) \neq \emptyset$ and $\cF^+(q_{l-1}^{M}) \cap \cF^-(g^{\pm m}(x_i)) \neq \emptyset$, for all $m< M$ and $i \in \{1, 2\}$. Again, we will have that $q_j^{M}$ is unlinked with $x_i$ for all $1< j < l-1$ and $i=1,2$, and for $j \neq 0$, $q_j^{M}$ is unlinked with each of $p_1, \ldots p_{k-1}$. 

For the reverse direction, we now suppose the condition holds for some $k,l$. 
By Lemma \ref{lem:line_or_square}, we have that $L_1$ and $L_2$ must be part of either a square or line of lozenges.   Suppose for a first contradiction that $L_1$ and $L_2$ are diagonal lozenges in a square.  Since $\cF^-(p_1) \cap \cF^+(x_i) \neq \emptyset$ for $i=1,2$ and $x_1$ and $x_2$ are in diagonal lozenges of the square, we have that $\cF^-(p_1)$ crosses the square and must intersect every leaf of $\cF^+$ in the square.  Similarly, $\cF^+(q_{l-1}^{1})$ must intersect every leaf of $\cF^+$ in the square, in particular it intersects $\cF^-(p_1)$, contradicting the given condition. 

We conclude that $L_1$ and $L_2$ lie in a line.  Now assume for contradiction that they are distinct. Up to switching the roles of $+$ and $-$ (and simultaneously reversing the ordering of the indexing of points in the alternating chains), we may assume that there is a leaf of $\cF^-$ passing through all lozenges in the line.  For simplicity, we assume going forward that no index switching was necessary.  

We first show that $\cF^{-}(p_1)$ cannot intersect the line of lozenges.  If $\cF^{-}(p_1)$ does intersect the line, then for sufficiently large $m$, the leaf $\cF^-(p_1)$ will separate $\cF^-(g^{m}(x_1))$ from either $\cF^-(g^m(x_2))$ or $\cF^-(g^{-m}(x_2))$, as shown in Figure \ref{fig:line_lozenge}, left.  
Since $\cF^+(q_{l-1}^{M})$ is assumed to intersect both $\cF^-(g^{m}(x_1))$ and $\cF^-(g^{\pm}(x_2))$ for sufficiently large $M$, it must therefore intersect $\cF^-(p_1)$, contradicting the condition given in the statement of the Proposition.   
 
Thus, $\cF^{-}(p_1) \cap L_1 = \emptyset$ and $\cF^{-}(p_1) \cap L_2 = \emptyset$.  Since $\cF^{-}(p_1) \cap \cF^+(x_i) \neq \emptyset$ for $i=1,2$, the point $p_1$ must be in a quadrant adjacent to $L_1$ or $L_2$, and diagonal to the other {\em at a nonsingular corner}.  In particular, the lozenges $L_1$ and $L_2$ are themselves adjacent.  For concreteness, we assume that $p_1$ is adjacent to $L_1$, as in figure \ref{fig:line_lozenge}.   

 \begin{figure}[h]
	\labellist 
	\small\hair 2pt
	\pinlabel  $\cF^{-}(p_1)$ at -20 118
	\pinlabel $\cF^-(g^m(x_1))$ at -10 155
	\pinlabel $\cF^-(g^m(x_2))$ at -10 55
	\pinlabel $L_1$ at 100 80
	\pinlabel $L_2$ at 250 80
	
	\pinlabel $L_1$ at 480 90
	\pinlabel $p_1$ at 480 172
	\pinlabel $L_2$ at 550 90
	\pinlabel $l_1$ at 433 70
	\pinlabel $q_{l-1}^M$ at 540 175
	\pinlabel $l_2$ at 599 70
	\endlabellist
	\centerline{ \mbox{
			\includegraphics[width=12cm]{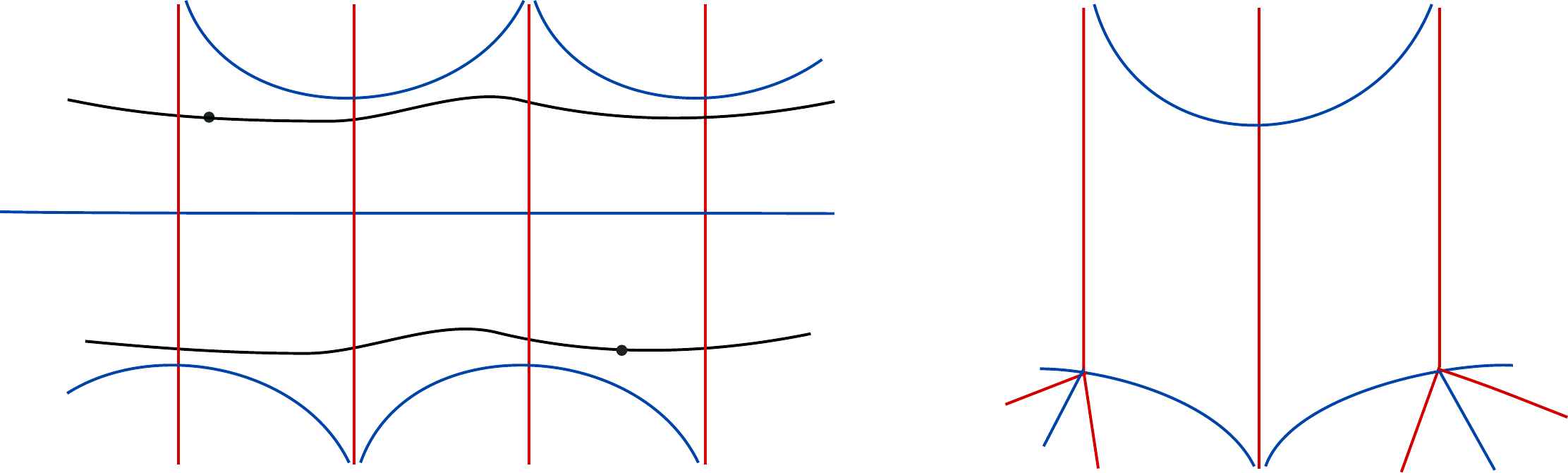} }}
	\caption{}
	\label{fig:line_lozenge}
\end{figure}

If $k=2$, then $p_1 = p_{k-1}$ and $\cF^{+}(p_1)$ must intersect $L_1$ since it intersects $\cF^-(x_1)$.  We next argue that, even for $k>2$, we always have that $\cF^{+}(p_{k-1})$ intersects either $L_1$ or $L_2$.  

Let $l_1$ and $l_2$, respectively, denote the sides of $L_1$ and $L_2$ in $\cF^+$ that are not common to both of them. 
If $k>2$, then Lemma \ref{lem:perf_web} implies that all points $p_0, \ldots p_{k-1}$ lie in the region bounded between $l_1$ and $l_2$.  
For $1<s<k-1$, we have by assumption that $p_s$ is unlinked with $x_1$, so $\cF^+(p_s)$ cannot intersect $L_1$ or $L_2$.  Thus, we in fact have that all $p_s, s \neq 0$, lie in the region above $L_1$ and $L_2$, i.e. separated from $L_1$ and $L_2$ by their shared $\cF^{-}$ boundary leaf.  This applies in particular to $p_{k-1}$.  Since  $\cF^{+}(p_{k-1})$ intersects $\cF^{-}(x_1)$, we conclude that $\cF^{+}(p_{k-1})$ passes through either $L_1$ or $L_2$.  

Since $q_{l-1}^{M}$ is unlinked with $p_{k-1}$ and $\cF^{+}(p_{k-1})$ intersects $L_1$ or $L_2$, it follows that $\cF^-(q_{l-1}^{M})$ cannot intersect the line of lozenges.  Thus, just as we argued in the case of $p_1$, we have that $q_{l-1}^{M}$ is adjacent to one $L_i$ and diagonal to the other at a nonsingular corner, i.e. as labeled in Figure \ref{fig:line_lozenge} right.  

Now we can conclude the argument by taking $M$ large.  As $m \to \infty$, the points $g^{\pm m}(x_i)$ approach the corners of $L_i$ and $\cF^+(g^{\pm m}(x_i))$ approach the sides of $L_i$
Choose $m$ large enough so that one of $\cF^+(g^{\pm m}(x_1))$ is separated from $\cF^+(x_2)$ by $\cF^+(p_{k-1})$, and fix $M>m$.  
By hypothesis, $\cF^-(q_{l-1}^{M})$ intersects both $\cF^+(g^{\pm m}(x_1))$ and $\cF^+(x_2)$, so must intersect $\cF^+(p_{k-1})$, a contradiction.  This concludes the proof. 
\end{proof}


\subsection{Proof of spectral rigidity in non-skew case.} \label{sec:nonskewed_proof}
Using our toolkit, we can now proceed with the proof of Theorem \ref{thm_main_general}.   
The definition of sign choice for (orbits of) trees of scalloped regions will appear at a natural point in the proof, see Definition \ref{def:same_sign_tree}.  

Assume that $P_1$ and $P_2$ are not skewed, thus the sets $\fixnc(\rho_1)$ and $\fixnc(\rho_2)$ are non-empty by Theorem \ref{thm_trivial_skewed_or_nowheredense}, and are equal by Proposition \ref{prop_non-corners_are_equal}. We start by defining a map $H\colon P_1 \to P_2$ on the non-corner fixed points, equivariant with respect to $\rho_1$ and $\rho_2$.  
We will extend this definition step-by-step to an orbit equivalence, only invoking the hypothesis on signs of scalloped regions in the last steps.

To start, if a point $x\in P_1$ is the (unique) fixed point of some element $g\in \fixnc(\rho_1)$,  define $H(x)$ to be the unique (by Corollary \ref{prop_non-corners_are_equal}) fixed point of $\rho_2(g)$. By Theorem \ref{thm_trivial_skewed_or_nowheredense} and Axiom \ref{Anosov_like_dense_fixed_points}, this map $H$ is a bijection from a dense subset of $P_1$ to a dense subset of $P_2$.  Our main goal is to show that $H$ extends continuously to a homeomorphism $P_1 \to P_2$.  (We will then easily verify that $H$ conjugates the two actions.)   Specifically, we wish to show that if $x_n \to x \in P_1$, then $H(x_n)$ converges to a unique point $y$ in $P_2$.   This requires several sub-steps, including an intermediate extension of the domain of definition of $H$ to include all points on the leaves of non-corner fixed points. We phrase all of these sub-steps in terms of $H$, but, by the symmetry of their constructions, they apply equally well to $H^{-1}$ and we will use these results for both directions.

\begin{lemma}[$H$ preserves total linkedness] \label{lem_linkedislinked}
Suppose that $a$ and $b$ are non-corner fixed points in $P_1$.  Then $H(a)$ and $H(b)$ are totally linked if and only if $a$ and $b$ are totally linked.  
\end{lemma}

\begin{proof}
Suppose $a$ and $b$ are totally linked non-corner fixed points and let $\alpha$ and $\beta$ be such that $a$ and $b$ are positive fixed points for $\alpha, \beta$ respectively.   Then $\rho_1(\alpha^n \beta^n)$ has a positive fixed point for all $n$ sufficiently large by Observation  \ref{obs_easy_tot_linked}.  Thus $\rho_2(\alpha^n \beta^n)$ has a fixed point also, for all $n$ sufficiently large.  
Proposition  \ref{prop_charac_tot_linked} then implies that $\rho_2(\alpha)$ and $\rho_2(\beta)$ have fixed points that are totally linked, which implies that $H(a)$ and $H(b)$ are totally linked.  
The converse direction holds by symmetry.  
\end{proof}

\begin{lemma}[Convergence to leaves of non-corner fixed points] \label{lem:leaves_of_non-corners}
Suppose $y$ is a positive, non-corner fixed point for $\rho_1(g)$, and $H(y)$ is a positive fixed point for $\rho_2(g)$.  
Let $x\in \cF_1^{-}(y)$ and suppose $(x_n)$ is a sequence of non-corner fixed points in $P_1$ converging to $x$. 
Then all accumulation points of $H(x_n)$ in $P_2 \cup \partial P_2$ lie in a bounded segment of $\cF^{-}_2(H(y))$. 

Consequently, if $x\in \cF_1^{-}(y) \cap \cF_1^{+}(z)$ for some non-corner, nonsingular fixed point $z$, then 
$\cF^{-}_2(H(y)) \cap \cF_2^{+}(H(z)) \neq \emptyset$ and  
 $H(x_n)$ converges to the (necessarily unique) point in their intersection.
\end{lemma}

\begin{rem}\label{rem:leaves_of_non-corners}
 Note that, in the proof of this property, we do not actually require the sequence $(x_n)$ to be non-corner fixed points. All we need is that it is a sequence of points where $H$ is well-defined. Hence, once we extend the definition of $H$, we may, and will, use this lemma with sequences $(x_n)$ in the new domain of definition of $H$.
\end{rem}

\begin{proof} 
Let $g, x_n, x$ and $y$ be as in the statement of the lemma. 

We show that every subsequence of $(H(x_n))$ admits a further subsequence converging to a point of $\cF_2^{-}(H(y))$, which immediately implies the first statement.  
Fix a subsequence $(x_{n_k})$.  Since $x_{n_k}$ converges to $x$ and $x \in \cF_1^{-}(y)$, any point sufficiently close to $x$ will be totally linked with $y$.  Thus, after dropping the first terms of the subsequence, we may assume all points $x_{n_k}$ are pairwise totally linked and are totally linked with $y$.   

Thus, by Lemma \ref{lem_linkedislinked}, the points $H(x_{n_k})$ are pairwise totally linked.  So, we may pass to a further subsequence along the $H(x_{n_k})$ that satisfies the conclusion of Lemma \ref{lem:TL_limit_trichotomy}.  For simplicity, reindex this sub-subsequence simply by $H(x_n)$, $n \in \mathbb{N}$.  
Note also that $H(y)$ is totally linked with $H(x_n)$ for all $n$, and $H(y)$ is a non-corner point fixed by $\rho_2(g)$. 

By Observation \ref{obs:powers_g}, condition \ref{item_star} is satisfied by $\rho_1(g)$ for the fixed point $y$ and sequence of points $(x_n)$.  By Lemma \ref{lem_linkedislinked}, we therefore have that \ref{item_star} is satisfied by $\rho_2(g)$, for the fixed point $H(y)$ and sequence of points $H(x_n)$.   
Thus, by Corollary \ref{cor:detect_non-corner_leaves}, $H(x_n)$ converges to some point in $\cF_2^{-}(H(y))$, which is what we needed to show. 

To prove the second statement, 
let $h \in G$ be such that $z$ is a negative fixed point for $h$.  
We may apply the same argument as above (passing again to a further subsequence) and using Remark \ref{rem:pos_neg_star} to show that the limit point of $H(x_n)$ lies in $\cF_2^{\pm}(H(z))$, as desired, where the sign $\pm$ depends on whether $H(z)$ was a positive or negative fixed point for $\rho_2(h)$.  Since the limit point was already assumed in $\cF_2^{-}(H(y))$, we conclude that $H(z)$ was a negative fixed point, and $\cF_2^{-}(H(y)) \cap \cF_2^{+}(H(z)) \neq \emptyset$. 
\end{proof}

As a consequence, we note the following global preservation (or global reversing) of positive and negative foliations by our map $H$.  

\begin{corollary} \label{cor:pos_is_pos}
Fix $g \in G$ such that $\rho_1(g)$ has a positive, non-corner, nonsingular fixed point for $y$, and relabel $\cF_2^\pm$ as $\cF_2^\mp$ if needed so that $H(y)$ is a positive fixed point for $\rho_2(g)$.  

With this labeling, a nonsingular point $z \in P_1$ is a positive, non-corner fixed point for some $\rho_1(h)$ if and only if $H(z)$ is a positive, non-corner fixed point for $\rho_2(h)$.
\end{corollary}

\begin{proof}
If $z$ is any positive, non-corner, nonsingular fixed point of $\rho_1(h)$ in $P_1$ that is either totally or partially linked with $y$, then Lemma \ref{lem:leaves_of_non-corners} directly implies that $H(z)$ is a positive fixed point of $\rho_2(h)$. Since any point in the domain of definition of $H$ can be joined to $y$ by a finite union of segments that are alternatively on the $\cF^+$ and $\cF^-$ leaves of non corner fixed points, the conclusion of the corollary follows by iteration.
\end{proof} 

Thus, by relabeling $\cF^\pm_2$ by $\cF^\mp_2$ as in Corollary \ref{cor:pos_is_pos} if needed, we adopt the following convention. 

\putinbox{
\begin{convention*}
Using Corollary \ref{cor:pos_is_pos}, we will now assume that {\em $H$ sends positive non-corner fixed points to positive non-corner fixed points.}
\end{convention*}}

Using this, we may now extend the domain of definition of $H$.  
We will now abuse terminology a little and call a leaf a \emph{non-corner leaf} if it contains a non-corner fixed point.

\begin{corollary}[First extension of $H$] \label{cor:define_non-corner_leaves}
Let $Q_i \subset P_i$ be the set of points lying on the intersection of two non-corner leaves.  
Extend $H$ to be defined on $Q_1$ as follows: 
given $x = \cF_1^{+}(y) \cap \cF_1^{-}(z) \in Q_1$, let  $H(x) = \cF_2^{+}(H(y)) \cap \cF_2^{-}(H(z))$. 

This is a well defined bijection $Q_1 \to Q_2$, is $G$-equivariant, sends leaves of $\cF_1^+$ (resp.~$\cF_1^-$), to leaves of $\cF_2^+$ (resp.~$\cF_2^-$), and preserves totally linked, partially linked, and unlinked pairs.
\end{corollary}

\begin{proof}
That this extension of $H$ is well defined, a bijection,  and sends leaves of $\cF_1^\pm$ to leaves of $\cF_2^\pm$ follows from Lemma \ref{lem:leaves_of_non-corners}, Corollary \ref{cor:pos_is_pos}, and our Convention. 

For equivariance, suppose $x = \cF_1^{+}(y) \cap \cF_1^{-}(z) \in Q_1$.  Then we have
$ \rho_1(g)(x) = \cF_1^{+}(\rho_1(g)(y)) \cap \cF_1^{-}(\rho_1(g)(z))$, which implies that
\[ H(\rho_1(g)(x)) = \cF_2^{+}(H(\rho_2(g)(y))) \cap \cF_2^{-}(H(\rho_2(g)(z))) = \rho_2(g)(H(x)),\]  
as desired. 

It remains to remark that $H$ preserves the linking properties.  Suppose $y_i, z_i$ are non-corner fixed points in $P_1$ for $i=1,2$.  Let $x_i =  \cF_1^{+}(y_i) \cap \cF_1^{-}(z_i)$. If $x_1$ and $x_2$ are totally linked, then $\cF_1^{+}(y_i) \cap \cF_1^{-}(z_j) \neq \emptyset$ when $i \neq j$.  Thus, $\cF_2^{+}(H(y_i)) \cap \cF_2^{-}(H(z_j)) \neq \emptyset$ also, so the images are totally linked.  
The argument is completely analogous for unlinked and partially linked pairs $x_1, x_2$.  
\end{proof}

\begin{lemma}[$H$ preserves lozenges]\label{lem_preserves_lozenges}
Let $g \in \fix(\rho_1)$.  
If $L_1$ is a $\rho_1(g)$-lozenge, then $H(L_1 \cap Q_1) = L_2 \cap Q_2$ for some $\rho_2(g)$-lozenge $L_2$ 
\end{lemma} 

\begin{proof} 
Let $x \in L_1 \cap Q_1$.  By Lemma \ref{lem_charac_in_lozenge}, $\rho_1(g)^{n }(x) \TL \rho_1(g)^{m}(x)$ for all $n, m\in\bZ$.  By Corollary \ref{cor:define_non-corner_leaves}, $\rho_2(g)^{n }(H(x)) \TL \rho_2(g)^{m}(H(x))$ for all $n, m$. Thus by Lemma \ref{lem_charac_in_lozenge}, $H(x)$ lies inside a $\rho_2(g)$-lozenge or a $\rho_2(g)$-invariant leaf. However, $H(x)\in Q_2$ by Corollary \ref{cor:define_non-corner_leaves}, and (by definition) $Q_2$ does not contain any point on a $\rho_2(g)$-invariant leaf since $\rho_2(g)$ fixes corner points, we deduce that  $H(x)$ lies inside a $\rho_2(g)$-lozenge.

We have left to show that all points of $L_1 \cap Q_1$ are sent to the \emph{same} lozenge by $H$. For this we use Proposition \ref{prop:same_lozenge_II}, which was designed for this purpose.  
Let $x_1, x_2 \in L_1 \cap Q_1$. Then $x_1$ and $x_2$ satisfy to the condition given in Proposition \ref{prop:same_lozenge_II}.  Let $\{p_i\}$ and $\{q_j^M\}$ be the $k$-web and $l$-webs given by the condition.  Since $H$ preserves intersections and non-intersections of $\cF^+$ and $\cF^-$ leaves, it sends webs to webs, and it follows that $H(x_1)$ and $H(x_2)$ also satisfy the condition given in Proposition \ref{prop:same_lozenge_II}, using the webs $H(p_i)$ and $H(q_j^{M})$. Thus $H(x_1)$ and $H(x_2)$ are in the same lozenge.
\end{proof}

When $L$ is a $\rho_1(g)$-lozenge, we abuse notation slightly and write $H(L)$ for the unique lozenge containing $H(L\cap Q_1)$.
Since $H$ preserves the intersections of $\cF^+$ and $\cF^-$ leaves, one easily deduces that $H$ sends lines of lozenges to lines of lozenges. We next show that $H$ also preserves the linear order of a line of lozenges.

\begin{definition}[Order of lines and ends of finite lines]
Let $\cL$ be a line of lozenges. An \emph{order} of $\cL$ is an enumeration $\cL = \{L_i\}_i\in I$, with $I\subset \bZ$, where $L_i$ is adjacent to $L_{i+1}$.

 Let $\cL= \{L_0, \dots, L_n\}$ be an ordered finite line of lozenges.
 Let $c_0$ be the corner of $L_0$ that is not a corner of $L_1$ and $c_n$ the corner of $L_n$ which is not corner of $L_{n-1}$.
 Then the \emph{negative end}, $\cE_-$ is the set of points $x$ such that $x$ is totally linked with $c_0$ and such that $\cF^-(x)$ (or $\cF^+(x)$) intersects every lozenge in $\cL$.
 
 Similarly, the \emph{positive end}, $\cE_+$ is the set of points $x$ such that $x$ is totally linked with $c_n$ and such that $\cF^-(x)$ (or $\cF^+(x)$) intersects every lozenge in $\cL$.
\end{definition}

\begin{lemma}[$H$ preserves ordered lines of lozenges]\label{lem_H_preserves_ordered_lines}
 Let $g\in \fix(\rho_1)$, and $\cL= \{L_i\}_{i\in I}$ be an ordered line of $\rho_1(g)$-lozenges, with $I=\{0,\dots,n\}$ or $I=\bZ$.  (Note that any infinite line contains infinitely many nonseparated leaves in its boundary, so is contained in a scalloped region and hence forms part of an ordered line indexed by $\bZ$.)
 
 Then $H(\cL)= \{H(L_i)\}_{i\in I}$ is an ordered line of lozenges. Moreover, if $\cL$ is finite than $H$ sends its negative (resp.~positive) end to the negative (resp.~positive) end of $H(\cL)$.
\end{lemma}

\begin{proof}
To fix notation, we assume that there exists a leaf of $\cF^-$ intersecting the whole line $\cL$, the other case is handled by the same proof after switching all $+$ and $-$ signs. 
For each $i \in I$ (except the initial element if the line is finite), we can choose a $k_i$-web $\{p_0^i, \dots, p_{k_i-1}^i\}$ of points in $Q_1$ such that $p_0^i\in L_{i-1}$, and $\cF^+(p_{k_i-1}^i)$ intersects $L_i$, but $\cF^{\pm}(p_{j}^i)$ do not intersect any lozenges in $\cL$ for any $0<j<k_{i}-1$.
 
Since $H$ preserves lozenges and intersections of leaves, we deduce that $H(p_0^i), \dots, H(p_{k_i-1}^i)$ is a $k_i$-web with $H(p_0^i)\in H(L_{i-1}\cap Q_1)$, $\cF^+(H(p_{k_i-1}^i))$ intersecting $H(L_i)$, and the leaves of $H(p_{j}^i)$ do not intersect any lozenges in $H(\cL)$ for any $0<j<k_{i}-1$.
 Thus we deduce by Lemma \ref{lem:perf_web} that $H(L_{i-1})$ and $H(L_{i})$ must be adjacent, that is, $\{H(L_i)\}_{i\in I}$ is an ordered line.
 
 We have only left to prove that, if $\cL$ is a finite line, then (the intersection of $Q_1$ with) the positive and negative ends are sent to positive and negative ends by $H$. This follows as above: for any point $x\in Q_1\cap \cE_-$, we can find a $k$-web $\{p_0, \ldots p_{k-1}\}$ with $p_0=x$, $\cF^+_1(p_{k-1})$ intersecting $L_0$ and $\cF^{\pm}_1(p_{j})$ not intersecting any lozenges in the line for $0<j<k-1$. Therefore, $\{H(p_0), \ldots H(p_{k-1})\}$ must be a $k$-web with $\cF^+_2(H(p_{k-1}))$ intersecting $H(L_0)$. Since this is a $k$-web $H(x)=H(p_0)$ is totally linked with the end corner of $H(L_0)$, i.e., it is in the negative end of $H(\cL)$. Similarly, the image of the positive end of $\cL$ must be sent into the positive end of $H(\cL)$.
\end{proof}

\begin{rem}\label{rem_ends_single_lozenge}
 Note that the preservation of ends of lines of lozenges by $H$ holds even when the ``line'' consists of a unique lozenge. That is, given $L$ a $\rho_1(g)$-lozenge and $l$ a (non-corner) leaf intersecting $L$. Then call $\cE_1,\cE_2$ the ``ends'' of $L$ in the adjacent quadrants intersecting $l$, the argument above shows that $H(\cE_1\cap Q_1)$ and $H(\cE_2\cap Q_1)$ are the respective ends of $H(L)$ intersecting $H(l)$.
 \end{rem}

We note the following immediate corollary of Lemma \ref{lem_H_preserves_ordered_lines}.
\begin{corollary} \label{cor:trees_are_same}
If $T_1$ is a tree of scalloped regions in $P_1$, then there exists a tree of scalloped regions $T_2 \subset P_2$ such 
$H(T_1 \cap Q_1) = T_2 \cap Q_2$.  If $\rho_1(h)$ fixes each corner of $T_1$, then $\rho_2(h)$ fixes each corner of $T_2$.   
\end{corollary}

Our next goal is to extend the domain of definition of $H$ to the remaining points on non-corner leaves by a continuity argument (Proposition \ref{prop_unique_limit_for_NCcapfixed}).  For this, we require several lemmas.

\begin{lemma}\label{lem:accumulation_for_non_corner_leaves}
 Suppose $x_n \in Q_1$ converges to $x \in \cF_1^{+}(c) $, where $c$ and each half leaf of $\cF_1^{\pm}(c)$ are invariant by $\rho_1(h)$.  Then all accumulation points of $H(x_n)$ in $P_2$ lie on $\rho_2(h)$-invariant leaves in $P_2$.  \end{lemma}

\begin{proof}
We can assume $c$ is a corner point, otherwise the conclusion reduces to the first statement of Lemma \ref{lem:leaves_of_non-corners} (using Remark \ref{rem:leaves_of_non-corners}). 
Suppose some subsequence $H(x_{n_k})$ converges to $y$ on a leaf $l^+$ not invariant by $\rho_2(h)$.  By passing to a further subsequence, we may assume that all the points $x_{n_k}$ lie in a single (closed) quadrant of $x$, specifically, we will use the fact that no leaf of $\cF^+(x) = \cF^+(c)$ separates some terms of the sequence from other terms.  

Abusing notation, reindex this subsequence as $H(x_n)$.
By Lemma \ref{lem_distinct_leaves_distinct_saturations}, 
 there exists some non-corner leaf $l^-$ that intersects $l^+$ but not $h(l^+)$\footnote{While Lemma \ref{lem_distinct_leaves_distinct_saturations} uses Axiom \ref{Anosov_like_totallyideal} in its proof, its use can easily be avoided here at the cost of lengthening the argument, we use it only as a shortcut.}.   
Thus, for sufficiently large $n$,  $\cF_2^+(H(x_n)) \cap l^- \neq \emptyset$ but $\cF_2^+(\rho_2(h)(H(x_n))) \cap l^- = \emptyset$.  Since these are intersections of non-corner leaves, we may apply $H^{-1}$ and conclude that $\cF_1^+(x_n) \cap H^{-1}(l^-) \neq \emptyset$ for all sufficiently large $n$.  It follows (as in Observation \ref{obs_sequence_of_totally_linked}) that $H^{-1}(l)$ either intersects, or makes a perfect fit with, $\cF_1^+(x) = \cF_1^+(c)$ or a leaf non-separated with it. 
However, we also have $\cF_1^+(\rho_1(h)(x_n)) \cap H^{-1}(l^-) = \emptyset$ for all sufficiently large $n$. Since $\rho_1(h)(x_n)$ also approaches a point on $\cF_1^+(c)$, and lies on the same side of $\cF_1^+(c)$ as $(x_n)$, we conclude that $H^{-1}(l^-)$ makes a perfect fit with $\cF_1^+(c)$ on the opposite side as this sequence.  But this contradicts the fact that $\cF_1^+(x_n) \cap H^{-1}(l^-) \neq \emptyset$ for all large $n$.  
\end{proof}

\begin{lemma}\label{lem_case_of_general_sequences}
Suppose $x \in \cF_1^{-}(w)\cap \cF^+(c)$, where $w$ is a non-corner fixed point of $\rho_1(g)$ and $c$ is a point fixed by $\rho_1(h)$.  Suppose $(x^1_n)$ and $(x^2_n)$ are sequences in $Q_1$ both converging to $x$, with $H(x^i_n)$ converging to $y_i$.  
Either $y_1 = y_2$, or $y_1$ and $y_2$ are on opposite sides of the same lozenge, or they are on sides of adjacent lozenges.
\end{lemma}

\begin{proof} 
If $c$ is non-corner then $y_1 = y_2$ by Corollary \ref{cor:define_non-corner_leaves}. So we assume from now on that $c$ is a corner point. Hence, by Lemma \ref{lem_power_fixes_max_chain}, after replacing $h$ by a power of itself that fixes all half-leaves through $c$, we may assume that there exists a chain of lozenges $\cC$ in $P_2$ fixed by $\rho_2(h)$.  
Since $w$ is a non-corner fixed point, by Lemma \ref{lem:leaves_of_non-corners} we have $y_1, y_2\in \cF^{-}_2(H(w))$, i.e. $\cF^{-}_2(y_1) = \cF^{-}_2(y_2)$.

By Lemma \ref{lem:accumulation_for_non_corner_leaves}, each of $y_1$ and $y_2$ must be on a $\rho_2(h)$-invariant leaf. Therefore $y_1$ and $y_2$ are either on the side of a $\rho_2(h)$-lozenge, or on a $\rho_2(h)$-invariant half-leaf which is not the side of any lozenge.

Suppose first that $y_1$ is not on the side of a $\rho_2(h)$-lozenge. So it lies in some connected component $S$ of $P_2 \smallsetminus \cC$.
Such a connected component $S$ has a unique corner $d$ of $\cC$ in its boundary, and all $\rho_2(h)$-invariant half-leaves that intersect $S$ are leaves of $\cF^\pm(d)$.  Since $y_1 \in \cF^{-}_2(H(w))$, we must have that $y_1 \in \cF^{+}(d)$, which intersects $\cF^{-}_2(H(w)) = \cF^{-}_2(y_2)$ at a unique point.  Thus, we are forced to have $y_1 = y_2$.   The same argument applies reversing the role of $y_1$ and $y_2$, thus we are left to treat only the case where $y_1$ and $y_2$ both lie on sides of $\rho_2(h)$-lozenges.

Since $y_1$ and $y_2$ are on the same $\cF^-$-leaf, they must be on the sides of lozenges that are contained inside a line (possibly a trivial line, i.e., a single lozenge) of $\rho_2(h)$-lozenges.

Up to passing to a subsequence, we may assume that each sequence $x^i_n$ lies on a single side of $\cF^+(x)$ (however, $x^1_n$ and $x^2_n$ need not necessarily lie on the same side).  Thus, each sequence is eventually in at most one $\rho_1(h)$-lozenge. Moreover, either they are in the same lozenge, or in adjacent lozenges, or, possibly, one of the sequences is in the last lozenge of a finite line, and the other in the adjacent end of the line.

Since $H$ preserves lozenges (Lemma \ref{lem_preserves_lozenges}), and ordered lines (Lemma \ref{lem_H_preserves_ordered_lines}), we deduce that, for $n$ large enough, $y_n^1$ and $y_n^2$ are either in the same lozenge, or in adjacent lozenges, or, one of the sequences is in the last lozenge of a finite line, and the other in the adjacent end of the line. In all cases we have either $y_1$ and $y_2$ are equal, or are on the opposite sides of the same lozenge, or are on opposite sides of adjacent lozenges.
\end{proof}

The next lemma implies that convergent sequences lying on one side of their limit leaf have images with unique limit points under $H$.  

\begin{lemma}\label{lem_case_of_sequences_on_same_side}
Let $x \in \cF_1^{-}(w)\cap \cF^+(c)$, where $w$ is a non-corner fixed point of $\rho_1(g)$ and $c$ is a point fixed by $\rho_1(h)$. 
Suppose that $(x^1_n)$ and $(x^2_n)$ are sequences in $Q_1$ lying on {\em the same side} of $\cF^+(x)$, both converging to $x$. 
If $H(x^i_n)$ converges to $y_i\in P_2$, $i=1,2$, then $y_1=y_2$.
\end{lemma}

\begin{proof}
Let us assume for a contradiction that $y_1 \neq y_2$.  Since (for large $n$), the points $x_n^i$ all lie in a common $\rho_1(h)$-lozenge or end, the same is true of their images $H(x_n^i) = y_n^i$.  
Since we assumed $y_1 \neq y_2$, by Lemma \ref{lem_case_of_general_sequences}, we have that $y_1$ and $y_2$ are on opposite sides of a common lozenge $L$, which contains the tail end of both sequences $y_n^i$.  
Consider the maximal (with respect to inclusion) line of lozenges $\cL^+$ containing $L$, intersecting a common $\cF_2^+$-leaf.  Note $\cL^+$ may be a trivial line containing only $L$. We consider separately the cases where $\cL^+$ is finite or infinite, arriving at a contradiction in each case.

\setcounter{case}{0}
\begin{case}[$\cL^+$ is infinite]
 If $\cL^+$ is infinite, then its boundary contains an infinite set of pairwise non-separated leaves in $\cF^+$.  Thus, it is in fact a bi-infinite line, making up a scalloped region $U$. Let $\cL^-$ be the other line of lozenges covering $U$ (i.e. every $\cF_2^-$-leaf intersecting $U$ intersects all lozenges in $\cL^-$). We fix an ordered indexing $\cL^- = \{F_i\}_{i\in \bZ}$. 
 Note that every lozenge in $\cL^-$ intersects $L$ hence, for any $n$ there exists some $i^1_n$ such that $y^1_n \in F_{i^1_n}$. Up to reversing the indexing, we may assume that $i_n \to +\infty$ as $n \to +\infty$. Let $i^2_n$ be the index such that $y^2_n \in F_{i^2_n}$.  Then $i^2_n \to -\infty$ as $n\to \infty$.
 
 Now consider $H^{-1}(\cL^-)$. Since $H^{-1}$ preserves the order of lines, we have that $x^1_n$ and $x^2_n$ cannot converge to the same side of $H^{-1}(L)$, a contradiction. 
\end{case}

\begin{case}[$\cL^+$ is finite]
In this case, up to renaming $y_1$ and $y_2$, we can find a point $b\in Q_2$ on an end of the line $\cL^+$ such that $b$ is totally linked with all $y^1_n$ but only partially linked with $y^2_n$, for $n$ large enough, as in Figure \ref{fig:same_side}.
However, this implies that $H^{-1}(b)$ is totally linked with all $x^1_n$ but 
only partially linked with $x^2_n$, for $n$ large enough. This is impossible since $x^1_n$ and $x^2_n$ both converge to $x$ from the same side of $\cF_1^+(x)$, so $H^{-1}(b)$ is either eventually totally linked with both sequences or eventually partially linked with both sequences.\qedhere  
\end{case}
\end{proof}

 \begin{figure}[h]
	\labellist 
	\small\hair 2pt
	\pinlabel  $L$ at 140 85
	\pinlabel $y^1_n$ at 120 110 
	\pinlabel $y^2_n$ at 160 110 
	\pinlabel $b$ at 132 3
	\endlabellist
	\centerline{ \mbox{
			\includegraphics[width=6cm]{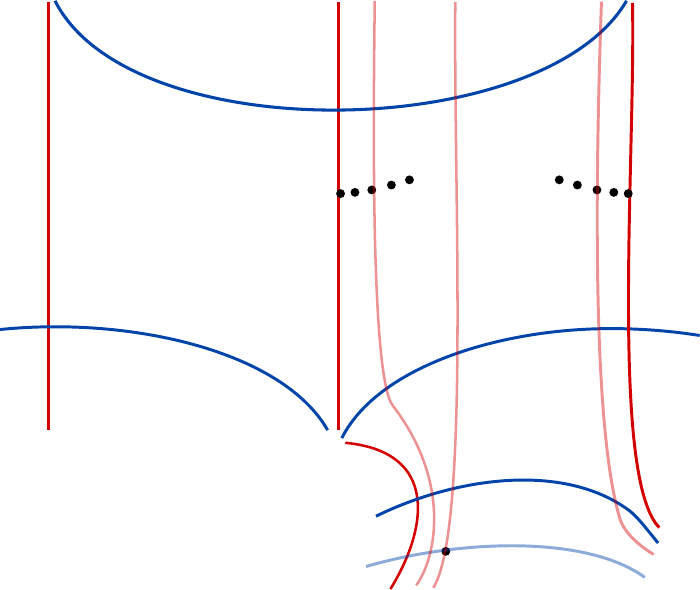} }}
	\caption{Position of $b$ in the case of a finite vertical line of lozenges.}
	\label{fig:same_side}
\end{figure}
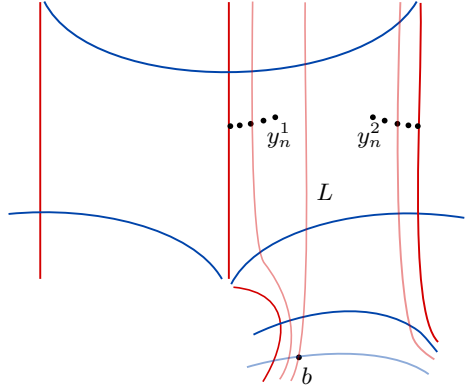

\begin{lemma} \label{lem:positive_corner}
Let $L$ be a $\rho_1(h)$-lozenge and suppose $(x_n) \in Q_1 \cap L$ converges to $x \in \cF_1^{-}(w)\cap \cF^+(c)$, where $w$ is a non-corner fixed point and $c$ is a positive corner fixed point of $\rho_1(h)$.  
Then all accumulation points of $H(x_n)$ are on the leaves of \emph{positive} corner fixed points of $\rho_2(h)$. 
\end{lemma} 

\begin{proof} 
Let $y$ be an accumulation point of $H(x_n)$, after passing to a subsequence we assume $H(x_n) \to y$.   Since $x_n \in L$, $H(x_n)$ lies in some $\rho_2(h)$-lozenge $L'$ and $y$ is on the side of $L'$.  
Let $a \in L$ be a positive non-corner fixed points of $\rho_1(\alpha)$.  Replacing $\alpha$ by some very high power $\alpha^N$ if needed,  Lemma \ref{lem:make_NC_points} says that for $m$ large enough, the elements $\rho_2(h^m\alpha)$ have non-corner fixed points $a_m \in L$. Moreover, as $m$ goes to infinity, $\cF_1^+(a_m)$ accumulates onto $\cF_1^+(x) = \cF_1^+(c)$.  Let $z_m = \cF_1^+(a_m) \cap \cF_1^-(w)$.  

Then $z_m$ converges to $x$ and lies on the same side of $\cF^+(x)$ as $x_n$, so Lemma \ref{lem_case_of_sequences_on_same_side} implies that  $H(z_m)$ converges to $y$.    However, $H(z_m) = \cF_2^+(H(a_m)) \cap \cF_2^-(H(w))$, and $\cF_2^+(H(a_m))$ converges to the side of $L'$ containing its positive corner.  Thus $y$ was on the side of a positive corner.   
\end{proof}

The primary application of Lemma \ref{lem:positive_corner} will be via the following immediate corollary. 
\begin{corollary} \label{cor:positive_corner} 
Let $(x^1_n)$ and $(x^2_n)$ be two sequences converging to $x$ as in Lemma \ref{lem:positive_corner}.  Then their images $H(x^1_n)$ and $H(x^2_n)$ cannot converge to opposite sides of the same $\rho_2(h)$-lozenge.  
\end{corollary} 

The next lemma shows that trees of scalloped regions are the only obstruction to uniqueness of limit points of images of convergent sequences under $H$.

\begin{lemma}\label{lem_case_of_sequences_on_opposite_sides}
Let $x \in \cF_1^{-}(w)\cap \cF^+(c)$, where $w$ is a non-corner fixed point of $\rho_1(g)$ and $c$ is a point fixed by $\rho_1(h)$. 
Suppose that  $(x^1_n)$ and $(x^2_n)$ are sequences in $Q_1$ lying on {\em opposite sides} of $\cF^+(x)$, both converging to $x$. 
If $H(x^i_n)$ converges to $y_i\in P_2$, $i=1,2$, then
\begin{enumerate}
\item either $y_1=y_2$, or
\item \label{item_case_discontinuity} $x$ is on a side of a tree of scalloped regions, with each lozenge fixed by $\rho_1(h)$, and $y_1, y_2$ are on opposite sides of adjacent $\rho_2(h)$-lozenges contained in a tree of scalloped regions.
\end{enumerate} 
\end{lemma}

\begin{proof} 
Assume that $y_1 \neq y_2$.  Corollary~\ref{cor:define_non-corner_leaves} implies that $c$ is a corner, Lemma \ref{lem_case_of_general_sequences} implies that $y_1$ and $y_2$ are either on opposite sides of the same lozenge or on sides of adjacent lozenges, and finally Corollary \ref{cor:positive_corner} implies that $y_1$ and $y_2$ are in fact on opposite sides of adjacent lozenges; without loss of generality we assume these are the sides containing the positive corners for $\rho_2(h)$ (if not, simply replace $h$ with its inverse).  Dropping the first terms of the sequences if needed, we may assume that $y_n^i$ lies in a single lozenge $L_i$.  
Let $L_i'$ denote $H^{-1}(L_i)$, then $L_i'$ contains the sequence $x^i_n$.  

Let $\cL$ be the line of lozenges containing $L_1$ and $L_2$. Our goal is to show this line lies in a tree of scalloped regions; the first step is to show that  $\cL$ is an infinite line.  
Let $L_0 \neq L_1$ be the lozenge (or end) in $\cL$ that contains $y_1$ in its boundary, and take a sequence $(y^0_n)$ of points in $Q_2\cap L_0$ that converges to $y_1$. Let $x^0_n = H^{-1}(y^0_n)$. Passing to a subsequence, we may assume that $x^0_n$ converges to a point $x_0$. Notice that we cannot have $x_0=x$, because the previous arguments would imply that $y^0_n$ and $y^2_n$ would then either be in the same or adjacent lozenges, which is not true by construction. As above, we further deduce that $x_0$ and $x$ are on opposite (positive corner) sides of adjacent lozenges. In particular, $L_0$ is a true lozenge, not an end, and so is $L_0':=H^{-1}(L_0)$.

Now we can repeat this process inside the lozenge $L_0'$: call $L_{-1}'$ the lozenge, or end, in the line $\cL'=H^{-1}(\cL)$ such that $x_0$ is on the side shared by $L_0'$ and $L_{-1}'$. Then, consider a sequence $x^{-1}_n$ in $Q_1\cap L_{-1}'$ that converges to $x_0$. By the above (applied to the sequences $x^{-1}_n$, $x^{0}_n$ and their images by $H$), the sequence $H(x^{-1}_n)$ converges to a point $y^{-1}$ on the side of $L_{-1}$ opposite to $L_0$. In particular $L_{-1}$ (and thus $L_{-1}'$) is a true lozenge in $\cL$, not an end.
Iterating this process, we deduce that $\cL$ must be a bi-infinite line of lozenges.  

Moreover, we have constructed for each $i$ a sequence of points $x^i_n \in L'_i \cap Q_1$ converging to a point on the positive side of $L_i$ in $\cF^+_1$, with the property that, for $i \in 2\bZ$ the sequences $x^i_n$ and $x^{i-1}_n$ have the same limit, but the sequences $H(x^i_n) = y^i_n$ and $H(x^{i-1}_n) = y^{i-1}_n$ do not.  (Rather it is the sequences $y^i_n$ and $y^{i+1}_n$, for $i \in 2\bZ$, that have common limit points).  See Figure \ref{fig:opposite_sides}.  We will use these in the next step of the proof.  

 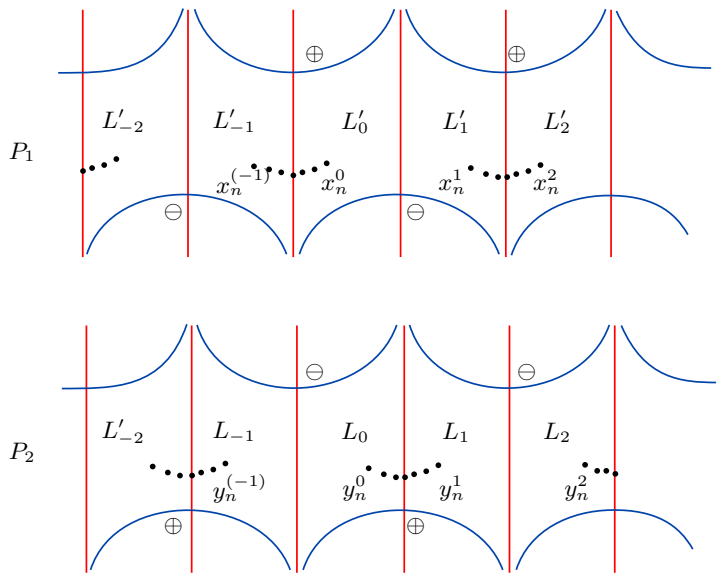
\begin{figure}[h]
	\labellist 
	\small\hair 2pt
	\pinlabel $P_1$ at -15 210
	\pinlabel  $L_{-2}'$ at 35 225
	\pinlabel  $L_{-1}'$ at 90 225
	\pinlabel  $L_{0}'$ at 150 225	
	\pinlabel  $L_{1}'$ at 200 225
	\pinlabel  $L_{2}'$ at 250 225
	\pinlabel $\oplus$ at 130 259
	\pinlabel $\oplus$ at 230 259
	\pinlabel $\ominus$ at 60 180
	\pinlabel $\ominus$ at 180 180
	\pinlabel  $x_{n}^{(-1)}$ at 95 195	
	\pinlabel  $x_{n}^0$ at 140 195
	\pinlabel  $x_{n}^1$ at 198 195	
	\pinlabel  $x_{n}^2$ at 245 195	
	\pinlabel $P_2$ at -15 60
	\pinlabel  $L_{-2}'$ at 35 70
	\pinlabel  $L_{-1}$ at 90 70
	\pinlabel  $L_{0}$ at 150 70	
	\pinlabel  $L_{1}$ at 200 70
	\pinlabel  $L_{2}$ at 250 70
	\pinlabel $\oplus$ at 60 23
	\pinlabel $\oplus$ at 180 23
	\pinlabel $\ominus$ at 130 100
	\pinlabel $\ominus$ at 235 100
	\pinlabel  $y_{n}^{(-1)}$ at 93 42	
	\pinlabel  $y_{n}^0$ at 150 42
	\pinlabel  $y_{n}^1$ at 198 42	
	\pinlabel  $y_{n}^2$ at 260 42	
		
	\endlabellist
	\centerline{ \mbox{
			\includegraphics[width=9cm]{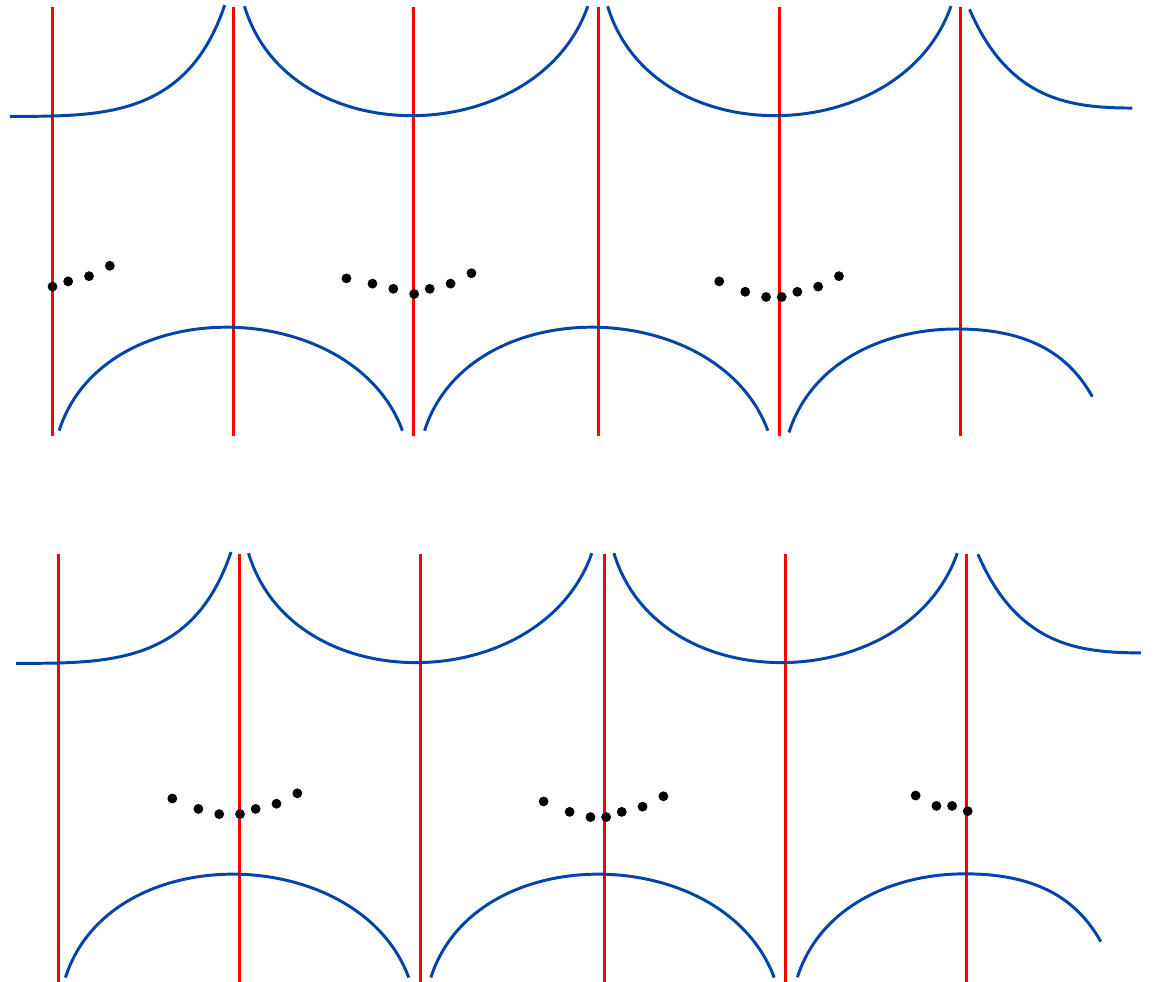} }}
	\caption{Sequences $x^i_n$ in a line of lozenges and their images $y^i_n$ under $H$. The positive (resp. negative) corners for the action of $h$ are denoted by $\oplus$ (resp. $\ominus$).}
	\label{fig:opposite_sides}
\end{figure}

We will now show that this infinite line $\cL$ is part of a tree of scalloped regions.  We  do this by repeating the arguments above to each lozenge $L_i$, but switching the roles of $\cF^+$ and $\cF^-$, thus showing that each $L_i$ must be part of \emph{another} infinite line (one that intersects a common $\cF^+_2$-leaf rather than a common $\cF^-_2$ leaf).   This process can then be iterated again, showing that each lozenge in the (as-yet-constructed) tree is part of two bi-infinite families, thus adjacent to lozenges on all four sides, which will complete the proof. 

Now for the details: Fix one lozenge $L_i$ in $\cL$. The indexing in the proof will be slightly different depending on whether $i$ is even or odd.  For the sake of readability and concreteness, we assume $i = 0$, the odd case being completely analogous.  
Let $c_0$ denote the positive corner of $L_0$.  From the construction in the first step of the proof we have $y^0_n \in L_0$ and $y^1_n \in L_1$ converging to a point in $\cF^+(c_0)$ on the side of $L_0$. 

Consider a sequence $(a_n)$ in $ L_0\cap Q_2$ that converges to a point $z \in \cF^-_2(c_0)\cap \cF_2^-(w_1)$, where $w_0$ is a non-corner fixed point for some non-trivial element $\rho_2(g_1)$. Pick another sequence $b_n\in Q_2$ that converges also to $z$ but from the other side of $\cF_2^-(c_0)$.  Notice that for each $n$, we can choose a $k$-web $b^0_n,\dots,b_n^{k-1}$ such that $b_n^0\in L_1$ and $b_n^{k-1}=b_n$, where $k$ depends on the number of prongs at $c_0$.  

We can apply all the arguments given thus far to the sequences $(a_n)$, $(b_n)$ and their images by $H^{-1}$, concluding that either their limits $a':= \lim_{n \to \infty} H^{-1}(a_n)$ and $b':=\lim_{n \to \infty} H^{-1}(b_n)$ are equal, or that $L_0$ is part of an infinite line of lozenge which intersect a common $\cF^+_2$-leaf.
Thus, all we have left to do is show that $a'\neq b'$.

Applying Lemma \ref{lem:positive_corner}, since $z\in \cF_2^-(c_0)$ and $c_0$ is positive for $\rho_2(h)$, we deduce that $a'$ and $b'$ must be on the $\cF^-_1$-leaf of a positive corner fixed point of $\rho_1(h)$. 
Since $a_n\in L_0$, we deduce that $a'$ must be on the side of $L_0'$ associated with the positive corner of $L_0'$ for $\rho_1(h)$. Thus, if $a'=b'$, we deduce that $H^{-1}(b_n)$ must be in the lozenge or end of the line containing $L_0'$ and intersecting a $\cF_1^+$-leaf that shares the corner $c_{-1}'$. But by construction, $b_n$ is part of a $k$-web joining $b_n$ to a point in $L_1$. Therefore, $H^{-1}(b_n)$ is part of a $k$-web connected to $H^{-1}(L_1)=L_1'$, thus cannot converge to $a'$, which is what we needed to show.  Thus, each lozenge of the original line $\cL$ is contained in another line, i.e. it has adjacent lozenges on all four sides; moreover, that line admits sequences playing a role analogous role to the $(y_n^i)$: their images are moved to opposite sides of lozenges in the corresponding line of lozenges in $P_1$ by $H^{-1}$.  Iterating this argument shows that $\cL$ and $\cL'$ both lie in trees of scalloped regions comprised of $\rho_2(h)$ or $\rho_1(h)$ invariant lozenges, respectively.  
\end{proof} 

Our next lemma says that the setting of item \ref{item_case_discontinuity} of Lemma \ref{lem_case_of_sequences_on_opposite_sides} happens for one sequence in a tree of scalloped regions if and only if it happens for all sequences in that tree of scalloped regions. 
For convenience we introduce the following definition.  

\begin{definition} 
For a tree of scalloped regions $T \subset P_i$ with every corner fixed by $\rho_i(h)$, we let $S_T$ denote the set of pairs of convergent sequences $((x^1_n), (x^2_n))$ of points in $Q_i \cap T$ such that $(x^1_n)$ and $(x^2_n)$ are in adjacent $\rho_i(h)$-lozenges and converge to the same point on a leaf of a positive corner for $\rho_i(h)$.    
\end{definition} 

\begin{lemma} \label{lem:flip_always_or_never}
Let $T \subset P_1$ be a tree of scalloped regions with every corner fixed by $\rho_1(h)$.  Let $((z^1_n), (z^2_n)) \in S_T$.  
If $H(z^1_n)$ and $H(z^2_n)$ converge to the same point in $P_2$, then we have $\lim\limits_{n \to \infty} H(x^1_n) = \lim\limits_{n \to \infty} H(x^2_n)$ for every $((x^1_n), (x^2_n)) \in  S_T$.  
\end{lemma}
As usual, this result also holds with the roles of $P_1$ and $P_2$ reversed, using $H^{-1}$ in place of $H$.  

\begin{proof} 
We prove the contrapositive: let $y_1:= \lim\limits_{n \to \infty} H(x^1_n)$ and $y_2: = \lim\limits_{n \to \infty} H(x^2_n)$ and assume $y_1 \neq y_2$ where $((x^1_n), (x^2_n))$ is some element of $S_T$, we will show that for any other $((z^1_n), (z^2_n)) \in S_T$, the limits of $H(z^1_n)$ and $H(z^2_n)$ also disagree.  

Let $L_i$ denote the lozenge containing $x^i_n$ in $T$, and let $c_{1,2}$ denote the corner shared by $L_1$ and $L_2$, which is a positive fixed point of $\rho_1(h)$, by assumption.  
Since $y_1 \neq y_2$, the proof of Lemma \ref{lem_case_of_sequences_on_opposite_sides} shows that the corner shared by $H(L_1)$ and $H(L_2)$ is negative; and this propagates globally: whenever adjacent lozenges in $T$, share a positive corner, their images under $H$ share a negative corner.  
Thus, for any other sequence $((z^1_n), (z^2_n)) \in S_T$, as their images $H(z^1_n)$ and $H(z^2_n)$ are in distinct, adjacent lozenges and must converge to the side of a positive corner for $\rho_2(h)$, we deduce that their limits disagree.  
\end{proof} 

 Using Lemma \ref{lem:flip_always_or_never}, we can now define the {\em sign data} for trees of scalloped regions that is the last ingredient in our classification theorem.  

Given an Anosov-like action $\rho$ of $G$ on a bifoliated plane $P$, let $\tau(P)$ denote the set of trees of scalloped regions in $P$ and $\tau(\rho)$ the set of $\rho(G)$-orbits of elements in $\tau(P)$; i.e., $\tau(\rho) := \tau(P)/\rho(G)$.  
If $\rho_i$ are Anosov-like actions on $(P_i, \cF^+_i, \cF^-_i)$, $i = 1, 2$ and 
 $P(\rho_1) = P(\rho_2)$, either $\tau(P_1) = \tau(P_2) = \emptyset$ or, by Corollary \ref{cor:trees_are_same}, 
 our map $H$ induces a natural, $G$-equivariant bijection $B\colon \tau(P_1) \to \tau(P_2)$ associating a tree of scalloped regions $T_{\rho_1(h)} \subset P_1$ with all corners fixed by $\rho_1(h)$ to the tree $T_{\rho_2(h)} \subset P_2$ of scalloped regions with all corners fixed by $\rho_2(h)$.   This descends to a bijection $\overline{B}\colon \tau(\rho_1) \to \tau(\rho_2)$, since an orbit $\rho_i(G)T_{\rho_i(h)}$ consists of the trees of scalloped regions of the form $T_{\rho_i(ghg^{-1})}$ for $g \in G$.

\begin{definition}[Signs for scalloped trees]
Assume $P(\rho_1) = P(\rho_2)$ for actions $\rho_i(G)$ on $(P_i, \cF^+_i, \cF^-_i)$.
We say that a pair $T_1 \in \tau(P_1)$, and $B(T_1) \in \tau(\rho_2)$ {\em have the same sign} if for some (and hence for {\em all}, by Lemma \ref{lem:flip_always_or_never}) elements $((x^1_n), (x^2_n)) \in S_{T_1}$, the sequences $(H(x^1_n))$ and $(H(x^2_n))$ have the same limit point.  
 \end{definition}
 
Note that, if $T_1$ and $B(T_1)$ have the same sign, then $\rho_1(g)(T_1)$ and $B(\rho_1(g)T_1)$ have the same sign, allowing the following definition

\begin{definition}[Signs for orbits of scalloped trees] \label{def:same_sign_tree}
Assume $P(\rho_1) = P(\rho_2)$ for actions $\rho_i(G)$ on $(P_i, \cF^+_i, \cF^-_i)$.  We say $\rho_1$ and $\rho_2$ {\em have the same sign data for scalloped trees} if $[T]$ and $\overline{B}[T]$ have the same sign, for each $[T] \in \tau(\rho_1)$.    Note this is automatically satisfied if $\tau(\rho_1) = \emptyset$ (equivalently, if either plane has no trees of scalloped regions).  
\end{definition}

\begin{rem} \label{rem:def_sign_tree_flow}
For transitive Anosov flows $\varphi_1$ and $\varphi_2$ on a compact $3$-manifold $M$, with $\Phi_*P(\varphi_1) = P(\varphi_2)$ for some $\Phi \in \mathrm{Aut}(\pi_1(M))$, then $P(f\varphi_1 f^{-1}) = P(\varphi_2)$ for any $f$ inducing $\Phi$ on $\pi_1(M)$.  The orbit space of $\varphi_1$ and $f\varphi_1 f^{-1}$ are naturally homeomorphic via the map induced by a lift of $f$.   This induces a $\pi_1(M)$-equivariant  bijection between trees of scalloped regions in the orbit spaces of $\varphi_1$ and $\varphi_2$, and ``same sign" can be defined exactly as above.  
\end{rem}

Returning to our current setting, we can summarize the results of Lemmas \ref{lem_case_of_sequences_on_same_side} and \ref{lem_case_of_sequences_on_opposite_sides} in the following statement, the proof is immediate from the Lemmas.

\begin{proposition}[Unique limit points for intersections of non-corner and fixed leaves]\label{prop_unique_limit_for_NCcapfixed}
Suppose that $(P_1,\cF_1^+,\cF_1^-)$ and  $(P_1,\cF_1^+,\cF_1^-)$ have the same sign data for scalloped trees.  
Suppose $x_n \in Q_1$ and $x_n \to x \in \cF_1^{-}(w) \cap \cF_1^+(c)$, where $w$ is a non-corner fixed point of $\rho_1(g)$ and $c$ is a (possibly corner) fixed point of $\rho_1(h)$. Then $H(x_n)$ converges to a point in $P_2$.
\end{proposition}

\paragraph{\textbf{Second extension of H}}
From this point forward, we assume that one (and hence both) of the planes $P_i$ have the same sign data for scalloped trees.   Let $R_i\subset P_i$ be the subset consisting of points $x\in P_i$ such that $x \in \cF_i^{-}(w)\cap \cF_i^+(c)$, where $w$ is a non-corner fixed point and $c$ is a (corner or not) fixed point.
 We extend the domain of definition of $H$ to $R_1$ as follows.  
Given $x\in R_1$,  pick any sequence $(x_n)$ in  $Q_1$ converging to $x$ and define $H(x):=\lim_{n \to \infty} H(x_n)$. By Proposition \ref{prop_unique_limit_for_NCcapfixed}, this is well-defined, independent of the choice of sequence.

\begin{lemma}\label{lem_extension_to_R}
 The map $H$ is a bijection from $R_1$ to $R_2$.
\end{lemma}
\begin{rem}
 One could actually prove at this point that $H\colon R_1\to R_2$ is a homeomorphism, using the same arguments as we will use later to prove that a further extension of $H$ is a homeomorphism (see Lemma \ref{lem_continuous}). However, since we do not need to use continuity at this point, we will only give the (shorter) argument for bijectivity here.  
\end{rem}

\begin{proof}
We will show that $H(R_1)\subset R_2$.  Having shown this, one may similarly define $H^{-1}$ on $R_2$ by using convergent sequences in $Q_2 = H(Q_1)$, and repeat the proof to see that $H^{-1}(R_2) \subset R_1$ or equivalently $H(R_1) \supset R_2$.  
 
To show $H(R_1)\subset R_2$, let $x \in \cF_1^{-}(w)\cap \cF_1^+(c)$, where $w$ is a non-corner fixed point of $\rho_1(g)$ and $c$ is a point fixed by $\rho_1(h)$.  Up to replacing $h$ by a power, we assume that $\rho_1(h)$ fixes every half-leaf of $c$. By Lemma \ref{lem:accumulation_for_non_corner_leaves}, we know that $H(x)$ is on $\cF_2^-(H(w))$ as well as either inside the interior of a $\rho_2(h)$-lozenge or on a $\rho_2(h)$-invariant leaf.
 
 Since $x\in \cF^+(c)$, we can pick two sequences $(x^1_n),(x^2_n)$ converging to $x$ and on opposite sides of $\cF^+(c)$. In particular, $x^1_n$ and $x^2_n$ are \emph{not} in the same $\rho_1(h)$-lozenge. (Rather, they are either in distinct $\rho_1(h)$-lozenges, or one or both does not lie in any $\rho_1(h)$-lozenge.)
 By Lemma \ref{lem_preserves_lozenges} we deduce that $H(x^1_n)$ and $H(x^2_n)$ cannot be in the same lozenge either. Thus $H(x)$ cannot be inside the interior of a $\rho_2(h)$-lozenge, and hence must be on a $\rho_2(h)$-invariant leaf. That is, $H(x)\in R_2$.
\end{proof}

Now we extend Proposition \ref{prop_unique_limit_for_NCcapfixed} to all points on leaves of non-corner fixed points.

\begin{proposition}[Unique limit points for non-corner leaves]\label{prop_unique_limit_for_NC}
Suppose $x_n \in Q_1$ and $x_n \to x \in \cF_1^{-}(w)$, where $w$ is a non-corner fixed point of $\rho_1(g)$. Then $H(x_n)$ converges to a unique point in $P_2$.
\end{proposition}

Since $w$ is a non-corner, Lemma \ref{lem:leaves_of_non-corners}, says that any accumulation point of $H(x_n)$ lies inside $P_2$, not on $\Pbound_2$. 
So all we have to do in order to prove the Proposition is to show uniqueness of accumulation points.
This is the one and only place where we will need to use Axiom \ref{Anosov_like_totallyideal} -- specifically, we will use Lemma \ref{lem_distinct_leaves_distinct_saturations} which depended on this axiom. 

\begin{proof}[Proof of Proposition \ref{prop_unique_limit_for_NC}] 
Let $x_n \in Q_1$ and assume $x_n \to x \in \cF^-_1(w)$.  Let $(x^i_n)$, $i=1,2$, be two sub-sequences of $(x_n)$
 such that $H(x^i_n)\to y_i \in P_2$.  We need to show $y_1 = y_2$.  

By Proposition \ref{prop_unique_limit_for_NCcapfixed}, if $x\in R_1$, then we are already done. Similarly, if $y_1\in R_2$, then, by Lemma \ref{lem_extension_to_R}, $H^{-1}(y_1)\in R_1$ and we have $H^{-1}(y_1) = \lim_{n \to \infty}  H^{-1}(H(x^1_n)) = x$. So $x\in R_1$ which implies, by Proposition \ref{prop_unique_limit_for_NCcapfixed}, that $y_1=y_2$.  The same reasoning applies if $y_2 \in R_2$.  

By Lemma \ref{lem:accumulation_for_non_corner_leaves}, we have $y_1,y_2\in\cF^-_2(H(w))$.  Thus, if $y_1 \neq y_2$, they must lie on distinct leaves of $\cF^+_2$. 
Assume now for contradiction that $y_1 \neq y_2$. As noticed above, this implies that neither $y_1$ nor $y_2$ would be in $R_2$, so in particular, the leaves $\cF^+_2(y_i)$ are both non-singular.
By Lemma \ref{lem_distinct_leaves_distinct_saturations}, up to switching the names of $y_1$ and $y_2$, there exists $l^-$ such that $l^-\cap \cF_2^+(y_1) \neq \emptyset$ and $l^-\cap \cF_2^+(y_2) =\emptyset$.

Using density of non-corners (and the fact that $\cF^+_2(y_1)$ is non-singular), we can then choose a non-corner fixed point $b$ close to and totally linked with $\cF^+(y_1) \cap l^-$ as in Figure \ref{fig:finding_non_totally_linked}, so that $b$ is totally linked with $y_1$ but not with $y_2$.  Consequently, $b$ must be totally linked with $H(x^1_n)$, for all $n$ large enough.

 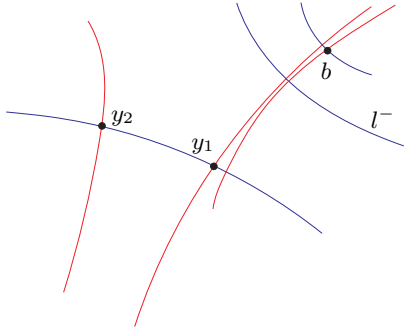
\begin{figure}[h]
   \labellist 
  \small\hair 2pt
     \pinlabel $y_2$ at 135 230 
    \pinlabel $y_1$ at 220 200 
    \pinlabel $b$ at 350 280
    \pinlabel $l^-$ at 410 230
 \endlabellist
     \centerline{ \mbox{
\includegraphics[width=5.5cm]{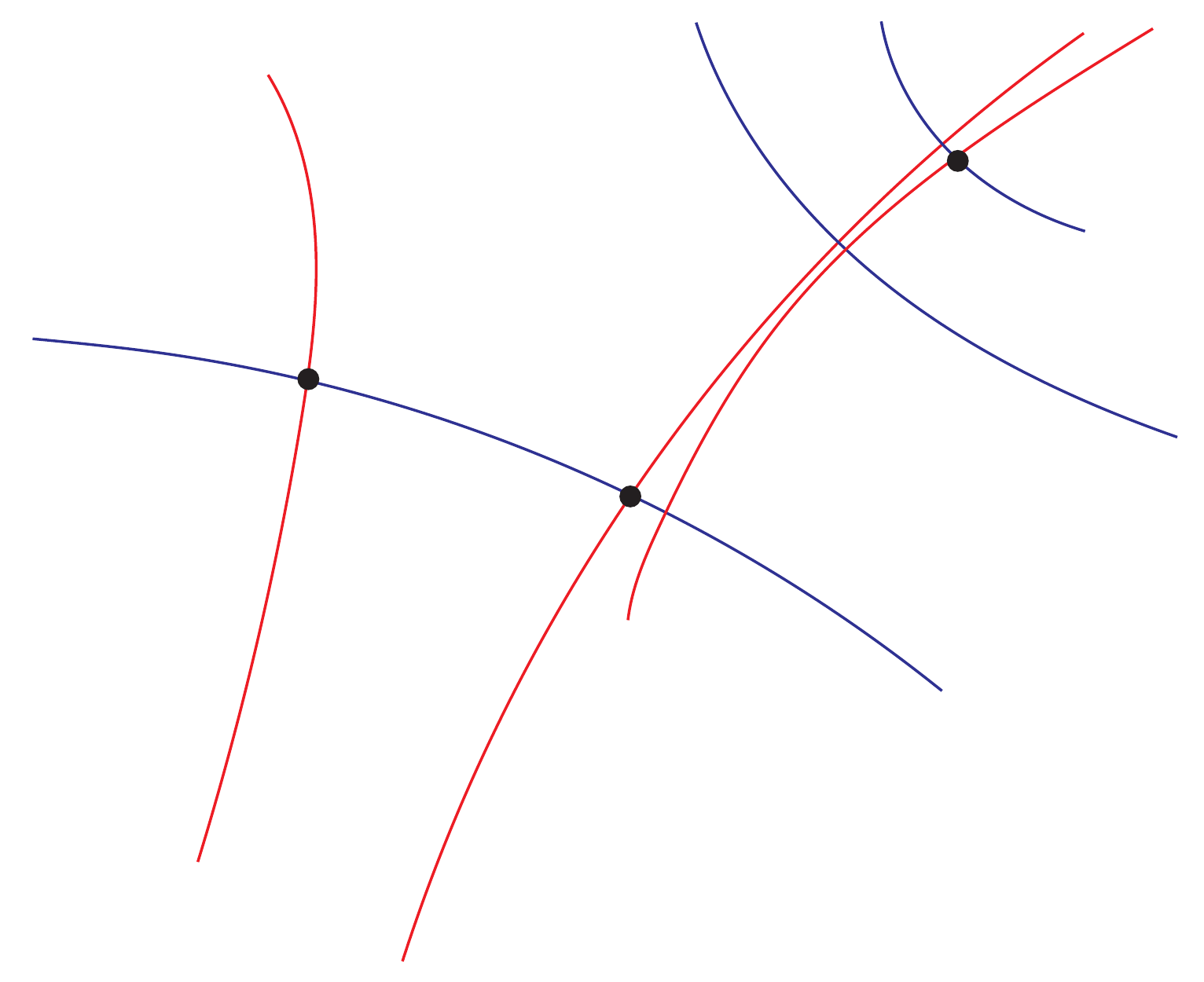}}}
\caption{A non-corner point $b$ totally linked with $y_1$ but not with $y_2$.}
 \label{fig:finding_non_totally_linked} 
\end{figure} 

First consider the case where $b$ is also totally linked with infinitely many $H(x^2_n)$.  In this case, Observation \ref{obs_sequence_of_totally_linked} implies that $b$ would either be totally linked with $y_2$ or that $\cF_2^-(b)$ intersects a leaf that is non-separated with $\cF_2^+(y_2)$ (since $b$ is a non-corner fixed point the ``perfect fit'' option of Observation \ref{obs_sequence_of_totally_linked} cannot happen).
Now, by construction, $b$ is not totally linked with $y_2$. Then we must have that $\cF_2^+(y_2)$ is non-separated with another leaf, which implies by Axiom \ref{Anosov_like_periodic_non-separated}, that $\cF^+_2(y_2)$ is fixed by some element $\rho_2(h)$. Since we also have $y_2\in \cF^-_2(H(w))$, we deduce that $y_2\in R_2$, and are done.

Now consider the remaining case where $b$ is not totally linked with infinitely many $H(x^2_n)$.  After dropping terms, we can assume that $b$ is totally linked with all $H(x^1_n)$ and none of the $H(x^2_n)$.
 Consider $b':= H^{-1}(b)$.  Since $H$ preserves linkedness (Corollary \ref{cor:define_non-corner_leaves}), $b'$ is not totally linked with any of the $x^2_n$. Thus $b'$ cannot be totally linked with $x$.  Since $b$ is totally linked with all the $x^1_n$, by Observation \ref{obs_sequence_of_totally_linked} as above, we conclude that $\cF_1^-(b')$ must intersect a leaf that is non-separated with $\cF_1^+(x)$. But this implies that $\cF_1^+(x)$ is fixed by some element $\rho_1(h)$, so $x\in R_1$.  This concludes the proof.   
\end{proof}

\paragraph{\textbf{Third extension of H}}
Let $S_i \subset P_i$ denote the union of all leaves of  
non-corner fixed points.  
We extend the domain of $H$ to $S_1$ by defining $H(x)$ to be the unique accumulation point of $H(x_n)$ where $x_n$ is any sequence of pairwise totally linked non-corner points approaching $x$.    This is well defined by Proposition \ref{prop_unique_limit_for_NC}.

\begin{lemma}\label{lem_continuous}
$H$ is a homeomorphism $S_1 \to S_2$, where $S_i$ is given the subset topology from $P_i$.  
\end{lemma}

\begin{proof}
Note that $H(S_1) \subset S_2$ by Lemma \ref{lem:accumulation_for_non_corner_leaves}, and by symmetry we have $H^{-1}(S_2) \subset S_1$ and thus $H(S_1) = S_2$.  It is also immediate from uniqueness of limit points that $H$ is a bijection $S_1 \to S_2$, so we need only prove that $H$ is continuous (and conclude $H^{-1}$ is also continuous by symmetry).  

Let $x\in S_1$ and let $(y_n) \subset S_1$ be a sequence of points that converges to $x$. We want to show that $H(y_n)$ converges to $H(x)$.
To do this, for each $n$, choose a sequence $(y^k_n)$ of non-corner fixed points converging to $y_n$ as $k\to +\infty$. By definition, $H(y^k_n) \to H(y_n)$.
Given any neighborhood $U$ of $H(x)$, take some smaller neighborhood $U'$ of $H(x)$ with closure contained in $U$. If $H(y_n) \notin U$, then there exists $k(n)$ such that $H(y^k_n) \notin U'$ for all $k > k(n)$, by definition of $H$.   If infinitely many such $y_{n}$ failed to be in $ U$, we could take a sequence $m_n > k(n)$ such that $y^{m_n}_n \to x$ but $H(y^{m_n}_n) \notin U'$, contradicting the definition of $H$ using limits of non-corner points. 
\end{proof}

\begin{proof}[End of proof of Theorem \ref{thm_main_general}]
First, we show $H$ extends uniquely to a continuous map $P_1 \to P_2$.  
Define a {\em basic $S$-polygonal domain} in $P_i$ to be a region bounded by segments of leaves in $S_i$, either a rectangle containing no singular points in its interior, or a $2n$-gon containing a single $n$-prong singularity.   Note that each point in $P_i$ is contained in some basic $S$-polygonal domain.

The interior of a basic $S$-polygonal domain $K$ can be characterized by the following ``convexity" property: a point $x$ lies in the interior of $K$ if and only if $x$ is not on $\partial K$ and both $\cF^+(x)$ and $\cF^-(x)$ intersect $\partial K$. 

Now fix a basic $S$-polygonal domain $K \subset P_1$.  By continuity of $H$ on $S_1$, the image of its boundary is a closed, polygonal path in $S_2$, which bounds a domain $K_2$ in $P_2$. By the above characterization of the interior of $S$-polygonal domains, we deduce that $H(K\cap S_1) = K_2\cap S_2$.  Similarly, $H^{-1}$ sends basic $S$-polygonal domains to basic $S$-polygonal domains.

Given any point $x\in P_1$, take a sequence of nested basic $S$-polygonal domains $K_n$ such that $\bigcap_n K_n. = x$.  By the above, $H(\partial K_n)$ is the boundary of a basic $S$ polygonal domain $\overline{H(K_n\cap S_1)}$ in $P_2$.  We claim that $\bigcap_n \overline{H(K_n\cap S_1)}$ is a single point.  For if we had distinct $y_1 \neq y_2$ in $\bigcap_n K'_n$, then we could find two {\em disjoint} sequences $C^1_n,C^2_n$ of nested $S$-polygonal domains inside $\overline{H(K_n\cap S_1)}$,  with $\bigcap_n C^i_n = y_i$.  Thus, $H^{-1}(C^1_n)$ and $H^{-1}(C^2_n)$ would be disjoint nested sequences of $S$-polygonal domains in $K_n$, which contradicts the fact that they both need to contain $x$.

Thus, $\bigcap_n \overline{H(K_n\cap S_1)}$ is a single point in $P_2$; call this point $H(x)$. 
Note that this is well defined, since if we are given any two sequences of nested polygonal domains $K^i_n$ such that $\bigcap_n K^i_n. = x$, for $i=1,2$, then for each $n$ there exists $m(n)$ such that $K^1_{m(n)} \subset K^2_n$, so $\bigcap_n \overline{H(K^1_n\cap S_1)} \subset \bigcap_n \overline{H(K^2_n\cap S_1)}$.   
Thus, $H$ is well defined on all of $P_1$, and by symmetry we also have $H^{-1}$ defined on the whole of $P_2$.  The fact that both $H$ and $H^{-1}$ are continuous now follows from a ``diagonal sequence" argument exactly as in the proof of Lemma \ref{lem_continuous}.

\vspace{.4cm}

To finish the proof of Theorem \ref{thm_main_general}, it remains to show that $H$ conjugates the actions of $\rho_1$ and $\rho_2$.
 By definition of $H$, any non corner point $x\in P_1$ that is fixed by $\rho_1(g)$ is sent to $H(x)$ the unique fixed point of $\rho_2(g)$.
 
Given a non-corner point $x$, fixed by $\rho_1(g)$ and an element $h\in G$, the point $\rho_1(h) x$ is the unique fixed point of $\rho_1(hgh^{-1})$. Hence, $H(\rho_1(h) x)$ is the unique fixed point of $\rho_2(hgh^{-1})$, so $H(\rho_1(h) x) = \rho_2(h) H(x)$. 
Since the set of non-corner fixed points is dense and $H$ is continuous, $H$ conjugates the actions of $\rho_1$ and $\rho_2$.
\end{proof}

\section{Action on ideal circle determines the action on the plane}\label{sec:corollary_for_action_ideal_circle} 
Each Anosov-like action on a bifoliated plane induces an action on the ideal circle.  We show that this action at infinity determines the original action, as follows.  
 \begin{theorem}\label{thm:ideal_circle} 
  Let $\rho_1,\rho_2$ be two Anosov-like actions of a group $G$ on bifoliated planes $(P_1,\cF_1^+, \cF_1^-)$ and $(P_2,\cF_2^+, \cF_2^-)$. Let $\Pbound_1$ and $\Pbound_2$ be the respective ideal circles.

If $h\colon \Pbound_1 \to \Pbound_2$ is a homeomorphism conjugating the induced actions of $\rho_1$ and $\rho_2$, then $h$ extends uniquely to a homeomorphism $H\colon P_1\cup \Pbound_1 \to P_2\cup \Pbound_2$ that conjugates $\rho_1$ and $\rho_2$.
 \end{theorem}

 \begin{proof} 
First, if $(P_1,\cF_1^+, \cF_1^-)$ is trivial, then there is a finite index subgroup $G'$ that acts on $\Pbound_1$ with $4$ global fixed points (the four points that are not the endpoints of any leaves). Thus, the action of $\rho_2(G')$ on $\Pbound_2$ must also have $4$ global fixed points, so $\Pbound_2$ is trivial by 
 Proposition \ref{prop_4_2_0_global_fixed_points}.  
Besides the  four global fixed points, all boundary points correspond to unique endpoints of leaves, and points of $P_i$ correspond uniquely to intersections of leaves.  Thus, one obtains directly the homeomorphism $H$ from $h$: given a point $x \in P_1$ with $\xi^\pm$ any chosen boundary points of $\cF(x)^\pm$, define $H(x)$ to be the unique point of intersection of the leaves with endpoints $h(\xi^-)$ and $h(\xi^+)$, one easily checks this is a homeomorphism.  

So we can now assume that neither plane is trivial. 
If $(P_1,\cF_1^+, \cF_1^-)$ is skewed, then there exists a finite index subgroup $G'$ acting on $\Pbound_1$ with exactly two global fixed points. So $G'$ also acts with two global fixed points on $\Pbound_2$, and, since $P_2$ is not trivial, Proposition  \ref{prop_4_2_0_global_fixed_points} again gives that $P_2$ is skewed. Then, once again, the homeomorphism $H$ can be easily constructed directly via intersections of leaves, as each component of the complement of the global fixed set in the boundary can be identified with the leaf space of $\cF^+$ or $\cF^-$. (See e.g. \cite{BMB} for a more detailed account of this strategy).

Finally assume that the bifoliated planes are neither skewed nor trivial.  Here we will show that $\fix(\rho_1) = \fix(\rho_2)$.  
Let $g \in \fix(\rho_1)$.  Let $k$ be such that $\rho_1(g^k)$ fixes all half-leaves through (any of) its fixed point(s) in $P_1$.  
Then by Proposition \ref{prop:boundary_action_general}, $\rho_1(g^k)$ either has exactly $2n$ fixed points in $\Pbound_1$, alternatively attractors and repellers (these are the endpoints of $\cF^\pm(x)$, where $n$ is the number of prongs through $x$), or $x$ is a corner and $\rho_1(g^k)$ has more than four fixed points on $\Pbound_1$. 
Since the boundary actions are conjugate, the number of fixed points of $\rho_2(g^k)$ and the dynamics at these fixed points is the same as that for $\rho_1(g^k)$.  

We wish to show that $\rho_2(g^k) \in P(\rho_2)$, i.e. to eliminate the case where $\rho_2(g^k)$ acts freely.  If it did act freely, then we would be in case (3) of Proposition \ref{prop:boundary_action_general} and so $\rho_2(g^k)$ would have either at most two fixed points (disagreeing with the action of $\rho_1(g^k)$, or exactly four fixed points, the ideal corners of a scalloped region. However, the dynamics at the corners of an invariant 
scalloped region are not alternating attractor/repeller: If one of the four corners of the scalloped region is an attractor for the action of $\rho_2(g^k)$, then the adjacent corners also are limits of points under positive iteration by $\rho_2(g^k)$, coming from endpoints of leaves through the scalloped region. Thus, we conclude that $\rho_2(g^k)$ does not act freely. We have left to show that this implies that $\rho_2(g)$ does not act freely either, i.e $g \in \fix(\rho_2)$\footnote{For Anosov flows, this was proven in \cite{Fenley_homotopy}}. Indeed, if $\rho_2(g)$ acts freely on $P$, then by Proposition \ref{prop:boundary_action_general} it has either at most two fixed points on $\Pbound_2$ or it preserves a scalloped region. The latter cannot happen, since $\rho_2(g^k)$ would also act freely. If $\rho_2(g)$ has one or two fixed points on $\Pbound_2$ and acts freely in $P$, then $\rho_2(g)$ or $\rho_2(g^2)$ acts as a translation on each connected component of the complement of the fixed points in $\Pbound_2$. This again contradicts the fact that $\rho_2(g^k)$ has at least four fixed points in $\Pbound_2$. Thus  $g \in \fix(\rho_2)$. 
By symmetry, we obtain $\fix(\rho_1) = \fix(\rho_2)$.

If either $P_1,\cF^{\pm}_1$ or $P_2,\cF^{\pm}_2$ has no tree of scalloped regions, we are already done by Theorem \ref{thm_main_general}.    In the case where one (hence both) of the planes does have a tree of scalloped regions, we need to show that having conjugate boundary actions implies that their trees of scalloped regions have the same signs.  Note that, since $P(\rho_1) = P(\rho_2)$, our map $H$ is defined on $Q_1$ and the definition of {\em same sign} makes sense.  

Let $T$ be a tree of scalloped regions in $P_1$ with each corner fixed by $\rho_1(h)$, and let $((x^1_n), (x^2_n)) \in S_{T}$ with $(x^i_n)$ both converging to $x$ on a $\rho_1(h)$-invariant leaf $l$.  (See Definition \ref{def:same_sign_tree} and preceding remarks for this notation.)    Say, without loss of generality that $l \in \cF_1^+$, the case of $\cF_1^-$ is completely analogous.  

All points $x^i_n$ are regular, non-corner fixed points; let $g_{i,n} \in G$ denote an element such that $x^i_n$ is a positive fixed point of $\rho_1(g_{i,n})$.   Then the endpoints of $\cF^+(x^i_n)$ are the unique attracting fixed points of the action of $\rho_1(g_{i,n})$ on $\partial P_1$, and as $n \to \infty$ these converge to the endpoints of $l$.   Our map $H$ is defined on $x^i_n$, and  $H(x^i_n)$ is the unique positive fixed point of $\rho_2(g_{i,n})$.  Thus, the endpoints $\xi^i_n, \eta^i_n$ of $\cF^+(H(x^i_n))$ are the unique attracting fixed points of $\rho_2(g_{i,n})$ on $\Pbound_2$.  Since $\rho_1$ and $\rho_2$ are conjugate, as $n \to \infty$ both $\xi^1_n$ and $\xi^2_n$ converge to the same point (the image of an endpoint of $l$ under the conjugating map), and similarly $\eta^1_n$ and $\eta^2_n$ converge to the same point.  Thus, $H(x^1_n)$ and $H(x^2_n)$ cannot limit onto distinct $\cF^+$-leaves forming opposite sides of a pair of lozenges, so $T$ and the corresponding tree $B(T)$ in $P_2$ have the same sign.  
Since our choice of tree of scalloped region was arbitrary, we conclude that $\rho_1$ and $\rho_2$ have the same sign data for their trees of scalloped regions so the actions on the planes are conjugate by Theorem \ref{thm_main_general}.

Finally, the uniqueness of the extension $H$ conjugating the two actions is immediate in the trivial and skew case; in the general case it follows directly from density of non-corner fixed points, whose images are uniquely defined under any conjugacy.  
\end{proof}

\section{Building counterexamples}\label{sec:transitive}

There are two conditions in Theorem \ref{thm_main_Anosov}: sign data for trees of scalloped regions and transitivity.  Here we explain how to build counterexamples when the transitivity condition is not met, leaving the other case to the forthcoming paper \cite{BFM_in_preparation}.

It is clear that our proof of Theorem \ref{thm_main_Anosov} requires the flows to be transitive.
In a personal communication, Bin Yu suggested a construction of some counterexamples for the non-transitive case. We give an explicit version of this construction. It is based on the construction via DA and gluings developed in \cite{BBY} that we will not recall explicitely here.

\begin{theorem}[Bin Yu, personal communication]\label{thm_BinYu}
 There exists a manifold $M$ carrying two non-orbit equivalent, non-transitive Anosov flows, $\varphi_1, \varphi_2$, such that $\cP(\varphi_1) = \cP(\varphi_2)$.
\end{theorem}

\begin{proof} 
Before going into details we give the basic strategy.  
We first construct a manifold $N'$ with Anosov flow $\varphi_{N'}$.  The key properties of $N'$ are that $\varphi_{N'}$ has no direction-reversing self-orbit equivalence (this is ensured by giving $\varphi_{N'}$ an attracting set but no repelling set), but nevertheless it admits a diffeomorphism $H\colon N' \to N'$ such that $H_*(\cP(\varphi_{N'})) = \cP(\varphi_{N'})$.    We then perform an attracting and repelling DA blow-up on $N'$, and glue on new pieces in two different ways, to produce a new manifold $M$ with two different Anosov flows $\varphi_1$ and $\varphi_2$ that each extend $\varphi_{N'}$ in such a way that we can extend $H$ to a map satisfying $H_*(\cP(\varphi_{1})) = \cP(\varphi_{2})$.  However, our construction will ensure that any purported orbit equivalence of $\varphi_1$ with $\varphi_2$ induces a direction-reversing orbit equivalence of $\varphi_{N'}$, which is impossible.

Now we give the construction. To build $N'$ we start with a genus two surface $\Sigma$ equipped with a hyperbolic metric that admits a reflection $h$ on a geodesic $\beta_0$.  Let $N = T^1\Sigma$ denote its unit tangent bundle, $\psi$ the associated geodesic flow on $N$, and let $Th \colon N \to N$ denote the map induced by $h$.
Then $\beta_0$ is a closed orbit satisfying $Th(\beta_0) = \beta_0$. 
Let $Q$ be another manifold equipped with another Anosov flow $\theta$, perform an attracting DA blow-up on $\beta_0$ and a repelling DA blow-up on a periodic orbit of $Q$ (see, e.g.~\cite{BBY} for the construction).  Removing a small neighborhood of the orbits, chosen so that its boundary is transverse to the flow, and equivariantly with respect to $Th$ on $N$, we obtain {\em hyperbolic plugs} in the sense of \cite{BBY} on two manifolds with torus boundaries $\hat N$ and $\hat Q$. Then, we can glue them to obtain a new Anosov flow $\varphi_N$ on a manifold $N'$.  

Note that $Th$ gives a finite order diffeomorphism of the torus on which $\hat N$ and $\hat Q$ are glued, so by performing an isotopy to the identity in a small neighborhood of the boundary on $\hat N$, we may extend $Th$ to a diffeomorphism $H$ of $N'$ that preserves each periodic orbit of $\theta$ on $\hat Q$.  Note also that $N'$ has the desired property that no self-orbit equivalence of $\phi_{N'}$ can flip the direction of the flow.  

Now we glue on two more building blocks.  
Let $P_1$ and $P_2$ be two non-diffeomorphic $3$-manifolds supporting two Anosov flows $\psi_1$ and $\psi_2$. (In fact, one could even take $P_1 = P_2$ provided the flows are inequivalent with different spectra of periodic orbits.) 
Let $\alpha_i$ be a periodic orbit of $\psi_i$, chosen so that the two noncompact manifolds $P_i - \alpha_i$ are non-diffeomorphic. 
Let $\beta_1$ be a periodic orbit of $Th$ that is not preserved by $Th$, it corresponds also to a periodic orbit of $\varphi_{N'}$.  Let $\beta_2 = h(\beta_1)$. 
Perform an attracting DA blow-up on $\alpha_1$ and $\beta_1$ and a repelling DA blow-up on $\alpha_2$ and $\beta_2$, and remove small tubular neighborhoods transverse to the flow as before, equivariantly with respect to $H$ on $N'$ so that $h$ induces a diffeomorphism of the blown-up manifold sending the torus boundary component $T_1$ coming from $\beta_1$ to the torus boundary component $T_2$ from $\beta_2$.  Call the resulting manifolds $\hat N'$, $\hat P_1$ and $\hat P_2$, these are again {\em hyperbolic plugs}. We further choose the manifolds $P_i$, so that $\hat N'$, $\hat P_1$ and $\hat P_2$ are not homeomorphic.
Using the techniques of \cite{BBY} we may now glue these plugs in two different ways to produce Anosov flows on the glued manifolds.    
\begin{itemize}
\item Way 1:  Glue $(\hat P_1, \hat \psi_1)$ onto $T_1$, and $(\hat P_2, \hat \psi_2)$ to $T_2$.  Call the resulting manifold $M_1$ with Anosov flow $\varphi_1$. 
\item Way 2:  Glue $(\hat P_1, \hat \psi_1^{-1})$ onto $H(T_1) = T_2$, and  $(\hat P_2,\hat \psi_2^{-1})$ onto $H(T_2) = T_1$.  Call the resulting manifold $M_2$ with Anosov flow $\varphi_2$.  
\end{itemize} 
Thanks to \cite{BBY}, we can choose the gluings in the same isotopy classes, hence the map $H$ can again be extended to realize a diffeomorphism between $M_1$ and $M_2$, which we (abusing notation slightly) also denote by $H$.   Note that each periodic orbit of $\varphi_1$ or $\varphi_2$ is a periodic orbit of one of the flows $\varphi_N$, $\psi_1$ or $\psi_2$, and thus $H_\ast\cP(\varphi_1) = \cP(\varphi_2)$. 

It remains to show that $\varphi_1$ and $\varphi_2$ are not orbit equivalent.  Suppose for contradiction that $f\colon M_1 \to M_2$ was an orbit equivalence. Since we chose the pieces $\hat P_1$ and $\hat P_2$ non-homeomorphic to $\hat N'$, $f$ needs to preserve each pieces. Then, by performing an isotopy along orbits, $f$ can be chosen to preserve the transverse tori.  But $f(\hat P_1)$ cannot be equal to $\hat P_2$, since $P_i \smallsetminus \alpha_i$ were chosen to be non-diffeomorphic. 
Thus, $f$ agrees with $H$ on $T_i$.  However, since the $T_1$ is attracting and $T_2$ is repelling, $f$ must induce a {\em direction-reversing} orbit equivalence of $\varphi_{N'}$, which we already showed to be impossible.  
\end{proof}

\section{Further questions} \label{sec:further_q}
We conclude this work by advertising three directions for future study.

\subsection{Dynamical characterizations of pseudo-Anosov actions.} 

The {\em convergence group theorem} of Tukia, Gabai, and Casson-Jungreis, gives a dynamical characterization of subgroups of $\mathrm{Homeo}(S^1)$ that are conjugate to Fuchsian groups, i.e. the boundary actions of fundamental groups of hyperbolic orbifolds and surfaces.   
Thurston's definition of {\em extended convergence group} (see \cite[Section 7]{Thurston:3MFC}) gives a similar characterization applicable to skew-Anosov flows on compact 3-manifolds.  

In the spirit of these results, we ask whether there is a purely dynamical condition that ensures that a group acting on a bifoliated plane is isomorphic to the fundamental group of a compact 3-manifold, and the action conjugate to that on the orbit space from a pseudo-Anosov flow. 
The examples in Remark \ref{rem_affine_counterexample} and \ref{rem:skew_counterexample} 
show that our Anosov-like conditions are not sufficient for trivial and skew planes. A good starting point to the general program would be to answer the following: 

\begin{question} Suppose $G$ is a group with a faithful, Anosov-like action on a bifoliated plane that is neither skewed nor trivial. Is $G$ necessarily a three-manifold group, and the action conjugate to that on the orbit space of a pseudo-Anosov flow?   If not, what additional dynamical hypotheses characterize such actions? 
\end{question}

\subsection{Towards a minimal set of axioms} \label{sec:digression_on_axioms} 

Our choice of the Axioms \ref{Anosov_like_A1} - \ref{Anosov_like_totallyideal} as the definition of ``Anosov-like action" was chosen for their mix of generality and convenience.   However, it is natural to ask whether a shorter or simpler list of axioms would suffice.  
While it would certainly be interesting to study the class of groups that admit actions satisfying only the hyperbolicity axiom \ref{Anosov_like_A1}, one is unlikely to recover such a rich structure theory for such actions as developed in Section \ref{sec:bifoliated}.  

Thus, it is natural to ask for more. Motivated by the example of transitive Anosov flows, one could then ask for topological transitivity and cocompactness of the action.  
A cocompact hyperbolic action will automatically verify Axiom \ref{Anosov_like_prongs_are_fixed}. However, while cocompactness is essential in the proof that Axiom \ref{Anosov_like_periodic_non-separated} holds for flows by Fenley (\cite[Theorem D]{Fenley_structure_branching}), it is not clear to us that it will be sufficient to prove Axiom \ref{Anosov_like_prongs_are_fixed} holds for an action with only hyperbolicity properties.
It is also not clear to us if a cocompact, topologically transitive hyperbolic action will automatically satisfy \ref{Anosov_like_dense_fixed_points}.

An alternative approach would be to strengthen the hyperbolicity condition as follows:  

\begin{definition} \label{def:extended_hyperbolic}
 	Let $G$ be a group acting on a bifoliated plane $(P,\cF^+, \cF^-)$, preserving each foliation.
	The action is said to be \emph{everywhere hyperbolic} if it satisfies the following properties:
	\begin{enumerate}
	 \item Topological hyperbolicity, as in \ref{Anosov_like_A1}.
	 \item For any sequence of pairwise distinct elements $g_n$ in $G$ and any point $x\in X$, if $g_n(x)$ converges to a point $y$, then for any foliated box neighborhood $V$ of $y$ and any compact sets $K^{\pm}$ inside $\cF^{\pm}(x)$, there exists $N$ such that for all $n>N$, we have either $g_n(K^+)\subset V$ and $g_n(K^-)\not\subset V$ or $g_n(K^-)\subset V$ and $g_n(K^+)\not\subset V$.
	 \item For any $x\in X$, there exists two sequences $g_n^{\pm}$ in $G$ such that $g_n^{\pm}(x)$ converges to points $y^{\pm}$ and, for any neighborhoods $V^{\pm}$ of $y^{\pm}$, and any compact sets $K^{\pm}$ inside $\cF^{\pm}(x)$, we have, for $n$ large enough, $g_n^{+}(K^+)\subset V^+$ and $g_n^{-}(K^-)\subset V^-$.
	\end{enumerate}
\end{definition}

 Condition (3) above mimics the fact that for a pseudo-Anosov flow on a compact manifold, hyperbolicity happens everywhere and each orbit has both its $\omega$- and $\alpha$-limit sets non-empty. This is the additional feature used by Fenley to prove that non-separated leaves are periodic in \cite[Theorem D]{Fenley_structure_branching}, as well as the non-existence of totally ideal quadrilaterals in \cite[Proposition 4.4]{Fen_qgpA16}.

\begin{rem}
While these convergence-type conditions (2) and (3) in Definition \ref{def:extended_hyperbolic} seem technical, they effectively captures the dynamics of Anosov-like actions.  
One should be able to show that topological transitivity, together with the everywhere hyperbolic condition imply the other axioms of Anosov-like flows.  Transitivity together with condition (2) can give an analogue of the Anosov closing lemma, and 
condition (3) should make it possible to follow the strategy of proof in 
\cite[Theorem D]{Fenley_structure_branching} and \cite[Proposition 4.4]{Fen_qgpA16} to deduce Axioms \ref{Anosov_like_periodic_non-separated} and \ref{Anosov_like_totallyideal}.
\end{rem}

\subsection{Spectral rigidity for nontransitive examples} 

In the counterexamples for nontransitive flows constructed in Section \ref{sec:transitive}, each basic set of $\varphi_1$ corresponds to a basic set of $\varphi_2$, and the restriction of the flows to each of these corresponding basic sets are orbit equivalent. It would be interesting to know whether this is necessary: 
 \begin{question}
  Let $\varphi_1$, $\varphi_2$ be two Anosov flows on $M$ such that $\cP(\varphi_1) = \cP(\varphi_2)$.
  Let $\Delta_1, \dots, \Delta_{n_1}$ and $\Lambda_1, \dots, \Lambda_{n_2}$ be the basic sets of $\varphi_1$ and $\varphi_2$ respectively.
  Is it necessarily true that $n_1 = n_2$? Further (assuming no trees of scalloped regions, for simplicity), up to reordering, does there exist, for each $i$, a homeomorphism $h_i \colon \Delta_i \to \Lambda_i$ that realizes an orbit equivalence between the restrictions of the flows to those basic sets?
 \end{question}

\bibliographystyle{amsalpha}
\bibliography{spectral_bib}

\providecommand{\bysame}{\leavevmode\hbox to3em{\hrulefill}\thinspace}
\providecommand{\MR}{\relax\ifhmode\unskip\space\fi MR }
\providecommand{\MRhref}[2]{%
  \href{http://www.ams.org/mathscinet-getitem?mr=#1}{#2}
}
\providecommand{\href}[2]{#2}
\begin{thebibliography}{AFW15}

\bibitem[AFW15]{AFW}
M.~Aschenbrenner, S.~Friedl, and H.~Wilton, \emph{3-manifold groups}, EMS
  Series of Lectures in Mathematics, European Mathematical Society (EMS),
  Z\"{u}rich, 2015.

\bibitem[Bar95a]{Bar_caracterisation}
T.~Barbot, \emph{Caract\'erisation des flots d'{A}nosov en dimension 3 par
  leurs feuilletages faibles}, Ergodic Theory Dynam. Systems \textbf{15}
  (1995), no.~2, 247--270.

\bibitem[Bar95b]{Barbot_MPOT}
\bysame, \emph{Mise en position optimale de tores par rapport \`a un flot
  d'{A}nosov}, Comment. Math. Helv. \textbf{70} (1995), no.~1, 113--160.

\bibitem[Bar96]{Barbot96}
\bysame, \emph{Flots d'{A}nosov sur les vari\'{e}t\'{e}s graph\'{e}es au sens
  de {W}aldhausen}, Ann. Inst. Fourier (Grenoble) \textbf{46} (1996), no.~5,
  1451--1517.

\bibitem[Bar98]{Barb_ActionGroupe}
\bysame, \emph{Actions de groupes sur les 1-vari\'{e}t\'{e}s non
  s\'{e}par\'{e}es et feuilletages de codimension un}, Ann. Fac. Sci. Toulouse
  Math. (6) \textbf{7} (1998), no.~4, 559--597.

\bibitem[BBY17]{BBY}
F.~B\'{e}guin, C.~Bonatti, and B.~Yu, \emph{Building {A}nosov flows on
  3-manifolds}, Geom. Topol. \textbf{21} (2017), no.~3, 1837--1930.

\bibitem[BF13]{BarbFen_pA_toroidal}
T.~Barbot and S.~Fenley, \emph{Pseudo-{A}nosov flows in toroidal manifolds},
  Geom. Topol. \textbf{17} (2013), no.~4, 1877--1954.

\bibitem[BF15]{BF_totally_per}
\bysame, \emph{Classification and rigidity of totally periodic pseudo-{A}nosov
  flows in graph manifolds}, Ergodic Theory Dynam. Systems \textbf{35} (2015),
  no.~6, 1681--1722.

\bibitem[BFM23]{BFM_in_preparation}
T.~Barthelm\'e, S.~Fenley, and K.~Mann, \emph{On totally periodic anosov
  flows}, In preparation (2023).

\bibitem[BM20]{BM}
J.~Bowden and K.~Mann, \emph{${C}^0$ stability of boundary actions and
  inequivalent {A}nosov flows}, 2020.

\bibitem[BM23]{BMB}
T.~Barthelm\'e and K.~Mann, \emph{Orbit equivalences of $\mathbb{R}$-covered
  {A}nosov flows and applications}, Geom. Topol. to appear (2023).

\bibitem[Bru93]{Brunella}
M.~Brunella, \emph{Separating the basic sets of a nontransitive {A}nosov flow},
  Bull. London Math. Soc. \textbf{25} (1993), no.~5, 487--490.

\bibitem[Cal07]{Calegari_book}
D.~Calegari, \emph{Foliations and the geometry of 3-manifolds}, Oxford
  Mathematical Monographs, Oxford University Press, Oxford, 2007.

\bibitem[CP]{CP}
A.~Clay and T.~Pinsky, \emph{Graph manifolds that admit arbitrarily many anosov
  flows}.

\bibitem[Fen94]{Fen_Anosov_flow_3_manifolds}
S.~Fenley, \emph{Anosov flows in {$3$}-manifolds}, Ann. of Math. (2)
  \textbf{139} (1994), no.~1, 79--115.

\bibitem[Fen95a]{Fenley_one_sided}
\bysame, \emph{One sided branching in {A}nosov foliations}, Comment. Math.
  Helv. \textbf{70} (1995), no.~2, 248--266.

\bibitem[Fen95b]{Fenley_QGAF}
\bysame, \emph{Quasigeodesic {A}nosov flows and homotopic properties of flow
  lines}, J. Differential Geom. \textbf{41} (1995), no.~2, 479--514.

\bibitem[Fen97]{Fenley_homotopy}
\bysame, \emph{Homotopic indivisibility of closed orbits of {$3$}-dimensional
  {A}nosov flows}, Math. Z. \textbf{225} (1997), no.~2, 289--294.

\bibitem[Fen98]{Fenley_structure_branching}
\bysame, \emph{The structure of branching in {A}nosov flows of
  {$3$}-manifolds}, Comment. Math. Helv. \textbf{73} (1998), no.~2, 259--297.

\bibitem[Fen99]{Fen_foliation_good_geo}
\bysame, \emph{Foliations with good geometry}, J. Amer. Math. Soc. \textbf{12}
  (1999), no.~3, 619--676.

\bibitem[Fen12]{Fen_ideal_boundaries}
\bysame, \emph{Ideal boundaries of pseudo-{A}nosov flows and uniform
  convergence groups with connections and applications to large scale
  geometry}, Geom. Topol. \textbf{16} (2012), no.~1, 1--110.

\bibitem[Fen16]{Fen_qgpA16}
\bysame, \emph{Quasigeodesic pseudo-{A}nosov flows in hyperbolic 3-manifolds
  and connections with large scale geometry}, Adv. Math. \textbf{303} (2016),
  192--278.

\bibitem[FH19]{FH_book}
T.~Fisher and B.~Hasselblatt, \emph{Hyperbolic flows}, Zurich Lectures in
  Advanced Mathematics, 2019.

\bibitem[FM01]{FenMosher}
S.~Fenley and L.~Mosher, \emph{Quasigeodesic flows in hyperbolic 3-manifolds},
  Topology \textbf{40} (2001), no.~3, 503--537.

\bibitem[FP19]{FenPot}
S.~Fenley and R.~Potrie, \emph{Ergodicity of partially hyperbolic
  diffeomorphisms in hyperbolic 3-manifolds}, 2019.

\bibitem[Fra13]{Fra_Mobiuslike}
S.~Frankel, \emph{Quasigeodesic flows and {M}\"{o}bius-like groups}, J.
  Differential Geom. \textbf{93} (2013), no.~3, 401--429.

\bibitem[Ghy84]{Ghys84}
\'{E}. Ghys, \emph{Flots d'{A}nosov sur les {$3$}-vari\'{e}t\'{e}s fibr\'{e}es
  en cercles}, Ergodic Theory Dynam. Systems \textbf{4} (1984), no.~1, 67--80.

\bibitem[GMT03]{GMT}
D.~Gabai, R.~Meyerhoff, and N.~Thurston, \emph{Homotopy hyperbolic 3-manifolds
  are hyperbolic}, Ann. of Math. (2) \textbf{157} (2003), no.~2, 335--431.

\bibitem[Mos92]{Mosher_homologynormI}
L.~Mosher, \emph{Dynamical systems and the homology norm of a {$3$}-manifold.
  {I}. {E}fficient intersection of surfaces and flows}, Duke Math. J.
  \textbf{65} (1992), no.~3, 449--500.

\bibitem[Nav11]{Navas_book}
A.~Navas, \emph{Groups of circle diffeomorphisms}, spanish ed., Chicago
  Lectures in Mathematics, University of Chicago Press, Chicago, IL, 2011.

\bibitem[Pla81]{Plante81}
J.~F. Plante, \emph{Anosov flows, transversely affine foliations, and a
  conjecture of {V}erjovsky}, J. London Math. Soc. (2) \textbf{23} (1981),
  no.~2, 359--362.

\bibitem[Sma67]{Smale}
S.~Smale, \emph{Differentiable dynamical systems}, Bull. Amer. Math. Soc.
  \textbf{73} (1967), 747--817.

\bibitem[Thu]{Thurston:3MFC}
W.~Thurston, \emph{3-manifolds, foliations and circles {I}}.

\bibitem[Wal68]{Waldhausen}
F.~Waldhausen, \emph{On irreducible {$3$}-manifolds which are sufficiently
  large}, Ann. of Math. (2) \textbf{87} (1968), 56--88.

\end{thebibliography}
 
\end{document}